\documentclass[a4paper,10pt,twoside,english]{scrartcl}
\usepackage{marginnote}
\usepackage[utf8x]{inputenc}
\usepackage{amsmath,amssymb,amsthm}
\usepackage{mathrsfs}  
\usepackage{dsfont} % used for \id 
\usepackage[T1]{fontenc}
\usepackage{color}
\usepackage{lmodern}
\usepackage{graphicx}
\usepackage{caption}
\usepackage{subcaption}
\usepackage{enumitem}
\usepackage{setspace}
\usepackage{xspace} 
\usepackage{blkarray}
\usepackage[subnum]{cases}
\usepackage{pgfplots}
  \pgfplotsset{compat=newest}
  \usetikzlibrary{plotmarks}
  \usetikzlibrary{arrows.meta}
  \usepgfplotslibrary{patchplots}
  \usepackage{grffile}

\usepackage[unicode=true]{hyperref}
\usepackage[all]{hypcap}

\usepackage{ifdraft}
\usepackage{tikz}
\usepackage{pgfplots}
\usepackage{etoolbox,listings,xstring,wasysym}

\usetikzlibrary{backgrounds,calc,graphs,decorations.pathmorphing,decorations.pathreplacing,decorations.shapes}
\usetikzlibrary{arrows.meta, graphs,graphs.standard,decorations.markings, quotes, automata,patterns}
\usetikzlibrary{positioning}
\usetikzlibrary{fadings}

\tikzfading[name=fade r,left color=white!100, right color=black!100]
\tikzfading[name=fade l,left color=black!100, right color=white!100, ]

\makeatletter

\tikzset{
  fl/.style = {path fading=fade l},
  fr/.style = {path fading=fade r},
  wh/.style = {draw=none,fill=none},
  Gc/.style = {draw=none,circle split, inner sep=0pt,minimum size=8pt,rotate=90,path picture={\draw[pattern=#1] (0,0.07) circle (1.5pt); }},
  Gd/.style = {draw=none,circle split, inner sep=0pt,minimum size=8pt,rotate=270,path picture={\draw[pattern=#1] (0,0.07) circle (1.5pt); }},
  Gf/.style={draw=none,circle split, inner sep=1pt,minimum size=8pt,rotate=90},
  G/.style={circle,draw,minimum size=8pt,inner sep=1pt,font=\tiny},
  xx/.style={circle,fill,draw,inner sep=0pt,minimum size=3pt},
  ab/.style={circle,fill,draw=none,inner sep=0pt,minimum size=0pt,as=},
  Gend/.style={inner sep=0pt,minimum size=3pt},
  r/.style={draw=red},
  o/.style={draw=black,fill=white},
  lb/.style args={#1}{label=below:#1},
  lt/.style args={#1}{label=above:#1},
  lr/.style args={#1}{label=right:#1},
  ll/.style args={#1}{label=left:#1},
  g/.style={postaction={decorate,decoration={
        markings,
        mark=at position .7 with {\arrow[#1]{Stealth[sep=-3pt]}}
      }}},
  s/.style={densely dashed,postaction={decorate,decoration={
        markings,
        mark=at position .7 with {\arrow[#1]{Stealth[sep=-3pt]}}
      }}},
  R/.style ={draw=lightgray, thick,densely dotted,edge node={node[above=-7pt] {\color{gray} $R$ }},anchor=south,pos=0.5,postaction={decoration={
        markings,
        mark=at position .7 with {\arrow[#1]{Stealth[sep=-3pt]}}
      },decorate}},
  K/.style ={draw=lightgray, thick,densely dotted,edge node={node[above=-7pt] {\color{gray} $K$ }},anchor=south,pos=0.5,postaction={decoration={
        markings,
        mark=at position .7 with {\arrow[#1]{Stealth[sep=-3pt]}}
      },decorate}},
we/.style args ={#1}{draw=lightgray, thick,densely dotted,edge node={node[above=-7pt] {\color{gray} #1 }},anchor=south,pos=0.5},
IEg/.style ={draw=lightgray, thick,densely dotted},
  S/.style ={draw=lightgray, thick,densely dotted,edge node={node[above=-7pt] {\color{gray}$S$}},anchor=south,pos=0.5},  
  J/.style ={edge node={node[above=-6pt] {$J$}},anchor=south,pos=0.5},  
  T/.style ={draw=lightgray, thick,densely dotted,edge node={node[above=-7pt] {\color{gray}$T$}},anchor=south,pos=0.5,postaction={decoration={
        markings,
        mark=at position .7 with {\arrow[#1]{Stealth[sep=-3pt]}}
      },decorate}},
  rpd/.style ={draw=lightgray,thick,densely dotted,edge node={node[above=-7pt] {\color{red}$\partial$}},anchor=south,pos=0.5,postaction={decoration={
        markings,
        mark=at position .7 with {\arrow[#1]{Stealth[sep=-3pt]}}
      },decorate}},
  pd/.style ={draw=lightgray,thick,densely dotted,edge node={node[above=-7pt] {\color{gray}$\partial$}},anchor=south,pos=0.5,postaction={decoration={
        markings,
        mark=at position .7 with {\arrow[#1]{Stealth[sep=-3pt]}}
      },decorate}},
  pdr/.style ={g,draw=lightgray,thick,densely dotted,edge node={node[above=-7pt] {\color{red}$\partial$}},anchor=south,pos=0.5},
  pdkr/.style args={#1}{g,draw=lightgray,thick,densely dotted,edge node={node[above=-7pt] { \color{gray}$\partial^{#1}$\color{red}$\partial$}},anchor=south,pos=0.5},
  pdh/.style ={g,draw=lightgray,thick,densely dotted,edge node={node[above=-7pt] {\color{gray}$\widehat\partial$}},anchor=south,pos=0.5},
  pdk/.style args={#1}{g,draw=lightgray,thick,densely dotted,edge node={node[above=-7pt] {\color{gray}$\partial^{#1}$}},anchor=south,pos=0.5},
  eq/.style = {double,postaction={decoration={name=none}}},
  inl/.style args={#1}{initial,initial where=left, initial text=#1,initial distance=10pt},
  pdn/.style args={#1}{g,draw=lightgray,thick,densely dotted,edge node={node[above=-7pt] {\color{gray}$\partial^{#1}$}},anchor=south,pos=0.5},
  inr/.style args={#1}{initial,initial where=right, initial text=#1,initial distance=10pt},
  int/.style args={#1}{initial,initial where=above, initial text=#1,initial distance=10pt},
  inb/.style args={#1}{initial,initial where=below, initial text=#1,initial distance=10pt},
  lbr/.style args={#1}{initial,initial where=left, initial text=$\br$,initial distance=10pt},
  tbr/.style args={#1}{initial,initial where=above, initial text=$\br$,initial distance=10pt},
  bbr/.style args={#1}{initial,initial where=below, initial text=$\br$,initial distance=10pt},
  rbr/.style args={#1}{initial,initial where=right, initial text=$\br$,initial distance=10pt},
  lbk/.style args={#1}{initial,initial where=left, initial text=$\bk^{(b)}$,initial distance=10pt},
  tbk/.style args={#1}{initial,initial where=above, initial text=$\bk^{(b)}$,initial distance=10pt},
  bbk/.style args={#1}{initial,initial where=below, initial text=$\bk^{(b)}$,initial distance=10pt},
  rbk/.style args={#1}{initial,initial where=right, initial text=$\bk^{(b)}$,initial distance=10pt},
  lmok/.style args={#1}{initial,initial where=left, initial text=$\bm_o\bk^{(b)}$,initial distance=10pt},
  tmok/.style args={#1}{initial,initial where=above, initial text=$\bm_o\bk^{(b)}$,initial distance=10pt},
  bmok/.style args={#1}{initial,initial where=below, initial text=$\bm_o\bk^{(b)}$,initial distance=10pt},
  rmok/.style args={#1}{initial,initial where=right, initial text=$\bm_o\bk^{(b)}$,initial distance=10pt},
  lmdk/.style args={#1}{initial,initial where=left, initial text=$\bm_d\bk^{(b)}$,initial distance=10pt},
  mdtk/.style args={#1}{initial,initial where=above, initial text=$\bm_d\widetilde{\bk}^{(b)}$,initial distance=10pt},
  bmdk/.style args={#1}{initial,initial where=below, initial text=$\bm_d\bk^{(b)}$,initial distance=10pt},
  rmdk/.style args={#1}{initial,initial where=right, initial text=$\bm_d\bk^{(b)}$,initial distance=10pt},
  l1/.style args={#1}{initial,initial where=left, initial text=$\mathbf 1$,initial distance=10pt},
  t1/.style args={#1}{initial,initial where=above, initial text=$\mathbf 1$,initial distance=10pt},
  b1/.style args={#1}{initial,initial where=below, initial text=$\mathbf 1$,initial distance=10pt},
  r1/.style args={#1}{initial,initial where=right, initial text=$\mathbf 1$,initial distance=10pt},
  lm/.style args={#1}{initial,initial where=left, initial text=$\vm$,initial distance=10pt},
  tm/.style args={#1}{initial,initial where=above, initial text=$\vm$,initial distance=10pt},
  bm/.style args={#1}{initial,initial where=below, initial text=$\vm$,initial distance=10pt},
  rm/.style args={#1}{initial,initial where=right, initial text=$\vm$,initial distance=10pt},
  inr/.style args={#1}{initial,initial where=right, initial text=#1,initial distance=10pt},
  B/.style args={#1}{thick,draw=lightgray,decorate,decoration={snake,amplitude=.4mm,segment length=.8mm,post length=1.4mm},edge node={node[above=-7pt] {\color{gray} #1 }},anchor=south,pos=0.5},
  br/.style = {bend right},
  b0/.style = {bend left=0},
  bl/.style = {bend left},
  glb/.style = {looseness=20,in =220, out=320},
  gll/.style = {looseness=20,in =130, out=230},
  glr/.style = {looseness=20,in =310, out=50},
  glt/.style = {looseness=20,in =40, out=140},
  gm/.style={postaction={decorate,decoration={
        markings,
        mark=at position .3 with {\arrow[#1]{Diamond[open,sep=-3pt,width=5pt]}}
      }}},
}

\newcommand\ssGraph[1]{
\begin{tikzpicture}[subgraph text none,grow=right,baseline={([yshift=-2pt]current bounding box.center)},font=\footnotesize,>=Stealth]% draw the graph
  \graph[spring electrical layout,components go down left aligned,nodes={draw,circle,as=,minimum size=3pt,inner sep=0pt,fill},node distance = 15pt,component sep=15pt]{ #1  };
  \end{tikzpicture}
}

\tikzset{circle split part fill/.style  args={#1}{%
 alias=tmp@name, % Jake's idea !!
  postaction={%
    insert path={
     \pgfextra{% 
     \pgfpointdiff{\pgfpointanchor{\pgf@node@name}{center}}%
                  {\pgfpointanchor{\pgf@node@name}{east}}%            
     \pgfmathsetmacro\insiderad{\pgf@x}
      \fill[white,fill opacity=0] (\pgf@node@name.base) ([xshift=-\pgflinewidth]\pgf@node@name.east) arc
                          (0:180:\insiderad-\pgflinewidth)--cycle;
      \fill[fill=white,preaction={fill, white},pattern=#1] (\pgf@node@name.base) ([xshift=\pgflinewidth]\pgf@node@name.west)  arc
                           (180:360:\insiderad-\pgflinewidth)--cycle;   
      \draw[line width=0.4pt] (\pgf@node@name.base) ([xshift=\pgflinewidth]\pgf@node@name.west)  arc
                           (180:360:\insiderad-\pgflinewidth)--cycle;                                %  \end{bscope}   
         }}}}}  
\makeatother  
\makeatletter
\tikzset{my loop/.style =  {to path={
  \pgfextra{}
  [looseness=6,min distance=4mm]
  \tikz@to@curve@path},font=\sffamily\small
  }}  
\makeatletter

\csedef{pat0}{}
\csedef{pat1}{north east lines}
\csedef{pat2}{crosshatch dots}
\csedef{pat3}{crosshatch}
\csedef{pat4}{grid}

\definecolor{col0}{HTML}{FFFFFF}
\definecolor{col1}{HTML}{A2B969}
\definecolor{col2}{HTML}{EBCB38}
\definecolor{col3}{HTML}{0D95BC}
\definecolor{col4}{HTML}{063951}
\definecolor{col5}{HTML}{F36F13}
\definecolor{col6}{HTML}{C13018}
\definecolor{lightgray}{HTML}{CCCCCC}

\newcommand\csum[1]{%
\sum_{\forcsvlist{\createColorCircle@item}{#1}}
}

\newcommand\createColorCircle@item[1]{
\StrDel{#1}{h}[\colNum]
\newif\ifhalf
\IfSubStr{#1}{h}{\halftrue}{\halffalse}
{\color{col\colNum}\ifhalf\circ\else\bullet\fi}
}

\newcommand\plotLambda[1]{
  \def\mgraphspecs{}
  \foreach \ll [count = \countl] in {#1} {
    \csedef{firstCol}{col0}
    \def\graphspecs{}
    \StrCount{\ll}{,}[\graphlen]
    \ifnum\graphlen>1 
    \foreach \node [count = \g] in \ll {
      \StrDel{\node}{m}[\nodenumber]
      \StrDel{\nodenumber}{.}[\nodenumber]
      \global\csedef{tempCol}{col\nodenumber}
      \global\csdef{tempPat}{\col{1}}
      \ifnum\g=1
        \global\csedef{firstCol}{\tempCol}
        \IfSubStr{\node}{.}{\global\csedef{open}{y}}{\global\csedef{open}{n}}
        \IfStrEqCase{\open}{
          {y}{\xappto\graphspecs{\countl\g[as=,draw=none] }}
          {n}{\xappto\graphspecs{\countl\g[as=,preaction={fill, white},fill=\tempCol,pattern=\csname pat\nodenumber\endcsname] }}
        }
      \else
        \IfBeginWith{\node}{m}{\global\csedef{secondArrow}{{<[sep=-3pt,length=8pt]}}}{\global\csedef{secondArrow}{}}%
        \ifnum\g=2
          \IfStrEqCase{\open}{
            {y}{\xappto\graphspecs{ --[dotted,thick] \countl\g[as=,preaction={fill, white},fill=\tempCol,pattern=\csname pat\nodenumber\endcsname] }}
            {n}{\xappto\graphspecs{ --[\firstArrow-\secondArrow] \countl\g[as=,preaction={fill, white},fill=\tempCol,pattern=\csname pat\nodenumber\endcsname] }}
          }
        \else
          \xappto\graphspecs{ --[\firstArrow-\secondArrow] \countl\g[as=,preaction={fill, white},fill=\tempCol,pattern=\csname pat\nodenumber\endcsname] }%
        \fi
      \fi
      \IfEndWith{\node}{m}{\global\csedef{firstArrow}{{>[sep=-3pt,length=8pt]}}}{\global\csedef{firstArrow}{}}
      \global\csedef{prevTempCol}{\tempCol} 
    }
    \IfStrEqCase{\open}{
      {n}{\IfBeginWith{\ll}{m}{\global\csedef{secondArrow}{{<[sep=-3pt,length=8pt]}}}{\global\csedef{secondArrow}{}}
          \xappto\graphspecs{ --[\firstArrow-\secondArrow,decorate,decoration={snake,amplitude=.3mm,segment length=.6mm}] \countl1; } }
      {y}{\xappto\graphspecs{ --[dotted,thick] \countl0[draw=none,as=] ; } }
    }
    \else
      \ifnum\graphlen=0
        \StrDel{\ll}{m}[\nodenumber]
        \csedef{tempCol}{col\nodenumber}
        \IfSubStr{\ll}{m}{\csedef{firstArrow}{{>[sep=-3pt,length=8pt]}}}{\csedef{firstArrow}{}}
        \def\graphspecs{ \countl1[as=,preaction={fill, white},fill = \tempCol,pattern=\csname pat\nodenumber\endcsname] --[my loop, decorate,decoration={snake,amplitude=.3mm,segment length=.6mm},\firstArrow-] \countl1; }
      \else
        \IfSubStr{\ll}{.}{
          \StrDel{\ll}{.}[\nodenumber]
          \StrDel{\nodenumber}{,}[\nodenumber]
          \csedef{tempCol}{col\nodenumber}
          \def\graphspecs{ \countl0[draw=none,as=,orient = left] --[dotted,thick] \countl1[as=,fill = \tempCol,pattern=\csname pat\nodenumber\endcsname] --[dotted,thick] \countl2[draw=none,as=,nudge down=10pt]; }
        }{
          \StrBefore{\ll}{,}[\nodeOne]
          \StrBehind{\ll}{,}[\nodeTwo]
          \StrDel{\nodeOne}{m}[\nodeOneNumber]
          \StrDel{\nodeTwo}{m}[\nodeTwoNumber]
          \def\tempColOne{col\nodeOneNumber}
          \def\tempColTwo{col\nodeTwoNumber}
          \IfEndWith{\nodeOne}{m}{\global\csedef{firstArrowOne}{{>[sep=-3pt,length=8pt]}}}{\global\csedef{firstArrowOne}{}}
          \IfEndWith{\nodeTwo}{m}{\global\csedef{firstArrowTwo}{{<[sep=-3pt,length=8pt]}}}{\global\csedef{firstArrowTwo}{}}
          \IfBeginWith{\nodeOne}{m}{\global\csedef{secondArrowOne}{{>[sep=-3pt,length=8pt]}}}{\global\csedef{secondArrowOne}{}}
          \IfBeginWith{\nodeTwo}{m}{\global\csedef{secondArrowTwo}{{<[sep=-3pt,length=8pt]}}}{\global\csedef{secondArrowTwo}{}}
          \def\graphspecs{ \countl1[as=,preaction={fill, white},fill = \tempColOne,pattern=\csname pat\nodeOneNumber\endcsname] --[bend right,decorate,decoration={snake,amplitude=.3mm,segment length=.6mm}, \firstArrowOne-\secondArrowTwo ] \countl2[as=,preaction={fill, white},fill = \tempColTwo,pattern=\csname pat\nodeTwoNumber\endcsname]; \countl1 --[bend left,\secondArrowOne-\firstArrowTwo] \countl2;}
        }
      \fi
    \fi
    \xappto\mgraphspecs{ \graphspecs }
  }
  \xdef\mgraphspecs{\noexpand\graph[simple necklace layout,componentwise,component packing=skyline,components go right center aligned,orient=0,nodes=G]{ \mgraphspecs }}
  \begin{tikzpicture}[baseline={([yshift=-2pt]current bounding box.center)},font=\tiny,>=Stealth, node distance = 15pt,node sep=12pt,component sep=5pt]% draw the graph
    \mgraphspecs;
  \end{tikzpicture}%
}

\newcommand\pB[1]{
  \StrDel{#1}{h}[\patNum]
  \IfSubStr{#1}{h}{
  \begin{tikzpicture}[baseline={([yshift=-2pt]current bounding box.center)},font=\tiny,>=Stealth, node distance = 0pt,node sep=1pt,component sep=1pt]% draw the graph
   \graph[simple necklace layout,componentwise,component packing=skyline,components go right center aligned,orient=0,nodes=G] { 1[Gf,rotate=270,circle split part fill={\csname pat\patNum\endcsname},as=,minimum size=6pt]; };
  \end{tikzpicture}
  }{
  \begin{tikzpicture}[baseline={([yshift=-2pt]current bounding box.center)},font=\tiny,>=Stealth, node distance = 0pt,node sep=1pt,component sep=1pt]% draw the graph
   \graph[simple necklace layout,componentwise,component packing=skyline,components go right center aligned,orient=0,nodes=G] { 1[pattern=\csname pat\patNum\endcsname,as=,minimum size=6pt]; };
  \end{tikzpicture}
  }
}
\newcommand\plotlLambda[1]{
  \def\mgraphspecs{}
  \foreach \ll [count = \countl] in {#1} {
    \csedef{firstCol}{col0}
    \csedef{openEnd}{no}
    \def\graphspecs{}
    \StrCount{\ll}{,}[\graphlen]
    \foreach \node [count = \g] in \ll {
      \StrDel{\node}{m}[\nodenumber]
      \csedef{type}{n}
      \csedef{marking}{no}
      \IfSubStr{\node}{x}{\csedef{type}{x}}{}
      \IfSubStr{\node}{c}{\csedef{type}{c}}{}
      \IfSubStr{\node}{d}{\csedef{type}{d}}{}
      \IfSubStr{\node}{.}{\csedef{type}{open}}{}
      \IfBeginWith{\node}{m}{\csedef{firstMarking}{yes}\csedef{marking}{yes}}{\csedef{firstMarking}{no}}
      \IfEndWith{\node}{m}{\csedef{secondMarking}{yes}\csedef{marking}{yes}}{\csedef{secondMarking}{no}}
      \StrDel{\nodenumber}{.}[\nodenumber]
      \StrDel{\nodenumber}{c}[\nodenumber]
      \StrDel{\nodenumber}{d}[\nodenumber]
      \StrDel{\nodenumber}{x}[\nodenumber]
      \global\csedef{tempCol}{col\nodenumber}
      \ifnum\g=1
        \global\csedef{firstCol}{\tempCol}
        \IfStrEq{\type}{open}{
          \xappto\graphspecs{ \countl0[xx,as=,grow right,draw=none,fill=white] --[dotted,thick] \countl1[fill=white,preaction={fill, white},pattern=\csname pat\nodenumber\endcsname,as=]}
        }{
        \IfSubStr{\node}{x}{
          \xappto\graphspecs{ \countl0[xx,as=,grow right,fill=lightgray,draw=none] --[color=lightgray]\countl1[Gf,circle split part fill={\csname pat\nodenumber\endcsname},as=]}}
        {
          \IfStrEq{\marking}{yes}
          {
            \xappto\graphspecs{ \countl0[xx,as=,grow right] --[-{<[fill=lightgray,color=black,sep=-3pt,length=8pt]}] \countl1[fill=white,preaction={fill, white},pattern=\csname pat\nodenumber\endcsname,as=]}
          }{
            \xappto\graphspecs{ \countl0[xx,as=,grow right] -- \countl1[fill=white,preaction={fill, white},pattern=\csname pat\nodenumber\endcsname,as=]}
          }
        }}
      \else
        \IfStrEq{\firstMarking}{yes}{\csedef{secondArrow}{{<[sep=-3pt,length=8pt]}}}{\csedef{secondArrow}{}}
        \IfStrEqCase{\type}{
          {n}{\xappto\graphspecs{ --[\firstArrow-\secondArrow] \countl\g[fill=white,preaction={fill, white},pattern=\csname pat\nodenumber\endcsname,as=] }}%
          {c}{\xappto\graphspecs{ --[\firstArrow-\secondArrow] \countl\g[Gc=\csname pat\nodenumber\endcsname,rotate=180,circle split part fill={\csname pat\nodenumber\endcsname},as=] }}%
          {d}{\xappto\graphspecs{ --[\firstArrow-\secondArrow] \countl\g[Gd=\csname pat\nodenumber\endcsname,rotate=180,circle split part fill={\csname pat\nodenumber\endcsname},as=] }}%
          {open}{\global\csedef{openEnd}{yes}}%
        }
      \fi
      \IfStrEq{\secondMarking}{yes}{
        \global\csedef{firstArrow}{{>[sep=-3pt,length=8pt]}}%
      }{
        \global\csedef{firstArrow}{}%
      }
      \global\csedef{prevTempCol}{\tempCol} 
    }
    \IfStrEqCase{\openEnd}{
      {yes}{\xappto\mgraphspecs{ \graphspecs --[dotted,thick] \countl17[Gend,as=,draw=none]; }}%
      {no}{\xappto\mgraphspecs{ \graphspecs -- \countl17[Gend,as=]; }}%
    }%
  }
  \xdef\mgraphspecs{\noexpand\graph[tree layout,componentwise,component packing=skyline,components go down left aligned,nodes=G]{ \mgraphspecs }}
  \begin{tikzpicture}[grow=right,baseline={([yshift=-2pt]current bounding box.center)},font=\tiny,>=Stealth, node distance = 5pt,node sep=12pt,component sep=5pt]% draw the graph
  \mgraphspecs;
  \end{tikzpicture}%
}

\usepgfplotslibrary{external} 			% add before uploading to arxiv 
\tikzsetfigurename{tikz-figure} 		% add before uploading to arxiv 
\pgfplotsset{plot coordinates/math parser=false}

\newcommand{\titel}{Spectral radius of random matrices with independent entries}
\title{\titel} 
\author{Johannes Alt\footnote{Partially supported by ERC Starting Grant RandMat No.\ 715539 and the SwissMap grant of Swiss National Science Foundation.}\\ 
{\small \begin{tabular}{c}{University of Geneva}\\{johannes.alt@unige.ch} \end{tabular}} 
\and László Erd\H{o}s\footnote{Partially supported by ERC Advanced Grant RanMat No.\ 338804.}
\\{\small \begin{tabular}{c} IST Austria\\ lerdos@ist.ac.at\end{tabular}} 
\and Torben Krüger\footnote{Partially supported by the Hausdorff Center for Mathematics in Bonn.}\\
{\small \begin{tabular}{c} University of Copenhagen\\ tk@math.ku.dk\end{tabular}}
}
\date{}

\hypersetup{
	pdftitle={\titel},
	pdfauthor={Johannes Alt, László Erd\H{o}s, Torben Krüger}%,
}

\numberwithin{equation}{section}

\newcommand{\R}{\mathbb{R}}  % The real numbers.
\ifdefined\C\renewcommand{\C}{\mathbb{C}}\else\newcommand{\C}{\mathbb{C}}\fi % Complex numbers
\renewcommand{\Im}{\mathrm{Im}\,} %imaginary part of a complex number
\renewcommand{\Re}{\mathrm{Re}\,} %real part of a complex number
\newcommand{\N}{\mathbb{N}}  % Positive integers.	
\newcommand{\E}{\mathbb{E}}  % expected value of random variable	
\renewcommand{\P}{\mathbb{P}}  % probability measure
\newcommand{\di}{\mathrm{d}} % differential
\newcommand{\sign}[1]{\mathrm{sign}(#1)} % differential
\newcommand{\eps}{\varepsilon} % "correct" epsilon
\newcommand*{\defeq}{\mathrel{\vcenter{\baselineskip0.5ex \lineskiplimit0pt\hbox{\scriptsize.}\hbox{\scriptsize.}}}=}

\newcommand{\pt}{\partial}

\DeclareMathOperator{\supp}{supp}

\DeclareMathOperator{\ran}{ran}
\DeclareMathOperator{\rank}{rank}

\DeclareMathOperator{\graph}{Gr} %Graph of a map f: X \to Y in X\times Y

\newcommand{\DD}{\mathbb{D}}

\newcommand{\smallS}{\mathscr{S}}
\newcommand{\smallF}{\mathscr{F}}
\newcommand{\smallT}{\mathscr{T}}

\newcommand{\etaf}{\eta_{\mathrm{f}}}

\newcommand{\dM}{{\mathcal{M}_\mathrm{d}}}
\newcommand{\oM}{{\mathcal{M}_\mathrm{o}}}

\newcommand{\bu}{\mathbf u}
\newcommand{\bv}{\mathbf v}
\newcommand{\bx}{\mathbf x}
\newcommand{\be}{\mathbf e}
\newcommand{\by}{\mathbf y}

\newcommand{\br}{\mathbf r}
\newcommand{\bk}{\mathbf k} 
\newcommand{\bm}{\mathbf m}

\newcommand{\cA}{\mathcal{A}}
\newcommand{\cB}{\mathcal{B}} 
\newcommand{\cC}{\mathcal{C}}
\newcommand{\cD}{\mathcal{D}}
\newcommand{\cK}{\mathcal{K}}
\newcommand{\cL}{\mathcal{L}}

\newcommand{\cP}{\mathcal{P}}

\newcommand{\cQ}{\mathcal{Q}}
 
\newcommand{\cF}{\mathcal{F}} 
\newcommand{\cS}{\mathcal{S}} 
\newcommand{\cT}{\mathcal{T}} 
\newcommand{\cV}{\mathcal{V}}

\newcommand{\PK}{\mathcal{P}_{\mathcal{K}}}
\newcommand{\QK}{\mathcal{Q}_{\mathcal{K}}}

\renewcommand{\rm}{\mathrm} %upright

\newcommand{\normtwo}[1]{\lVert #1 \rVert_{2}}
\newcommand{\normtwoop}[1]{\lVert #1 \rVert_{2\to 2}}

\newcommand{\normtwoinf}[1]{\lVert #1 \rVert_{2\to\lVert\,\cdot\,\rVert}}

\newcommand{\normstar}[1]{\lVert #1 \rVert_{*}}
\newcommand{\normstarop}[1]{\lVert #1 \rVert_{\ast \to \ast}}

\newtheoremstyle{test}% name
  {}%      Space above, empty = `usual value'
  {}%      Space below
  {\itshape}% Body font
  {}%         Indent amount (empty = no indent, \parindent = para indent)
  {\bfseries}% Thm head font
  {.}%        Punctuation after thm head 
  { }% Space after thm head: \newline = linebreak
  {}%         Thm head spec

\theoremstyle{test}
\newtheorem{defi}{Definition}[section]

\newtheorem*{rem*}{Remark}   %no numbering
\newtheorem*{ex*}{Example}   %no numbering
\newtheorem*{pro*}{Proposition} %no numbering
\newtheorem*{def*}{Definition}
\newtheorem*{coro*}{Corollary}
\newtheorem*{thm*}{Theorem}

\theoremstyle{test}
    \newtheorem{theorem}[defi]{Theorem}
    \newtheorem{proposition}[defi]{Proposition}
    \newtheorem{corollary}[defi]{Corollary}
    \newtheorem{lemma}[defi]{Lemma}
    \newtheorem{definition}[defi]{Definition}% howto make rm-style text inside definitions
    
    \newtheorem{convention}[defi]{Convention}
    \newtheorem{remark}[defi]{Remark}

\newcommand{\bels}[2] {
        \begin{equation} \label{#1} \begin{split} 
                #2 
        \end{split} \end{equation}
        }

\renewcommand{\cal}{\mathcal}

\newcommand{\ol}[1]{\overline{#1} \!\,} %overline
\newcommand{\wh}{\widehat}
\newcommand{\wt}{\widetilde}

\newcommand{\ord} {\mathcal{O}}

\renewcommand{\P}{\mathbb{P}}

\newcommand{\ii}{\mathrm{i}} %\newcommand{\mi}{\mathrm{i}}
\newcommand{\dd}{\mathrm{d}}

\newcommand{\pb}[1]{\bigl({#1}\bigr)}
\newcommand{\pbb}[1]{\biggl({#1}\biggr)}

\renewcommand{\sb}[1]{\bigl[{#1}\bigr]}

\newcommand{\abs}[1]{\lvert #1 \rvert}
\newcommand{\absb}[1]{\big\lvert #1 \big\rvert}

\newcommand{\absbb}[1]{\bigg\lvert #1 \bigg\rvert}

\newcommand{\absa}[1]{\left\lvert #1 \right\rvert}

\newcommand{\norm}[1]{\lVert #1 \rVert}
\newcommand{\normb}[1]{\big\lVert #1 \big\rVert}

\newcommand{\normbb}[1]{\bigg\lVert #1 \bigg\rVert}

\newcommand{\avg}[1]{\langle #1 \rangle}

\newcommand{\scalar}[2]{\langle{#1} \mspace{2mu}, {#2}\rangle}

\DeclareMathOperator{\diag}{diag}

\DeclareMathOperator{\Tr}{Tr}

\DeclareMathOperator{\im}{Im}

\DeclareMathOperator*{\spec}{Spec}						%Spectrum

\newcommand{\2} {\mspace{2 mu}}

\newcommand{\genarg} {{\,\cdot\,}}  % general argument: f(\genarg) produces dot with little space around it

\makeatletter
\def\blfootnote{\gdef\@thefnmark{}\@footnotetext}
\makeatother

\begin{document}
\maketitle

\vspace*{-1cm} 

\begin{abstract}
We consider random $n\times n$ matrices $X$ with independent and centered entries and a general variance profile. 
We show that the spectral radius of $X$ converges with very high probability to the square root of the spectral radius of the variance matrix of $X$ when $n$
tends to infinity. We also establish the optimal rate of convergence,  that is a new result even for general i.i.d.\ matrices beyond
 the explicitly solvable Gaussian cases. 
The main ingredient is the proof of the local inhomogeneous circular law \cite{Altcirc} at the spectral edge. 
\end{abstract}

\blfootnote{Date: \today}
\blfootnote{MSC2010 Subject Classfications: 60B20, 15B52.} 

\tableofcontents

\section{Introduction} 

Girko's celebrated \emph{circular law}  \cite{Girko1984,bai1997}\footnote{The original proof in \cite{Girko1984}  
  was not
considered complete and Bai published a clean version  under more restrictive conditions \cite{bai1997}. 
An  extended version of Girko's original proof with explanations and corrections appeared in \cite[Chapter~6]{GirkoBook1998}, see also \cite{Girko2012}.}
asserts that the spectrum of an $n\times n$ random matrix $X$ with  centered, independent, identically distributed
(i.i.d.) entries with variance $\E |x_{ij}|^2=1/n$ converges, as $n\to \infty$,  to the unit disc with a uniform 
limiting density of eigenvalues.  The cornerstone of the proof is the \emph{Hermitization
formula} (cf.\ \eqref{eq:girko}) that connects eigenvalues of $X$ to the eigenvalues of a family 
of Hermitian matrices $(X-z)^*(X-z)$ with a complex parameter $z$  \cite{Girko1984}. 
The circular law for i.i.d.\ entries with
the minimal second moment condition was established by Tao and Vu  \cite{tao2010} after several partial
results \cite{goetze2010,Pan2010,tao2008}, see  \cite{bordenave2012} for the extensive history and literature.
 We also refer to the recent circular law for very sparse matrices \cite{Rudelson2019}. 

The circular law establishes the weak limit of the empirical density
of eigenvalues  and thus it accounts for most but not all of them. In particular, it
does not give information on  the spectral radius $\varrho(X)$ of $X$ since the largest (in absolute value)
eigenvalue may behave very differently than the bulk spectrum. In fact, such outliers do not exist  but
this requires a separate proof.
The convergence of the spectral radius of $X$ to 1,
\begin{equation}\label{bai}
   \varrho(X)\to 1, \qquad \mbox{almost surely as $n\to \infty$,}
\end{equation}
was proven by Bai and Yin 
in  \cite{Bai1986} under the fourth moment condition, $\E  | n^{1/2} x_{ij} |^4 \le C$, using Wigner's moment method.
Under stronger conditions the upper bound in \eqref{bai} was independently proven in \cite{Geman1986}, see also \cite{GemanHwang1982,Nemish2018}. 
More recently in  \cite{BordenaveSpectralRadius2018} the convergence $\varrho(X)\to 1$ in probability was
shown assuming only finite
$2+\epsilon$ moment as well as a symmetric entry distribution.

Precise information on the spectral radius  is available only for the Ginibre ensemble,
i.e.\ when $x_{ij}$ are Gaussian; in this case it is known
\cite{Rider2003,RiderSinclair2014} 
that 
\begin{equation}\label{rider}
   \varrho(X) \approx 1 + \sqrt{\frac{\gamma_n}{ 4n}} + \frac{1}{ \sqrt{4n\gamma_n}} \xi, \qquad \gamma_n\defeq\log \frac{n}{2\pi}
   -2\log\log n,
\end{equation}
where $\xi$ is a Gumbel distributed random variable.

In this paper we drop the condition that the matrix elements are identically distributed
and we study the spectral radius of $X$ when the variances $\E |x_{ij}|^2$
have a non-trivial profile given by the matrix $\smallS=(\E |x_{ij}|^2)_{i,j=1}^n$. In our previous work  \cite{Altcirc} we showed
that  the spectral radius of $X$ is arbitrarily close  to the square root of the spectral radius of $\smallS$. More precisely,
for any fixed $\epsilon>0$ we have
\begin{equation}\label{oldrho}
\sqrt{\varrho(\smallS)} - \epsilon\le \varrho(X)\le \sqrt{\varrho(\smallS)} + \epsilon
\end{equation}
with very high probability for large $n$.  Motivated by \eqref{rider} we expect that the precision 
of the approximation in \eqref{oldrho} can be greatly improved and 
the difference between $\varrho(X)$ and $\sqrt{\varrho(\smallS)}$ should not exceed
$n^{-1/2}$ by  much. Indeed, our first  main result proves that  for any $\epsilon>0$ we have
\begin{equation}\label{oldrho1}
\sqrt{\varrho(\smallS)} - n^{-1/2+\epsilon}\le \varrho(X)\le \sqrt{\varrho(\smallS)} + n^{-1/2+ \epsilon}
\end{equation}
with very high probability for large $n$.  Apart from the $n^{\epsilon}$ factor this result
is optimal considering \eqref{rider}. Note that \eqref{oldrho1} is new even for the i.i.d.\ 
case beyond  Gaussian, i.e.\ there is no previous result on the speed of convergence in \eqref{bai}.

We remark that, compared with the spectral radius, much more is known about 
 the largest singular value of $X$ since it is equivalent to the (square root of the)
 largest eigenvalue of the sample covariance matrix $XX^*$. For the top eigenvalues of $XX^*$
 precise   limiting behavior (Tracy-Widom) is known  if $X$ has general i.i.d.\ matrix elements
 \cite{pillai2014}, and even 
 general diagonal  population matrices are allowed \cite{Lee2016}.  Note, however, the largest singular value of $X$ in the i.i.d.\ case
 converges to 2, i.e.\ it is very different from the spectral radius, indicating that $X$ is far from 
being normal.  We stress that understanding the spectral radius is a genuinely non-Hermitian problem hence 
in general it is much harder  than  studying the largest singular value.

While the largest singular value is very important for statistical applications, the
spectral radius is relevant for  time evolution of complex systems. More precisely, 
the spectral radius
controls the eigenvalue with largest real part that  plays
an important role in understanding the long time behavior of  large systems of linear ODE's
with random coefficients of the form 
\begin{equation}\label{ode}
\frac{\di}{\di t}  u_t =  -  g u_t + X u_t 
\end{equation}
 with a tunable coupling constant $g$.
 Such ODE system was first introduced in an ecological model to 
study the interplay between complexity and stability
in May's seminal paper \cite{may1972will},  see also the recent exposition \cite{Allesina2015}.  
 It has since been applied to  many situations when a  transience phenomenon 
is modelled in dynamics of complex systems; especially for neural networks, e.g.\
 \cite{Sompolinsky1988, HENNEQUIN20141394, Grela2017}. 
 Structured neural networks require to generalize May's original i.i.d.\ model to  non-constant variance profile $\smallS$
  \cite{Aljadeff2015,Muir2015,Rajan2006,Gudowska-Nowak2020}
which we study in full generality.
The long time evolution of~\eqref{ode} at critical coupling $g_c\defeq\sqrt{\varrho(\smallS)}$
 in the i.i.d.\ Gaussian case was  computed in \cite{ChalkerMehlig1998} after some non-rigorous steps;
the full mathematical analysis  even for general distribution and beyond the i.i.d. setup
 was given in \cite{ErdosKrugerRenfrew2018,ErdosKrugerRenfrew2019}.  The time-scale on which the solution of~\eqref{ode} 
 at criticality can be computed  depends
 on how precisely $\varrho(X)$ can be controlled by $\sqrt{\varrho(\smallS)}$. In particular, the current improvement of this precision 
 to~\eqref{oldrho1} allows one to extend the result of~\cite[Theorem 2.6]{ErdosKrugerRenfrew2018} to very long
  time scales of order $n^{1/2-\epsilon}$. These applications require a separate analysis, we will not pursue them
   in the present work.

We now explain the key novelties of this paper, more details will be given in Section~\ref{sec:outline}
after presenting the precise results  in Section~\ref{sec:main_results}. The spectral radius of $X$ is ultimately
related to our second main result,  the \emph{local law} for $X$ near the spectral edges, i.e.\ a description of
the eigenvalue density on local scales but still above the eigenvalue spacing; in this case $n^{-1/2}$.
As a byproduct, we also prove the optimal $1/n$ speed of convergence in the inhomogeneous  circular law
\cite{Altcirc,  Cook2018}. 
Note that the limiting density has a discontinuity at the boundary of its support, the disk of radius $\sqrt{\varrho(\smallS)}$
\cite[Proposition~2.4]{Altcirc}, hence the 
typical eigenvalue spacing at the edge and in the bulk coincide, unlike for the Hermitian problems.
The local law  in the bulk for $X$ with a general variance profile has been established in \cite[Theorem~2.5]{Altcirc} on  scale $n^{-1/2+\epsilon}$ and with optimal error bounds. This entails an optimal  local law near zero for the 
\emph{Wigner-type} Hermitian matrix
\begin{equation} \label{eq:def_H_z_intro} 
 H_z \defeq \begin{pmatrix} 0 & X -z \\ (X-z)^* & 0 \end{pmatrix} 
\end{equation}
appearing in Girko's formula. As long as $z$ is in the bulk spectrum of $X$, the relevant spectral parameter 0
lies in the bulk spectrum of $H_z$. Still, the local law for Wigner-type matrices \cite{Ajankirandommatrix}
is not applicable since the \emph{flatness condition},
that requires the variances of all matrix elements of $H_z$ be comparable, is violated
by the large zero blocks in $H_z$. 
In fact, the 
corresponding Dyson equation has an unstable direction due to the \emph{block symmetry} of $H_z$. The main achievement
of \cite{Altcirc} was to handle this instability.

When $z$ is near the spectral edge of $X$, the density of $H_z$ develops a cusp singularity
at 0. The optimal cusp local law for Wigner-type matrices with flatness condition was proven recently in 
\cite{Cusp1}  relying on (i) the improved \emph{fluctuation averaging mechanism} and (ii)
 the deterministic analysis of the corresponding Dyson equation
in \cite{Altshape}. Due to the cusp, the Dyson equation has a natural unstable direction and the corresponding
non-Hermitian perturbation
theory is governed by a cubic equation. 

The Dyson equation corresponding to the matrix $H_z$ for $z$ near the spectral edge of $X$
 exhibits  \emph{both} instabilities simultaneously.
This leads to the main technical achievement of this paper:
 we  prove an  \emph{optimal local law in the cusp regime with the block instability}.
Most of the paper contains  our  refined analysis of the Dyson equation with two instabilities,
a delicate synthesis of the methods developed in  \cite{Altcirc} and \cite{Altshape}.
The necessary  fluctuation averaging argument, however, turns out to be simpler than in \cite{Cusp1},
the block symmetry here helps.
 
We remark that bulk and edge local laws for the i.i.d.\ case have been proven earlier 
\cite{Bourgade2014,BYY_circular2}  with the optimal scale at the edge in \cite{Y_circularlaw}  and later with improved moment assumptions in \cite{Gotze2017};
see also \cite{tao2015}  for similar results under  three moment 
matching condition.   However, these works did not provide the improved local law outside of the spectrum that is necessary to identify the spectral radius.   The main difference is that the i.i.d.\ case results in an explicitly solvable
scalar-valued Dyson equation, so the entire stability analysis boils down to analysing explicit formulas.
The inhomogeneous variance profile $\smallS$ leads to a vector-valued Dyson equation
with no explicit solution at hand; all stability properties must be obtained inherently from the equation itself.
 Furthermore, even in the i.i.d.\ case the local law for $H_z$ in~\cite{BYY_circular2,Y_circularlaw} 
 was not optimal in the edge regime $|z|\approx 1$ and the authors directly estimated only the specific error terms in Girko's formula. 
  The optimality of our local law for $H_z$ at the edge
 is the main reason why the proof of the local circular law in Section~\ref{sec:Xlaw} is very transparent. 
In fact, our current local law is formulated  in the \emph{isotropic} sense (see \eqref{eq:G_minus_M_standard} later)
which is more general than the result in \cite{Bourgade2014,BYY_circular2} even in the i.i.d. case. This generalised version 
was an essential ingredient in the recent proof of edge universality for i.i.d. matrices \cite{cipolloni2019edge}.

\paragraph{Acknowledgements.}  The authors are grateful to Dominik Schr\"oder for  valuable insights, 
discussions on adapting the fluctuation averaging mechanism in \cite{Cusp1} to the current setup
as well as for kindly making his graph drawing macros available to us. We also thank Ga\v{s}per Tka\v{c}ik 
for helping us with the physics literature of complex networks. 
The authors thank Jonas Jalowy for pointing out a step in the proof that was not explained in sufficient detail.

\section{Main results}  \label{sec:main_results} 

Let $X=(x_{ij})_{i,j=1}^n\in \C^{n\times n}$ be a matrix with independent, centered entries.  
Let $\smallS\defeq(\E\abs{x_{ij}}^2)_{i,j=1}^n$ be the matrix collecting the variances of the entries of $X$. 
Furthermore, our main results will require a selection of the following assumptions 
(we remark that the last assumption \ref{assum:bounded_density} can be substantially relaxed, see Remark~\ref{rem:alternative_A3} below).

\paragraph{Assumptions}
\begin{enumerate}[label=(A\arabic*)]
\item \label{assum:flatness} The variance matrix $\smallS$ of $X$ is \emph{flat}\footnotemark, i.e.\ there are constants $s^* > s_* >0$  
such that 
\begin{equation}  \label{eq:condition_flatness} 
  \frac{s_*}{n} \leq \E \abs{x_{ij}}^2 \leq \frac{s^*}{n}  
\end{equation}
for all $i,j =1, \ldots, n$. 
\end{enumerate} 
\footnotetext{The flatness condition in \eqref{eq:condition_flatness} agrees with the concept of flatness introduced for general matrices with independent entries in \cite[Eq.~(2.1)]{Altcirc}. For Hermitian matrices, flatness is defined 
slightly differently, see (3.6) in \cite{Cusp1} and the explanation thereafter.}

\begin{enumerate}[label=(A\arabic*)]
\setcounter{enumi}{1}
\item \label{assum:bounded_moments} The entries of $X$ have \emph{bounded moments} in the sense that, for each $m \in \N$, there is $\mu_m>0$ such that 
\[ \E \abs{x_{ij}}^m  \leq \mu_m n^{-m/2} \] 
for all $i,j = 1, \ldots, n$. 
\item \label{assum:bounded_density} Each entry of $\sqrt{n}X$ has a bounded density on $\C$ in the following sense. There are probability densities $\nu_{ij} \colon \C \to [0,\infty)$ such that 
\[ \P \big( \sqrt{n}\, x_{ij} \in B \big) = \int_B \nu_{ij} (z)\, \di^2 z \] 
for all $i,j = 1, \ldots, n$ and all Borel sets $B \subset \C$ and these densities are bounded in the sense that there are $\alpha, \beta >0$ such that $\nu_{ij} \in L^{1+ \alpha}(\C)$ and 
\[ \norm{\nu_{ij}}_{1 + \alpha} \leq n^\beta \] 
for all $i,j = 1, \ldots, n$. 
\end{enumerate}

In \ref{assum:bounded_density} and in the following, $\di^2 z$ denotes the Lebesgue measure on $\C$. 
The main results remain valid if $X$ has all real entries, i.e.\ the density $\nu_{ij}$ of $\sqrt{n}\, x_{ij}$ in \ref{assum:bounded_density} is supported on $\R$ instead of $\C$ and we consider its $L^{1+\alpha}(\R)$-norm.  
In fact, the proofs are completely analogous. Hence, for simplicity, we only present the proofs in the complex case. 

The following theorem, our first main result, provides a convergence result for the spectral radius of the random matrix $X$. 
For any matrix $R \in \C^{n\times n}$, we write $\varrho(R)$ for its spectral radius, i.e.\ $\varrho(R) \defeq \max_{\lambda \in \spec(R)} \abs{\lambda}$. 

\begin{theorem}[Spectral radius of $X$] \label{thm:spectral_radius_X} 
Let $X$ satisfy \ref{assum:flatness} -- \ref{assum:bounded_density}. 
Then, for each (small) $\eps >0$ and (large) $D>0$, there is $C_{\eps,D} >0$ such that 
\begin{equation} \label{eq:spectral_radius}  
\P \Big( \absa{\varrho(X) - \sqrt{\varrho(\smallS)}} \geq n^{-1/2+\eps} \Big) \leq \frac{C_{\eps,D}}{n^{D}} 
\end{equation}
for all $n \in \N$. 

Here, the constant $C_{\eps, D}$ depends only on $s_*$, $s^*$ from \ref{assum:flatness}, the sequence $(\mu_m)_{m \in \N}$ from \ref{assum:bounded_moments} and 
$\alpha$, $\beta$ from \ref{assum:bounded_density} in addition to $\eps$ and $D$. 
\end{theorem}

\begin{remark}[Upper bound on the spectral radius of $X$ without \ref{assum:bounded_density}]  \label{rem:upper_bound_spectral_radius} 
Without Assumption~\ref{assum:bounded_density} our proof still implies the following upper bound on the spectral radius $\varrho(X)$ of $X$. That is, if \ref{assum:flatness} and \ref{assum:bounded_moments} are satisfied then for each $\eps>0$ 
and $D>0$, there is $C_{\eps,D}>0$ such that, for all $n \in \N$, we have 
\[ \P \Big( \varrho(X) \geq \sqrt{\varrho(\smallS)} + n^{-1/2 + \eps} \Big) \leq \frac{C_{\eps,D}}{n^{D}}. \] 
In particular, $X$ does not have any eigenvalue of modulus bigger than $\sqrt{\varrho(\smallS)} + n^{-1/2 + \eps}$ with very high probability.
The assumption \ref{assum:bounded_density} is only used to control the smallest singular value of $X- z$ when relating the eigenvalue density of $X$ and the one of the Hermitization $H_z$ of $X$ (see~\eqref{eq:def_H_z} below) in the proof of Theorem~\ref{thm:local_law_X} below. An eigenvalue of $X$ at $z$ can be excluded directly, without comparing the eigenvalue densities, if the kernel of $H_z$ is trivial. Therefore, \ref{assum:bounded_density} is not needed for the upper bound on $\varrho(X)$.  
\end{remark}

The next main result, Theorem~\ref{thm:local_law_X} below, shows that the eigenvalue density of $X$ is close to a deterministic density on all scales slightly above the typical eigenvalue spacing when $n$ is large. We now prepare the definition of this deterministic density. 
For each $\eta >0$ and $z \in \C$, we denote by $(v_1, v_2) \in (0,\infty)^n \times (0,\infty)^n$ the unique solution to the system of equations 
\begin{equation} \label{eq:dyson_vector_equation}
 \frac{1}{v_1} = \eta + \smallS v_2 + \frac{\abs{z}^2}{\eta + \smallS^t v_1}, \qquad \frac{1}{v_2} = \eta + \smallS^t v_1 + \frac{\abs{z}^2}{\eta + \smallS v_2}. 
\end{equation}
Here, any scalar is identified with the vector in $\C^n$ whose components agree all with the scalar. E.g. $\eta$ is identified with $(\eta, \ldots, \eta) \in \C^n$. 
Moreover, the ratio of two vectors in $\C^n$ is defined componentwise. 
The existence and uniqueness of $(v_1,v_2)$ has been derived in \cite[Lemma~2.2]{Altcirc} from abstract existence and uniqueness results in \cite{Helton01012007}. 

In the following, we consider $v_1=v_1(z,\eta)$ and $v_2=v_2(z,\eta)$ as functions of $\eta >0$ and $z \in \C$. 
In Proposition~\ref{pro:properties_sigma} below, we will show that there is a probability density $\sigma \colon \C \to [0,\infty)$ such that 
\begin{equation} \label{eq:sigma_equal_laplace_potential} 
 \sigma(z) = -\frac{1}{2\pi} \Delta_z \int_0^\infty \bigg( \avg{v_1(z,\eta)} - \frac{1}{1 + \eta} \bigg) \dd \eta, 
\end{equation}
where the equality and the Laplacian $\Delta_z$ on $\C$ are understood in the sense of distributions on $\C$. 
Moreover, $\avg{v_1}$ denotes the mean of the vector $v_1 \in \C^n$, i.e. $\avg{u} \defeq \frac{1}{n} \sum_{i=1}^n u_i$ 
for any $u = (u_i)_{i=1}^n \in \C^n$. 
In Lemma~\ref{lem:properties_L} below, we will show that the integral on the right-hand side of \eqref{eq:sigma_equal_laplace_potential} exists for each $z \in \C$. 
Proposition~\ref{pro:properties_sigma} also proves further properties of $\sigma$, in particular, that the 
support of $\sigma$ is a disk of radius $\sqrt{\varrho(\smallS)}$ around the origin. 

In order to analyze the eigenvalue density of $X$ on local scales, we consider shifted and rescaled test functions as follows. 
For any function $f \colon \C \to \C$, $z_0 \in \C$ and $a >0$, we define 
\[ f_{z_0,a} \colon \C \to \C, \qquad f_{z_0,a}(z) \defeq n^{2a} f(n^{a} (z-z_0)). \] 
The eigenvalues of $X$ are denoted by $\zeta_1, \ldots, \zeta_n$. Now we are ready to state our second main result.

\begin{theorem}[Local inhomogeneous circular law] \label{thm:local_law_X} 
Let $X$ satisfy \ref{assum:flatness} -- \ref{assum:bounded_density}. Let $a \in [0,1/2]$ and $\varphi >0$. 
Then, for every $\eps>0$ and $D>0$, there is $C_{\eps,D}>0$ such that 
\[ \P \bigg( \absbb{\frac{1}{n} \sum_{i =1}^n f_{z_0,a} ( \zeta_i) - \int_{\C} f_{z_0,a} (z) \sigma(z) \di^2 z} \geq \frac{\norm{\Delta f}_{L^1}}{n^{1- 2a - \eps}} \bigg) \leq \frac{C_{\eps,D}}{n^D} \] 
uniformly for all $n \in \N$, $z_0 \in \C$ satisfying $\abs{z_0} \leq \varphi$ and $f \in C_0^2(\C)$ satisfying $\supp f \subset \{ z \in \C \colon \abs{z} \leq \varphi\}$. 
The point $z_0$ and the function $f$ may depend on $n$. 
In addition to $\eps$ and $D$, the constant $C_{\eps,D}$ depends only on $s_*, s^*$ from \ref{assum:flatness}, 
$(\mu_m)_{m \in \N}$ from \ref{assum:bounded_moments}, $\alpha, \beta$ from \ref{assum:bounded_density} as well as $a$ and $\varphi$. 
\end{theorem} 

The bulk regime, $|z_0|< \sqrt{\varrho(\mathscr S)}$, in Theorem~\ref{thm:local_law_X}  has already been proven in \cite[Theorem 2.5]{Altcirc}.
Choosing $a=0$ and $z_0=0$ in Theorem~\ref{thm:local_law_X} amounts to the optimal $1/n$  speed of convergence in the 
inhomogeneous circular law.  

Finally, we state a corollary of our  result showing that all normalised eigenvectors $u=(u_i)_{i=1}^n\in \C^n$ of $X$ are completely  delocalized 
in the sense that $\max_{i=1}^{n} \abs{u_i}\leq n^{-1/2+\eps}$ with very high probability.  
Eigenvector delocalization  under somewhat different conditions and with very different methods has already  been established  in \cite{rudelson2015} with recent refinements in \cite{Rudelson2016,EigenvectorsLuh,EigenvectorsTikhomirov}. 

\begin{corollary}[Eigenvector delocalization]\label{thm:delocalization} 
Let $X$ satisfy \ref{assum:flatness} and \ref{assum:bounded_moments}. Then, for each $\eps>0$ and $D>0$, 
there is $C_{\eps,D}>0$ such that 
\[ \P \bigg( \exists\, u\neq 0 \, \colon X u = \zeta u \text{ for some }
\zeta \in \C \text{ and } \max_{i=1}^{n} \,\abs{u_i} \geq n^{-1/2 + \eps} \norm{u} \bigg) \leq \frac{C_{\eps,D}}{n^{D}} 
 \] 
for all $n \in \N$. Here, $\norm{u}$ denotes the Euclidean norm of $u$. 
\end{corollary}

\begin{remark}[Alternative to Assumption~\ref{assum:bounded_density}] \label{rem:alternative_A3} 
Theorem~\ref{thm:spectral_radius_X}, as well as Theorem~\ref{thm:local_law_X} (with an additional condition $\| \Delta f\|_{L^{2+\epsilon}}\le n^C\| \Delta f\|_{L^1}$,
with some large constant $C$, on the test function $f$)  
 hold if Assumption~\ref{assum:bounded_density} is replaced by the following anticoncentration condition. 
With the L\'evy concentration function 
\[ \cal{L}(Z,t) \defeq \sup_{u \in \R} \P \big( \abs{Z-u} < t) \] 
we require that $\max_{i,j} \cal{L}(\sqrt{n}\,x_{ij}, t) \leq b$ for some constants $t \geq 0$ and $b \in (0,1)$.
In our main proofs, we use \ref{assum:bounded_density} for pedagogical reasons and the necessary modifications will be explained in Remark~\ref{rem:alternative_changes_proof} at the end of Section~\ref{sec:Xlaw} below. 
\end{remark}

\subsection{Outline of the proof} \label{sec:outline}

In this subsection, we outline a few central ideas of the proofs of Theorem~\ref{thm:spectral_radius_X} and Theorem~\ref{thm:local_law_X}. 
The spectrum of the $n\times n$-matrix $X$ can conveniently be studied by analysing the kernel
 of the $2n \times 2n$ Hermitian matrices $H_z$ 
defined through
\begin{equation} \label{eq:def_H_z} 
 H_z \defeq \begin{pmatrix} 0 & X -z \\ (X-z)^* & 0 \end{pmatrix}
 \end{equation}
for $z \in \C$. In fact, $z$ is an eigenvalue of $X$ if and only if the kernel of $H_z$ is nontrivial. 

All spectral properties of a Hermitian matrix can be obtained from its resolvent. In fact, in many cases, the resolvent of a Hermitian random matrix becomes deterministic when its size tends to 
infinity and the limit is the solution to the associated Matrix Dyson equation. 
In our setup, the Matrix Dyson equation (MDE) for the deterministic counterpart $M=M(z,\eta)$ of the resolvent $G = G(z,\eta) = (H_z - \ii\eta)^{-1}$ of $H_z$ is given by 
\begin{equation} \label{eq:mde}
 - M^{-1}(z, \eta) = \begin{pmatrix} \ii \eta & z \\ \bar z & \ii \eta \end{pmatrix} + \cS[M(z,\eta)]. 
\end{equation}
Here, $\eta>0$ and $z \in \C$ are parameters and $\ii\eta$, $z$ and $\bar z$ are identified with the respective multiples of the $n\times n$ identity matrix. 
Moreover, we introduced the self-energy operator $\cS \colon \C^{2n\times 2n} \to \C^{2n\times 2n}$ given by
\begin{equation} \label{eq:def_cS} 
\cS[R] = \begin{pmatrix} \smallS r_2 & 0 \\ 0 & \smallS^t r_1 \end{pmatrix} 
\end{equation} 
for $R=(r_{ij})_{i,j=1}^{2n} \in \C^{2n \times 2n}$, where $r_1 \defeq (r_{ii})_{i=1}^n$, $r_2 \defeq (r_{ii})_{i=n+1}^{2n}$ 
and $\smallS\defeq(\E \abs{x_{ij}}^2)_{i,j=1}^n$. 
The matrix on the right-hand side of \eqref{eq:def_cS} denotes a $2n\times 2n$ diagonal matrix with the vector $(\smallS r_2, \smallS^t r_1) \in \C^{2n}$ on its diagonal. 

Two remarks about \eqref{eq:mde} and \eqref{eq:def_cS} are in order. 
In this paper we are interested exclusively in the kernel of $H_z$. Otherwise $\ii \eta$ on the right-hand side of \eqref{eq:mde} had to be replaced by $E+ \ii \eta$ for some $E\in \R$
 (see \cite{AjankiCorrelated,Erdos2017Correlated,AltEdge} for the general MDE in the random matrix setup). 
We also remark that the self-energy operator $\cS$ in \eqref{eq:def_cS} is chosen slightly differently compared to the choice of the self-energy operator for a Hermitian random matrix in 
\cite{AjankiCorrelated,Erdos2017Correlated,AltEdge}. Instead, we follow here  the convention from \cite{Altcirc}. 
For further details, see Remark~\ref{rem:convention_for_S}  below. 

First, we discuss Theorem~\ref{thm:spectral_radius_X}. 
Suppose we already know that $G$ is very well approximated by $M$. 
Owing to \cite[Proposition~3.2]{Altcirc} (see also Lemma~\ref{lem:scaling_relation_v_u} below), $\Im M(z,\eta)$ vanishes sufficiently fast for $\eta \downarrow 0$ as long as $\abs{z}^2 \geq \varrho(\smallS) + n^{-1/2+\eps}$. 
Then we can immediately conclude that the kernel of $H_z$ has to be trivial. 
Hence, any eigenvalue of $X$ has modulus smaller than $\sqrt{\varrho(\smallS)} + n^{-1/2 + \eps}$. 
Similarly, under the condition $\abs{z}^2 < \varrho(\smallS) - n^{-1/2 + \eps}$, the imaginary part $\Im M(z,\eta)$ is big enough as $\eta \downarrow 0$ due to \cite[Proposition~3.2]{Altcirc}. 
 This will imply that $H_z$ has a nontrivial kernel and, hence, $X$ has an eigenvalue close to $z$, thus completing the proof of~\eqref{eq:spectral_radius}.

Therefore, what remains is to prove a local law for $H_z$, i.e.  that $G$ is very well approximated by $M$. 
The resolvent $G$  satisfies a perturbed version of the MDE \eqref{eq:mde},
\begin{equation} \label{eq:perturbed_mde} 
 -G^{-1} = \begin{pmatrix} \ii \eta & z \\ \bar z & \ii \eta \end{pmatrix} + \cS[G] - DG^{-1},  \qquad \qquad D \defeq (H_z - \E H_z) G + \cS[G] G
\end{equation}
for all $\eta >0$ and $z \in \C$.  
The error matrix $D$ will be shown to be small in Section~\ref{sec:local_law_H} below. Consequently, we will consider \eqref{eq:perturbed_mde} as a perturbed version of the MDE, \eqref{eq:mde} and study its stability properties under small perturbations to conclude that $G$ is close to $M$. 

A simple computation starting from \eqref{eq:mde} and \eqref{eq:perturbed_mde} yields the stability equation associated to the MDE, 
\begin{equation} \label{eq:stability_equation}
\cB[G-M] =  M \cS[G-M](G-M) - MD.  
\end{equation}
Here, $\cB\colon \C^{2n\times 2n} \to \C^{2n\times 2n}$ is the \emph{linear stability operator} of the MDE, 
given explicitly by 
\begin{equation} \label{eq:def_cB} 
\cB[R] \defeq R- M\cS[R]M
\end{equation}
for any $R \in \C^{2n\times 2n}$. 

The stability equation \eqref{eq:stability_equation} is viewed as  a general quadratic equation of the
form
\begin{equation}\label{AB}
\cB[Y] - \cA[Y, Y] + {Z}=0
\end{equation}
for the unknown matrix $Y (= G-M)$ in the regime where ${Z} (=MD)$ is small. 
Here, $\cB$ is a linear map and $\cA$ is a bilinear map on the space of matrices. 
This problem would be easily solved by a standard implicit function theorem if $\cB$ had a stable  (i.e. bounded) inverse;
this is the case in the bulk regime. 
When $\cB$ has unstable directions, i.e. eigenvectors corresponding to eigenvalues very close to zero, 
then these directions need to be handled separately. 

The linear stability operator \eqref{eq:def_cB} for 
Wigner-type matrices with  a flat variance matrix
in the edge or cusp regime gives rise to one unstable direction $B$ with $\cB[B]\approx 0$.
In this case, the solution is, to leading order, parallel to the unstable direction $B$, hence it can be written as
$Y = \Theta B + \mathrm{error}$ with some complex scalar coefficient $\Theta$, determining the leading behavior of $Y$.
For such $Y$ the linear term in \eqref{AB} becomes lower order and the quadratic term as well as the 
error term in $Y$ play an important role. Systematically expanding $Y$ 
up to higher orders in the small parameter $\| { Z}\|\ll 1$, we arrive at an approximate cubic equation for $\Theta$
of the form $c_3 \Theta^3 + c_2 \Theta^2 + c_1 \Theta = \mathrm{small}$, with very precisely  computed coefficients. The
 full derivation of this cubic equation is given in \cite[Lemma A.1]{Cusp1}. In the bulk regime $|c_1|\sim 1$, hence
 the equation is practically linear.  In the regime where the density vanishes, we have $c_1\approx 0$, hence
 higher order terms become relevant. At the edge we have $|c_2|\sim 1$, so we have a quadratic equation, while
 in the cusp regime $c_2\approx 0$, but $|c_3|\sim 1$, so we have a cubic equation. It turns out 
 that under  the flatness condition no other cases are possible, i.e. $|c_1|+ |c_2|+|c_3| \sim 1$. 
 This trichotomic structural property of the underlying cubic equation was first discovered in \cite{AjankiSingularities}, developed further in \cite{Altshape},
and played an essential role in proving cusp local laws
for Wigner-type matrices in \cite{Ajankirandommatrix, Cusp1}. 
 
In our current situation, lacking flatness for $H_z$, 
a second unstable direction of $\cB$ is present due to the specific block structure of the matrix $H_z$
which creates a major complication. 
We denote the unstable directions of $\cB$ by $B$ and $B_*$. One of them, $B$, is the relevant one and it behaves very similarly to the one present in \cite{AltEdge,Altshape,Cusp1}. The novel unstable direction $B_*$ 
originates from the specific block structure of $H_z$ and $\cS$
 in \eqref{eq:def_H_z} and \eqref{eq:def_cS}, respectively, and is related to the unstable direction in \cite{Altcirc}. 
We need to treat both unstable directions separately.
In a generic situation, the solution to \eqref{AB}
would be of the form $Y = \Theta B + \Theta_* B_* + \mathrm{error}$, where
the complex scalars $\Theta$ and $\Theta_*$ satisfy a system of coupled cubic equations that is hard to analyse. 
Fortunately, for our applications, we have an additional input, namely we know that there is a matrix, 
concretely $E_-$, such that $Y=G-M$ is orthogonal to $E_-$, while $B_*$ is far from being orthogonal to $E_-$
(see \eqref{eq:def_E_pm} below for the definition of $E_-$ and \eqref{eq:scalar_E_minus_G_equals_0} and \eqref{eq:scalar_E_minus_M_equals_0} for the orthogonality to $G$ and $M$, respectively). 
The existence of such a matrix and a certain non-degeneracy of the two unstable directions guarantee that $\Theta_*$ is negligible and $Y$ is still essentially parallel to one unstable direction, $Y = \Theta B + \mathrm{error}$. Hence, we 
still need to analyse a single cubic equation for $\Theta$, 
albeit its coefficients, given in terms of $B$, $B_*$, $M$, $D$, $\cB$ and $\cS$,
see Lemma~\ref{lem:cubic_equation_abstract} for their precise form, 
 are much more complicated than those in \cite{Altshape,Cusp1}.

Summarizing, to understand the relationship between $G-M$ and $D$  from \eqref{eq:stability_equation}
requires an analysis of the small eigenvalues of $\cB$ in the regime, where $\abs{z}^2$ is close to $\varrho(\smallS)$ and $\eta$ is small. 
This analysis is based on viewing the non-normal operator $\cB$ as a perturbation around an operator of the form $1- \cC \cF$, where $\cC$ is unitary and $\cF$ is Hermitian. 
The unperturbed operator, $1- \cC\cF$, is also non-normal but simpler to analyze compared to $\cB$. In fact, $1- \cC \cF$ has a single small eigenvalue and this eigenvalue has (algebraic and geometric) multiplicity two
and we can construct appropriate eigendirections. 
A very fine perturbative argument reveals that after perturbation these two eigendirections will be associated to two different (small) eigenvalues $\beta$ and $\beta_*$. 
The distance between them is controlled from below which allows us to follow the perturbation of the eigendirections as well. 
Precise perturbative expansions of $B$ and $B_*$ around the corresponding eigenvectors of $1 - \cC \cF$ and a careful use of the specific structure of $\cS$ in \eqref{eq:def_cS} 
reveal that, up to a small error term, $B$ is orthogonal to $E_-$ while $B_*$ is far from orthogonal to $E_-$.

Moreover, we have to show that $M D$ in \eqref{eq:stability_equation} is sufficiently small in the unstable direction $B$ to compensate for the blow-up of $\cB^{-1}$ originating from the relevant small eigenvalue $\beta$. 
To that end, we need to adjust the \emph{cusp fluctuation averaging} mechanism discovered in \cite{Cusp1} to the current setup which will be done in Subsection~\ref{subsec:proof_D_bounds} below. 
This part also uses the specific block structure of $H_z$ in \eqref{eq:def_H_z}. 
We can, thus, conclude that $G-M$ is small due to \eqref{eq:stability_equation} which completes the sketch of the proof of Theorem~\ref{thm:spectral_radius_X}.

The proof of Theorem~\ref{thm:local_law_X} also follows from the local law for $H_z$
 since the observable of the eigenvalues of $X$ is related to 
the resolvent $G$ while the integral over $f_{z_0,a} \sigma$ is related to $M$. 
Indeed, \cite[Eq.'s~(2.10), (2.13) and (2.14)]{Altcirc} imply that 
\begin{equation} \label{eq:girko} 
\frac{1}{n} \sum_{i=1}^n f_{z_0,a}(\zeta_i) = \frac{1}{4\pi n} \int_{\C} \Delta f_{z_0,a}(z) \log \abs{\det H_z} \di^2 z =
 -\frac{1}{4\pi n} \int_{\C} \Delta f_{z_0,a}(z) \int_0^\infty \Im \Tr G(z,\eta) \di \eta \, \di^2 z. 
\end{equation}
The first identity in \eqref{eq:girko} is known as \emph{Girko's Hermitization formula}, the second identity 
 (after a regularization of the $\eta$-integral at infinity)  was first used in \cite{tao2015}. 
On the other hand, since the imaginary part of the diagonal of $M$ coincides with the solution $(v_1,v_2)$ of \eqref{eq:dyson_vector_equation} (see \eqref{eq:block_structure_M} below), 
the definition of $\sigma$ in \eqref{eq:sigma_equal_laplace_potential} yields 
\[ \int_\C f_{z_0,a} (z) \sigma(z)\di^2 z = -\frac{1}{4\pi n} \int_\C \Delta f_{z_0,a}(z) \int_0^\infty \Im \Tr M(z,\eta) \di \eta\,  \di^2 z. \] 
Therefore, Theorem~\ref{thm:local_law_X} also follows once the closeness of $G$ and $M$ has been established as explained above.

\subsection{Notations and conventions}  \label{sec:notations} 

In this section, we collect some notations and conventions used throughout the paper. 
We set $[k] \defeq \{1, \ldots, k \} \subset \N$ for any $k \in \N$. 
For $z \in \C$ and $r >0$, we define the disk $D_r(z)$ in $\C$ of radius $r$ centered at $z$ through $D_r(z) \defeq\{ w \in \C \colon \abs{z-w} < r \}$.  
We use $\di^2 z$ to denote integration with respect to the Lebesgue measure on $\C$. 

We now introduce some notation used for vectors, matrices and linear maps on matrices. 
Vectors in $\C^{2n}$ are denoted by boldfaced small Latin letters like $\bx$, $\by$ etc. 
For vectors $\bx = (x_a)_{a \in [2n]}, \,\by = (y_a)_{a \in [2n]} \in \C^{2n}$, 
we consider the normalized Euclidean scalar product $\scalar{\bx}{\by}$ and the induced
normalized Euclidean norm $\norm{\bx}$ defined by 
\[\scalar{\bx}{\by} = (2n)^{-1} \sum_{a \in [2n]} \overline{x_a} y_a, \qquad \qquad \norm{\bx} \defeq
\scalar{\bx}{\bx}^{1/2}. \] 
Functions of vectors such as roots, powers or inverse and operations such as products of vectors are understood entrywise. 

Matrices in $\C^{2n \times 2n}$ are usually denoted by capitalized Latin letters. We especially use $G$, $H$, $J$, $M$, $R$, $S$ and $T$. 
For a matrix $R\in \C^{2n \times 2n}$, we introduce the real part $\Re R$ and the imaginary part $\Im R$ defined 
through 
\[ \Re R \defeq \frac{1}{2} \big( R + R^*\big), \qquad \qquad \Im R\defeq \frac{1}{2\ii} \big( R - R^* \big). \] 
We have $R = \Re R + \ii \Im R$ for all $R \in \C^{2n \times 2n}$.  
On $\C^{2n\times 2n}$, we consider the normalized trace $\avg{\genarg}$ and the normalized Hilbert-Schmidt scalar product $\scalar{\genarg}{\genarg}$ defined by 
\[ \avg{V} \defeq \frac{1}{2n} \Tr(V) = \frac{1}{2n} \sum_{i=1}^{2n} v_{ii}, \qquad \scalar{V}{W} \defeq \frac{1}{2n} \Tr (V^* W) = \avg{V^*W} \] 
for all $V=(v_{ij})_{i,j=1}^{2n}, W \in \C^{2n \times 2n}$. The norm on $\C^{2n \times 2n}$ induced by the normalized Hilbert-Schmidt scalar product is denoted by $\normtwo{\genarg}$, i.e. $\normtwo{V} \defeq \avg{V^*V}^{1/2}$ for any $V \in \C^{2n \times 2n}$. 
Moreover, for $V \in \C^{2n \times 2n}$, we write $\norm{V}$ for the operator norm of $V$ induced by the normalized Euclidean norm $\norm{\genarg}$ on $\C^{2n}$.

We use capitalized calligraphic letters like $\cS$, $\cB$ and $\cT$ to denote linear maps on $\C^{2n\times 2n}$. 
In particular, for $A, B \in \C^{2n \times 2n}$, we define the linear map $\cC_{A,B} \colon \C^{2n\times 2n} \to \C^{2n\times 2n}$ through $\cC_{A, B} [R] \defeq A R B$ for all $R \in \C^{2n\times 2n}$. 
This map satisfies the identities $\cC_{A,B}^* = \cC_{A^*, B^*}$ and $\cC_{A,B}^{-1} = \cC_{A^{-1},B^{-1}}$, 
where the second identity requires the matrices $A$ and $B$ to be invertible. 
 We set $\cC_{A} \defeq \cC_{A,A}$ for any matrix $A \in \C^{2n\times 2n}$.  
For a linear map $\cT$ on $\C^{2n\times 2n}$, we consider several norms. 
We denote by $\norm{\cT}$ the operator norm of $\cT$ induced by $\norm{\genarg}$ on $\C^{2n\times 2n}$. 
Moreover, $\normtwoop{\cT}$ denotes the operator norm of $\cT$ induced by $\normtwo{\genarg}$ on $\C^{2n\times 2n}$. 
We write $\normtwoinf{\cT}$ for the operator norm of $\cT$ when the domain is equipped with $\normtwo{\genarg}$ and the target is equipped with $\norm{\genarg}$.  

In order to simplify the notation in numerous computations, we use the following conventions. 
In vector-valued relations, we identify a scalar with the vector whose components all agree with this scalar. 
Moreover, we use the block matrix notation 
\begin{equation} \label{eq:block_matrix_notation} 
 \begin{pmatrix} a & b \\ c & d \end{pmatrix} 
\end{equation}
exclusively for $2n\times 2n$-matrices. Here, each block is of size $n \times n$. If $a$, $b$, $c$ or $d$ are vectors (or scalars) then with a slight abuse of notations they are identified with the diagonal $n\times n$ matrices with $a$, $b$, $c$ or $d$, respectively, on the diagonal 
(or the respective multiple of the $n\times n$ identity matrix). 
Furthermore, we introduce the $2n\times 2n$ matrices $E_+$ and $E_-$ given in the block matrix notation of 
\eqref{eq:block_matrix_notation} by 
\begin{equation} \label{eq:def_E_pm} 
 E_+ \defeq \begin{pmatrix} 1 & 0 \\ 0 & 1 \end{pmatrix}, \qquad E_- \defeq \begin{pmatrix} 1 & 0 \\ 0 & -1 \end{pmatrix}. 
\end{equation}
We remark that $E_+$ coincides with the identity matrix in $\C^{2n\times 2n}$. 
In our argument, the following sets of $2n \times 2n$-matrices appear frequently. The \emph{diagonal} matrices $\dM\subset \C^{2n \times 2n}$ and the \emph{off-diagonal} matrices $\oM\subset \C^{2n \times 2n}$ are defined through 
\[ \dM \defeq \bigg\{ \begin{pmatrix} a & 0 \\ 0 & b \end{pmatrix} \colon a,b \in \C^n \bigg\}, 
\qquad \qquad \oM \defeq \bigg\{ \begin{pmatrix} 0 & a  \\  b & 0  \end{pmatrix} \colon a,b \in \C^n \bigg\}.  \] 
 Note that the subspaces $\dM$ and $\oM$ are orthogonal with respect to the normalized Hilbert-Schmidt scalar product defined above.  

In each section of this paper, we will specify a set of \emph{model parameters} which are basic parameters
of our model, e.g. $s_*$ and $s^*$ in \eqref{eq:condition_flatness}. All of our estimates will hold uniformly for all models that satisfy our assumptions with the same model parameters. 
For $f, g \in [0,\infty)$, the comparison relation $f \lesssim g$ is true if $f \leq C g$ for some constant $C>0$ that depends only on model parameters. 
We also write $f \gtrsim g$ if $g \lesssim f$ and $f \sim g$ if $f \lesssim g$ and $f \gtrsim g$. 
If $f(i)$ and $g(i)$ depend on a further parameter $i \in I$ and $f(i) \leq C g(i)$ for all $i \in I$ then we say $f \lesssim g$ uniformly for $i \in I$. 
We use the same notation for nonnegative vectors and positive semidefinite matrices. Here, for vectors 
$\bx = (x_a)_{a \in [2n]}, \,\by =(y_a)_{a \in [2n]}\in [0,\infty)^{2n}$, the comparison relation $\bx \lesssim \by$ means $x_a \lesssim y_a$
uniformly for all $a \in [2n]$, i.e. the implicit constant can be chosen independently of $a$. 
For positive semidefinite matrices $R_1, R_2 \in \C^{2n\times 2n}$, $R_1 \lesssim R_2$ if $\scalar{\bx}{R_1 \bx} \lesssim \scalar{\bx}{R_2 \bx}$ uniformly for all $\bx \in \C^{2n}$. 
For $\eps >0$, scalars $f_1, f_2 \in \C$, matrices $R_1, R_2 \in \C^{2n\times 2n}$ and operators $\cT_1, \cT_2$ on $\C^{2n\times 2n}$, we write $f_1 = f_2 +\ord(\eps)$, $R_1= R_2 + \ord(\eps)$ and $\cT_1 = \cT_2 +\ord(\eps)$ if $\abs{f_1 - f_2} \lesssim \eps$, $\norm{R_1 - R_2} \lesssim \eps$ and $\norm{\cT_1 - \cT_2} \lesssim \eps$, respectively.

\section{Analysis of the Matrix Dyson equation}  \label{sec:stability}

In this section, we study the linear stability of the MDE, \eqref{eq:mde}. According to the quadratic stability equation, \eqref{eq:stability_equation}, associated to the MDE
the linear stability is governed by the behaviour of the stability operator $\cB \defeq 1 - \cC_M \cS$ (compare \eqref{eq:def_cB}). The main result of this section, Proposition~\ref{pro:stability_operator} below, 
provides a complete understanding of the small, in absolute value, eigenvalues of $\cB$ in the regime when $\rho= \rho(z,\eta)$ is small. 
Here, $\rho = \rho(z,\eta)$ is defined through 
\begin{equation} \label{eq:def_rho} 
\rho \defeq \frac{1}{\pi} \avg{\Im M} 
\end{equation}
for $\eta >0$ and $z \in \C$, where $M$ is the solution to \eqref{eq:mde}.  
For the small eigenvalues and their associated eigenvectors, very precise expansions in terms of $M$ are derived in Proposition~\ref{pro:stability_operator}.

We warn the reader that $\rho$ should not be confused with the spectral radii $\varrho(X)$ and $\varrho(\smallS)$ used in Section~\ref{sec:main_results}. 
The function $\rho$ is the \emph{harmonic extension} of the \emph{self-consistent density of states} of $H_z$  
(see e.g. \cite[Eq.~(2)]{AltEdge} for the definition of the self-consistent density of states).  

In the remainder of the present section, we assume $\eta \in (0,1]$ and $ z \in D_\tau(0)$ for some fixed $\tau > 1$. In this section, the comparison relation $\lesssim$ introduced in 
Section~\ref{sec:notations} is understood with respect to the model parameters $\{s_*, s^*, \tau \}$. We recall that $s_*$ and $s^*$ constituted the bounds on the entries of $\smallS$ 
in \eqref{eq:condition_flatness}. 

The following proposition is the main result of the present section. 

\begin{proposition}[Properties of the stability operator $\cB$]  \label{pro:stability_operator} 
There are (small) $\rho_* \sim 1$ and $\eps \sim 1$ such that if $\rho + \eta/\rho \leq \rho_*$ then 
$\cB$ has two eigenvalues $\beta$ and $\beta_*$ in $D_\eps(0)$, i.e.  $\spec (\cB) \cap D_\eps(0) = \{ \beta, \beta_*\}$. 
Moreover, $\beta$ and $\beta_*$ have geometric and algebraic multiplicity one, $0<\abs{\beta_*} < \abs{\beta}$ and 
\begin{equation} \label{eq:beta_beta_star_scaling} 
 \abs{\beta_*} \sim \eta/\rho, \qquad \abs{\beta} \sim \eta/\rho + \rho^2. 
\end{equation}
Furthermore, $\cB$ has left and right eigenvectors $\wh{B}_*$, $\wh{B}$ and $B_*$, $B$ corresponding to $\beta_*$
and $\beta$, respectively, i.e., 
\[ \cB[B_*] = \beta_* B_*, \qquad \cB[B] = \beta B, \qquad \cB^*[\wh{B}_*] =\overline{\beta_*} \wh{B}_*, \qquad \cB^*[\wh{B}] = \bar{\beta} \wh{B}, \] 
which satisfy 
\begin{subequations} \label{eq:expansions_eigenvectors_cB} 
\begin{align}
 B & = \rho^{-1} \Im M - 2\ii \rho^{-1}(\Im M)(\Im M^{-1}) (\Re M) + \ord(\rho^2 + \eta/\rho),  \label{eq:expansion_B}\\ 
 B_* & = \rho^{-1} E_-\Im M + \ord(\rho^2 + \eta/\rho),\label{eq:expansion_B_star} \\
 \wh{B}  & = - \rho^{-1} \Im (M^{-1}) + \ord(\rho^2 + \eta/\rho),\label{eq:expansion_L} \\ 
 \wh{B}_* & = - \rho^{-1} E_-\Im (M^{-1}) + \ord(\rho^2 + \eta/\rho). \label{eq:expansion_L_star}
\end{align} 
\end{subequations} 
For fixed $z$, the eigenvalues $\beta$ and $\beta_*$ as well as the eigenvectors $B$, $B_*$, $\wh{B}$ and $\wh{B}_*$ are continuous functions of $\eta$ as long as $\rho + \eta/\rho \leq \rho_*$. 
We also have the expansions 
\begin{subequations} 
\begin{align} 
\beta \scalar{\wh{B}}{B} & =  \pi \eta \rho^{-1} + 2 \rho^2 \psi + \ord(\rho^3 + \eta\rho + \eta^2/\rho^2), \label{eq:expansion_beta}\\ 
 \beta_* \scalar{\wh{B}_*}{B_*} & = \pi\eta \rho^{-1} + \ord(\rho^3 + \eta\rho + \eta^2/\rho^2), \label{eq:expansion_beta_star}
\end{align} 
\end{subequations} 
where $\psi \defeq \rho^{-4} \avg{[(\Im M)(\Im M^{-1})]^2}$. We have $\psi \sim 1$, $\abs{\scalar{\wh{B}}{B}}\sim 1$ and $\abs{\scalar{\wh{B}_*}{B_*}} \sim 1$. 

Moreover, the resolvent of $\cB$ is bounded on the spectral subspace complementary to $\beta$ and $\beta_*$. That is, if $\cQ$ is the spectral projection of $\cB$ associated to $\spec(\cB)\setminus \{ \beta, \beta_*\}$
then 
\begin{equation} \label{eq:cB_inverse_Q_lesssim_one} 
 \norm{\cB^{-1} \cQ} + \norm{(\cB^*)^{-1} \cQ^*} \lesssim 1.  
\end{equation}
\end{proposition} 

We now make a few remarks about Proposition~\ref{pro:stability_operator}. 
First, owing to Lemma~\ref{lem:scaling_relation_v_u} below (also note \eqref{eq:normalization_spectral_radius} below), the condition $\rho + \eta/\rho \leq \rho_*$ with $\rho_* \sim 1$ is satisfied if $\abs{\abs{z}^2 - \varrho(\smallS)} \leq \delta$ and $\eta \in (0,\delta]$ 
for some (small) $\delta \sim 1$. 
Secondly, we note that $B$, $B_*$, etc. are called eigenvectors despite that they are in fact matrices in $\C^{2n\times 2n}$. 
Finally, the second term on the right-hand side of \eqref{eq:expansion_B} is of order $\rho$, hence it is subleading compared to the first term $\rho^{-1} \Im M \sim 1$. 

We now explain the relation between the solution $M$ to the MDE, \eqref{eq:mde}, and the solution $(v_1, v_2)$ to \eqref{eq:dyson_vector_equation}. 
The $2n \times 2n$ matrix $M$ satisfies 
\begin{equation} \label{eq:block_structure_M} 
 M(z, \eta) = \begin{pmatrix} \ii v_1 & -z u \\ -\bar z u & \ii v_2 \end{pmatrix},
\end{equation}
where $(v_1, v_2)$ is the unique solution of \eqref{eq:dyson_vector_equation} 
and $u$ is defined through 
\[ u \defeq \frac{v_1}{\eta + \smallS^t v_1} = \frac{v_2}{ \eta + \smallS v_2}. \] 
Note that $u \in (0,\infty)^n$. We remark that \eqref{eq:block_structure_M} is the unique solution to \eqref{eq:mde} with the side condition that $\Im M$ is a positive definite matrix. 
The existence and uniqueness of such $M$ follows from \cite{Helton01012007}. 

Throughout this section, the special structure of $M$ as presented in \eqref{eq:block_structure_M} will play an important role. 
As a first instance, we see that the representation of $M$ in \eqref{eq:block_structure_M} implies 
\begin{equation} \label{eq:Im_M_Re_M} 
 \Im M = \begin{pmatrix} v_1 & 0 \\ 0 & v_2\end{pmatrix} , \qquad \Re M = \begin{pmatrix} 0 & - z u \\ -\bar z u & 0 \end{pmatrix}. 
\end{equation}
Therefore, $\Im M \in \dM$ and $\Re M \in \oM$. This is an important ingredient in the proof of the following corollary. 

\begin{corollary} 
There is $\rho_* \sim 1$ such that $\rho + \eta/\rho \leq \rho_*$ implies 
\begin{equation} \label{eq:bound_scalar_E_minus_B} 
 \abs{\scalar{E_-}{B}}  \lesssim \rho^2 + \eta/\rho, 
\end{equation}
where $B$ is the right eigenvector of $\cB$ from Proposition~\ref{pro:stability_operator}. 
\end{corollary} 

\begin{proof} 
The expansion of $B$ in \eqref{eq:expansion_B} yields 
\[  \scalar{E_-}{B} = \rho^{-1} \scalar{E_-}{\Im M} - 2 \ii \rho^{-1} \scalar{E_-}{(\Im M)(\Im M^{-1})(\Re M)} + \ord(\rho^2 + \eta/\rho). \] 
We now conclude \eqref{eq:bound_scalar_E_minus_B} by showing that the first two terms on the right-hand side vanish. 
The identity \eqref{eq:scalar_E_minus_M_equals_0} below implies $\scalar{E_-}{\Im M} = 0$. 
Moreover, by \eqref{eq:Im_M_Re_M}, we have $\Re M \in \oM$ and $\Im M \in \dM$. Taking the imaginary part of \eqref{eq:mde} thus yields $\Im M^{-1} \in \dM$. 
Therefore, $\scalar{E_-}{(\Im M)(\Im M^{-1})(\Re M)}=0$ since $\Re M \in \oM$ while $E_- (\Im M)(\Im M^{-1}) \in \dM$. 
This completes the proof of \eqref{eq:bound_scalar_E_minus_B}. 
\end{proof} 

\subsection{Preliminaries} 

The MDE, \eqref{eq:mde}, and its solution have a special scaling when $\cS$ and, hence, $\smallS$, are rescaled by $\lambda>0$, i.e. $\cS$ in \eqref{eq:mde} is 
replaced by $\lambda \cS$. Indeed, if $M=M(z,\eta)$ is the solution to \eqref{eq:mde} with positive definite imaginary part then $M_\lambda(z,\eta) \defeq \lambda^{-1/2} M(z\lambda^{-1/2}, \eta \lambda^{-1/2})$ is the 
 solution to 
\[ - M_\lambda^{-1} = \begin{pmatrix} \ii \eta & z \\ \bar z & \ii \eta \end{pmatrix} + \lambda\cS[M_\lambda]. \] 
with positive imaginary part. 
The same rescaling yields the positive solution of \eqref{eq:dyson_vector_equation} when $\smallS$ is replaced by $\lambda \smallS$ (see the explanations around (3.7) in \cite{Altcirc}). 
Therefore, by a simple rescaling, we can assume that the spectral radius is one, 
\begin{equation} \label{eq:normalization_spectral_radius} 
\varrho(\smallS) = 1. 
\end{equation} 
 We remark that the other assumptions on $X$ are still satisfied since the flatness condition \ref{assum:flatness} directly implies $\varrho(\smallS) \sim 1$.  
In the remainder of the paper, we will always assume \eqref{eq:normalization_spectral_radius}. 

\paragraph{Balanced polar decomposition of $M$} 
We first introduce a polar decomposition of $M$ that will yield a useful factorization of $\cB$ which is the basis of its spectral analysis. 
To that end, we define 
\begin{equation} \label{eq:def_U_Q} 
 U \defeq \begin{pmatrix} \ii \sqrt{\frac{v_1 v_2}{u}} & - z \sqrt{u} \\ - \bar z \sqrt{u} & \ii\sqrt{\frac{v_1 v_2}{u}} \end{pmatrix}, \qquad 
Q \defeq \begin{pmatrix} \big(\frac{uv_1}{v_2}\big)^{1/4} & 0 \\ 0 & \big(\frac{uv_2}{v_1}\big)^{1/4} \end{pmatrix}, 
\end{equation}
where roots and powers of vectors are taken entrywise. 
Starting from these definitions, an easy computation shows that $M$ admits the following \emph{balanced polar decomposition} 
\begin{equation} \label{eq:balanced_polar_decomposition} 
 M = Q U Q. 
\end{equation}
Such polar decomposition for the solution of the Dyson equation was introduced in \cite{AjankiCorrelated}.  

The following lemma collects a few basic properties of $M$, $\rho$, $U$ and $Q$, mostly borrowed from \cite{Altcirc}.  

\begin{lemma}[Basic properties of $M$, $\rho$, $U$ and $Q$] \phantomsection \label{lem:scaling_relation_v_u} 
\begin{enumerate}[label=(\roman*)]
\item Let $\eta \in (0,1]$ and $z \in D_\tau(0)$. We have 
\begin{equation} \label{eq:scalar_E_minus_M_equals_0} 
 \scalar{E_-}{M} = 0. 
\end{equation}
Moreover, $Q = Q^* \in \dM$,  $\Im U \in \dM$, $\Re U \in \oM$ and $U$ is unitary. 
\item Uniformly for all $\eta\in (0,1]$ and $z \in D_\tau(0)$, $\rho$ satisfies the scaling relations 
\begin{equation} \label{eq:scaling_rho}  
\rho \sim \begin{cases} \eta^{1/3} + (1- \abs{z}^2)^{1/2}, & \text{ if } \abs{z} \leq 1, \\
\frac{\eta}{\abs{z}^2 - 1 + \eta^{2/3}}, & \text{ if } 1 \leq \abs{z} \leq \tau, \end{cases}
\end{equation} 
for the matrices $U$, $Q$ and $M$, we have the estimates 
\begin{equation}\label{eq:scaling_U_Q_M} 
 \Im U \sim \rho , \qquad Q \sim 1, \qquad \norm{M} \lesssim 1, 
\end{equation} 
and for the entries of $M$, we have 
\begin{equation}  \label{eq:scaling_v_1_v_2_u} 
v_1 \sim v_2 \sim \rho, \qquad u \sim 1.
\end{equation}
\item For fixed $z\in \C$, $M$, $\rho$, $U$ and $Q$ are continuous functions of $\eta$. 
\end{enumerate}
\end{lemma}

\begin{proof} 
First, we remark that $\abs{z}^2$ was denoted by $\tau$ in \cite{Altcirc}. 

The identity in \eqref{eq:scalar_E_minus_M_equals_0} follows from \cite[Eq.~(3.8)]{Altcirc}. 
Obviously, \eqref{eq:def_U_Q} and $v_1, v_2, u >0$ yield $Q = Q^* \in \dM$, $\Im U \in \dM$ and $\Re U \in \oM$. 
A simple computation reveals that $U$ is unitary as $u = v_1 v_2 + \abs{z}^2 u^2$ due to \cite[Eq.~(3.32)]{Altcirc}. 

From \cite[Eq.~(3.10),~(3.11)]{Altcirc}, we conclude that $v_1$, $v_2$ and $\Im M$ scale as the right-hand side of \eqref{eq:scaling_rho}. Hence, \eqref{eq:scaling_rho}
follows from the definition of $\rho$ in \eqref{eq:def_rho}. Consequently, $v_1 \sim \rho \sim v_2$. 
Owing to \cite[Eq.~(3.26)]{Altcirc}, we have $u \sim 1$ uniformly for all $z \in D_\tau(0)$ and $\eta \in (0,1]$.  
Thus, $v_1 \sim \rho \sim v_2$ yields the first two scaling relations in \eqref{eq:scaling_U_Q_M}. 
As $U$ is unitary we have $\norm{U} = 1$. Thus, \eqref{eq:balanced_polar_decomposition} and the first two scaling relations in \eqref{eq:scaling_U_Q_M} imply the last bound in \eqref{eq:scaling_U_Q_M}.  

For fixed $z \in \C$, the matrix $M$ is an analytic, hence, continuous function of $\eta$. Thus, $\rho$, $v_1$ and $v_2$ are continuous functions of $\eta$. Consequently, as $v_1, v_2, u>0$, the matrices  
$U$ and $Q$ are also continuous in $\eta$. 
This completes the proof of Lemma~\ref{lem:scaling_relation_v_u}. 
\end{proof} 

\paragraph{Factorization of $\cB$} 
We now present a factorization of $\cB$ which will be the basis of our spectral analysis of $\cB$ as a linear map on the Hilbert space $(\C^{2n\times 2n}, \scalar{\genarg}{\genarg})$. From \eqref{eq:balanced_polar_decomposition}, we easily obtain 
\begin{equation} \label{eq:cB_rep_C_F} 
 \cB = 1 - \cC_M \cS = \cC_Q(1- \cC_{U} \cF) \cC_Q^{-1}, 
\end{equation}
where we introduced the positivity-preserving and Hermitian operator $\cF$ on $\C^{2n \times 2n}$ defined by 
\begin{equation} \label{eq:def_cF} 
\cF \defeq \cC_Q \cS \cC_Q. 
\end{equation}

Owing to \eqref{eq:cB_rep_C_F} and $Q \sim 1$ by \eqref{eq:scaling_U_Q_M}, the spectral properties of $\cB$ stated in Proposition~\ref{pro:stability_operator} can be obtained by analysing 
$1 - \cC_U \cF$. 
If $\rho$ is small then $U$ is well approximated by $P$ defined through 
\begin{equation} \label{eq:def_P} 
 P \defeq \sign{\Re U} =  \frac{\Re U}{\abs{\Re U}}= \begin{pmatrix} 0 & - z/\abs{z} \\  - \bar z /\abs{z} & 0 \end{pmatrix}. 
\end{equation}
Indeed, $1- \abs{\Re U}  = 1- \sqrt{1 - (\Im U)^2} \lesssim (\Im U)^2 \lesssim \rho^2$ implies that 
\begin{equation} \label{eq:P_minus_U_lesssim_rho} 
\norm{P - \Re U} \lesssim \rho^2, \qquad \qquad \norm{P - U} \lesssim \rho. 
\end{equation}
Therefore, we will first analyse the operators $\cF$ and $\cC_P \cF$. The proof of Proposition~\ref{pro:stability_operator} will then follow by perturbation theory since \eqref{eq:P_minus_U_lesssim_rho} implies
\begin{equation} \label{eq:norm_cK_minus_cL} 
 \norm{1 - \cC_U \cF - (1- \cC_P \cF)} \lesssim \norm{P - U} \lesssim \rho. 
\end{equation}

\paragraph{Commutation relations} 

\begin{lemma}[Commutation relations of $E_-$ with $M$, $Q$, $U$ and $P$] 
We have 
\begin{subequations} 
\begin{alignat}{3} 
ME_- & = - E_-M^*,\qquad \qquad \qquad &  M^* E_- & = - E_- M, && \label{eq:M_and_Emin} \\ 
 Q E_- & = E_- Q, \qquad \qquad &    Q^{-1} E_- & = E_- Q^{-1} && \label{eq:Q_Emin} \\ 
 U E_- &= - E_- U^*,\qquad \qquad &   U^* E_- &= - E_- U,  && \label{eq:U_Emin} \\ 
 PE_- &= - E_- P, \qquad \qquad &&&& \label{eq:P_Emin}
\end{alignat} 
\end{subequations} 

\end{lemma} 

\begin{proof} 
The identities in \eqref{eq:M_and_Emin} follow by a simple computation starting from \eqref{eq:block_structure_M}.
Owing to $Q \in \dM$ we immediately obtain  \eqref{eq:Q_Emin}. 
The relations in \eqref{eq:U_Emin} are a direct consequence of \eqref{eq:M_and_Emin} and \eqref{eq:Q_Emin}. 
The matrix representation of $P$ in \eqref{eq:def_P} directly implies \eqref{eq:P_Emin}. 
\end{proof}

\paragraph{Spectral properties of $\cF$}

\begin{lemma}[Spectral properties of $\cF$] \label{lem:spec_cF} 
For all $\eta \in (0,1]$ and $z \in D_{\tau}(0)$, the following holds. 
\begin{enumerate}[label=(\roman*)]
\item The range of $\cF$ is contained in the diagonal matrices, i.e. $\ran \cF \subset \dM$. Moreover, for all $R \in \C^{2n \times 2n}$, 
\begin{equation} \label{eq:cF_Emin} 
 \cF[RE_-] = - \cF[R]E_- = - E_-\cF[R] = \cF[E_-R]. 
\end{equation}
\item The top eigenvalue $\normtwoop{\cF}$ of $\cF$ is simple and satisfies
\begin{equation} \label{eq:normtwo_cF} 
 1- \normtwoop{\cF} \sim \eta /\rho. 
\end{equation}
\item There is a unique positive definite eigenvector $F$ with $\normtwo{F} = 1$ associated to $\normtwoop{\cF}$. It satisfies $F \in \dM$. 
\item The eigenvalue $-\normtwoop{\cF}$ of $\cF$ is also simple and $E_-F$ is an eigenvector corresponding to it. 
\item There are $\rho_* \sim 1$ and $\vartheta \sim 1$ such that $\eta/\rho \leq \rho_*$ implies 
\[ \normtwo{\cF[R]} \leq \normtwoop{\cF}(1- \vartheta)  \normtwo{R} \] 
for all $R \in \C^{2n\times 2n}$ satisfying $R \perp F$ and $R \perp E_- F$. 
\end{enumerate} 
\end{lemma}

Before the proof of Lemma~\ref{lem:spec_cF}, we introduce $F_U$ defined through 
\begin{equation} \label{eq:def_F_U} 
 F_U \defeq \rho^{-1} \Im U = \rho^{-1} \begin{pmatrix} \sqrt{\frac{v_1v_2}{u}} & 0 \\ 0 & \sqrt{\frac{v_1v_2}{u}} \end{pmatrix}. 
\end{equation}
The importance of $F_U$ originates from the approximate eigenvector relation 
\begin{equation} \label{eq:F_U_eigenvector_cF} 
 (1-\cF)[F_U] = \frac{\eta}{\rho} Q^2, 
\end{equation}
which is a consequence of the MDE, \eqref{eq:mde}. Indeed, \eqref{eq:mde} and \eqref{eq:balanced_polar_decomposition} imply 
\[ -U^* = Q\begin{pmatrix} \ii \eta & z \\ \bar z & \ii \eta \end{pmatrix} Q + \cF[U]. \] 
Dividing the imaginary part of this identity by $\rho$ yields \eqref{eq:F_U_eigenvector_cF}.
Moreover, from \eqref{eq:scaling_U_Q_M}, we directly deduce that 
\begin{equation} \label{eq:F_U_sim_1} 
 F_U \sim 1. 
\end{equation}

\begin{proof} 
The definition of $\cS$ in \eqref{eq:def_cS} implies $\ran \cS \subset \dM$. Since $Q \in \dM$ by \eqref{eq:def_U_Q} we deduce $\ran \cF \subset \dM$. 
As $\ran \cS \subset \dM$, we also have $\cS[RE_-] = - \cS[R] E_- = - E_- \cS[R] = \cS[E_-R]$ for all $R \in \C^{2n \times 2n}$. 
This completes the proof of (i) due to \eqref{eq:Q_Emin}.

Since $\ran \cF \subset \dM$, the restriction $\cF|_\dM$ contains all spectral properties of $\cF$ (apart from information about the possible eigenvalue $0$).  
The restriction $\cF|_\dM$ is given by 
\[ \cF\bigg[ \begin{pmatrix} r_1 & 0 \\ 0 & r_2 \end{pmatrix} \bigg] = \begin{pmatrix} \smallF r_2 & 0 \\ 0 & \smallF^t r_1 \end{pmatrix}\] 
for $r_1, r_2 \in \C^n$, 
where we introduced the $n\times n$-matrix $\smallF$ defined by 
\begin{equation} \label{eq:def_smallF} 
 \smallF r \defeq  \bigg(\frac{u v_1}{v_2}\bigg)^{1/2} \smallS \bigg(r \bigg(\frac{uv_2}{v_1}\bigg)^{1/2}\bigg) 
\end{equation}
for $r \in \C^n$. 
Hence, in the standard basis of $\dM\cong \C^{2n}$, the restriction $\cF|_\dM$ is represented by the $2n \times 2n$ matrix 
\[\boldsymbol{F} = \begin{pmatrix} 0 & \smallF\\ \smallF^t & 0 \end{pmatrix},\]
which was introduced in \cite[Eq.~(3.27b)]{Altcirc} and analyzed in \cite[Lemma~3.4]{Altcirc}  using \cite[Lemma~3.3]{AltGram}. 
Using the notation of \cite[Lemma~3.3]{AltGram}, we have $L=2$ due to \eqref{eq:condition_flatness} and $r_+ \sim r_- \sim 1$ due to \eqref{eq:scaling_U_Q_M}.
Thus, \cite[Lemma~3.3]{AltGram} directly implies the simplicity of the top eigenvalue $\normtwoop{\cF}$ and the existence of a unique positive definite eigenvector $F$ of $\cF$ 
corresponding to $\normtwoop{\cF}$ with $\normtwo{F}=1$. Moreover, $F \in \dM$.  
Owing to the second relation in \eqref{eq:cF_Emin}, $E_-F$ is an eigenvector of $\cF$ associated to $-\normtwoop{\cF}$. 

For the proof of (ii), we apply $\scalar{F}{\genarg}$ to \eqref{eq:F_U_eigenvector_cF} and obtain 
\[ 1- \normtwoop{\cF} = \frac{\eta}{\rho} \frac{\avg{FQ^2}}{\avg{FF_U}} \sim \frac{\eta}{\rho} \frac{\avg{F}}{\avg{F}} \sim \frac{\eta}{\rho}.\] 
Here, we used the positive definiteness of $F$, $Q \sim 1$ by \eqref{eq:scaling_U_Q_M} and $F_U \sim 1$ by \eqref{eq:F_U_sim_1} in the second step. 
 Since $\wh{\lambda} \sim 1$ due to \cite[eq.~(3.17)]{AltGram} in \cite[Lemma~3.3]{AltGram} and $r_+ \sim r_- \sim 1$ (see above),  
 the bound in (v) for $R \in \dM$ follows from \cite[Lemma~3.3]{AltGram}.  Since $\cF$ vanishes on the orthogonal complement of $\dM$, this
 completes the proof of Lemma~\ref{lem:spec_cF}.  
\end{proof} 

\subsection{Spectral properties of $\cC_P \cF$ and $\cC_U \cF$}

For brevity we introduce the following shorthand notations for the operators in the following lemma. 
We define 
\begin{equation}  \label{eq:def_cK_cL} 
 \cK \defeq 1 - \cC_P \cF, \qquad \cL \defeq 1 - \cC_U \cF. 
\end{equation}
In the following lemma, we prove some resolvent bounds for these operators and show that they have at most two small eigenvalues. 

\begin{lemma}[Resolvent bounds, number of small eigenvalues]  \label{lem:small_eigenvalues} 
There are (small) $\rho_* \sim 1$ and $\eps \sim 1$ such that for all $z \in D_\tau(0)$ and $\eta >0$ satisfying $\rho + \eta/\rho \leq \rho_*$ and, for all $\cT \in \{ \cK, \cL\}$, 
the following holds. 
\begin{enumerate}[label=(\roman*)]
\item \label{item:small_eigenvalues_resolvent_bound} 
For all $\omega \in \C$ with $\omega \notin D_\eps(0) \cup D_{1- 2 \eps}(1)$, we have 
\[ \normtwoop{(\cT-\omega)^{-1}} + \norm{(\cT- \omega)^{-1}} + \norm{(\cT^*- \omega)^{-1}} \lesssim 1.\] 
\item The spectral projection $\cP_\cT$ of $\cT$, defined by 
\begin{equation} \label{eq:cP_cT_contour_integral} 
\cP_\cT \defeq - \frac{1}{2\pi \ii}  \int_{\pt D_\eps(0)} (\cT- \omega)^{-1} \di \omega, 
\end{equation}
satisfies $\rank \cP_\cT = 2$. Moreover, for $\cQ_\cT \defeq 1- \cP_\cT$, we have 
\begin{equation}\label{eq:QT}
 \norm{\cP_\cT} + \norm{\cQ_\cT} +  \norm{\cP_\cT^*} + \norm{\cQ_\cT^*}  + \normtwoop{\cT^{-1}\cQ_\cT} + \norm{(\cT^*)^{-1} \cQ_\cT} + \norm{\cT^{-1}\cQ_\cT}  \lesssim 1.
 \end{equation}
 \item For fixed $z \in D_\tau(0)$, the spectral projections $\cP_\cT$ and $\cQ_{\cT}$ are continuous in $\eta$ as long as $\rho + \eta/\rho \leq \rho_*$. 
\end{enumerate} 
\end{lemma} 

The proof of Lemma~\ref{lem:small_eigenvalues} is motivated by the proofs of \cite[Lemma~4.7]{AltEdge} and \cite[Lemma~5.1]{Altshape}. However, the additional extremal eigendirection of $\cF$ 
requires a novel flow interpolating between $1- \cF^2$ and $1- (\cC_P \cF)^2$ instead of $1- \cF$ and $1- \cC_P \cF$.

\begin{proof}
From \eqref{eq:condition_flatness}, we deduce that $\cS[R] \lesssim \avg{R}$ for all positive semidefinite matrices $R \in \C^{2n\times 2n}$. 
Thus, \cite[Lemma~B.2(i)]{Altshape} implies that $\normtwoinf{\cS} \lesssim 1$. 
Therefore, for all $\cT \in \{ \cK, \cL\}$, we have $\normtwoinf{1-\cT} \lesssim 1$ due to $Q \sim 1$ and $\norm{U} =1$ by Lemma~\ref{lem:scaling_relation_v_u}. 
Hence, owing to \cite[Lemma~B.2(ii)]{Altshape} and $\abs{\omega -1}\gtrsim 1$, it suffices to find $\eps \sim 1$ such that 
\begin{enumerate}[label=(\roman*)]
\item uniformly for all $\omega \notin D_\eps(0) \cup D_{1-2\eps}(1)$, we have 
\begin{equation} \label{eq:normtwo_cT_omega} 
\normtwoop{(\cT - \omega)^{-1}} \lesssim 1, 
\end{equation} 
\item the rank of $\cP_\cT$ equals 2, i.e., $\rank \cP_\cT =2$ 
\end{enumerate}
for $\cT \in \{  \cK, \cL \}$. 
Both claims for $\cT = \cK$ will follow from the corresponding statements for $\cT = 1- (\cC_P \cF)^2$ which we now establish by interpolating between $1 - \cF^2$ and $1- (\cC_P \cF)^2$.  
If $\cT = 1- \cF^2$ then both assertions follow from Lemma~\ref{lem:spec_cF}. Moreover, a simple perturbation argument using Lemma~\ref{lem:spec_cF} shows that 
\begin{equation} \label{eq:eigenvector_cF_approx_F_U} 
 F = \normtwo{F_U}^{-1} F_U + \ord(\eta/\rho), 
\end{equation}
where $F$ is the eigenvector of $\cF$ introduced in Lemma~\ref{lem:spec_cF} (cf.\ the proof of \cite[Lemma~3.5]{Altcirc} for a similar argument). 

In order to interpolate between $1 - \cF^2$ and $1- (\cC_P \cF)^2$ we use the following flow. For any $t \in [0,1]$, we define 
\[ \cT_t \defeq 1- \cV_t \cF, \qquad \cV_t \defeq (1 - t) \cF + t \cC_P \cF\cC_P. \] 
Then $\cT_0 = 1 - \cF^2$ and $\cT_1 = 1- (\cC_P \cF)^2$. 
We now show \eqref{eq:normtwo_cT_omega} for $\cT = \cT_t$ uniformly for all $t \in [0,1]$.  
To that end, we verify that $\normtwo{(\cT_t - \omega)[R]}\gtrsim 1$ uniformly for $R \in \C^{2n \times 2n}$ satisfying $\normtwo{R} = 1$. 
If $\abs{\omega} \geq 3$ then this follows from $\normtwoop{\cV_t} \leq \normtwoop{\cF} \leq 1$ by \eqref{eq:normtwo_cF}. 
Let $\abs{\omega} \leq 3$ and $R \in \C^{2n \times 2n}$ satisfy $\normtwo{R} =1$. 
We have the orthogonal decomposition $R = \alpha_+ F + \alpha_- E_- F + R_\perp$, where $R_\perp \perp E_\pm F$ (recall $E_+=1$ from \eqref{eq:def_E_pm}), and estimate 
\begin{equation} \label{eq:proof_resolvent_bound_aux1}  
\begin{aligned} 
\normtwo{(\cT_t -\omega)[R]}^2 & = \abs{\omega}^2 (\abs{\alpha_+}^2 + \abs{\alpha_-}^2) + \normtwo{(1-\omega - \cV_t \cF)[R_\perp]}^2 + \ord(\eta/\rho)\\ 
& \geq \eps^2 ( \abs{\alpha_+}^2 + \abs{\alpha_-}^2) + (\vartheta - 2\eps)^2 \normtwo{R_\perp}^2 + \ord(\eta/\rho). 
\end{aligned} 
\end{equation} 
We now explain how \eqref{eq:proof_resolvent_bound_aux1} is obtained. 
The identity in \eqref{eq:proof_resolvent_bound_aux1} follows from $\cV_t \cF [E_\pm F] = E_\pm F + \ord(\eta/\rho)$ due to \eqref{eq:eigenvector_cF_approx_F_U}, \eqref{eq:normtwo_cF}, 
 $\cC_P [F_U] = F_U$, \eqref{eq:P_Emin} and \eqref{eq:cF_Emin}. 
The lower bound in \eqref{eq:proof_resolvent_bound_aux1} is a consequence of 
\[ \normtwo{(1- \omega - \cV_t \cF)[R_\perp]} \geq (\abs{1-\omega} - \normtwoop{\cF}(1-\vartheta)) \normtwo{R_\perp}
\geq (\vartheta - 2\eps) \normtwo{R_\perp},  \] 
where, in the first step, we used $\normtwoop{\cV_t} \leq 1$ and $\normtwo{\cF[R_\perp]} \leq \normtwoop{\cF}(1-\vartheta) \normtwo{R_\perp}$ due to part (v) of Lemma~\ref{lem:spec_cF}. 
In the second step, we employed $\normtwoop{\cF} \leq 1$ and $\abs{1- \omega} \geq 1 -2 \eps$. 
This shows \eqref{eq:proof_resolvent_bound_aux1} which implies \eqref{eq:normtwo_cT_omega} 
for $\cT = \cT_t$ if $\eps \sim 1$ and $\rho_* \sim 1$ are chosen sufficiently small. 

A similar but simpler argument to the proof of \eqref{eq:proof_resolvent_bound_aux1} shows that 
$\normtwo{(\cK - \omega)[R]} \gtrsim 1$ uniformly for all $R \in \C^{2n \times 2n}$ satisfying $\normtwo{R} =1$.
This implies \eqref{eq:normtwo_cT_omega} for $\cT = \cK$.  
 In particular, by \cite[Lemma~B.2(ii)]{Altshape} and $\abs{\omega -1}\gtrsim 1$, the 
bound $\normtwoop{(\cK-\omega)^{-1}}\lesssim 1$ from \eqref{eq:normtwo_cT_omega} implies the same bound in the norm $\norm{\genarg}$, i.e., 
Lemma~\ref{lem:small_eigenvalues} \ref{item:small_eigenvalues_resolvent_bound} for $\cT = \cK$. Hence, the contour integral representation for $\cP_\cT$ in \eqref{eq:cP_cT_contour_integral} implies the bounds on the projections in \eqref{eq:QT}.
The remaining bounds in \eqref{eq:QT} follow similarly from $\eps \sim 1$ and the contour integral representation 
\[ \cK^{-1} \cQ_{\cK} = -\frac{1}{2\pi \ii} \int_{\pt D_{1-2\eps}(1)} \omega^{-1} \big( \cK - \omega)^{-1} \dd \omega, \] 
which is a consequence of Lemma~\ref{lem:small_eigenvalues} \ref{item:small_eigenvalues_resolvent_bound} for $\cT = \cK$.

Owing to Lemma~\ref{lem:scaling_relation_v_u} (iii) and $\cF = \cC_Q \cS \cC_Q$ (cf.~\eqref{eq:def_cF}), $\cK$ and $\cL$ are continuous functions of $\eta$. Hence, the contour integral representation of $\cP_\cT$ in \eqref{eq:cP_cT_contour_integral} 
implies (iii). 

What remains in order to complete the proof of Lemma~\ref{lem:small_eigenvalues} for $\cT = \cK$ is showing $\rank \PK = 2$. 
The bound in \eqref{eq:normtwo_cT_omega} with $\cT = \cT_t$ implies that $\cP_{\cT_t}$ is well defined for all $t \in [0,1]$. 
Moreover, the map $t \mapsto \rank \cP_{\cT_t}$ is continuous and, hence, constant as a continuous, integer-valued map. 
Therefore, 
\begin{equation} \label{eq:rank_P_T_1} 
 \rank \cP_{\cT_1} = \rank \cP_{\cT_0} = 2,  
\end{equation}
where we used in the last step that $\cT_0 = 1-\cF^2$ and Lemma~\ref{lem:spec_cF} (ii), (iv) and (v). 
Since the generalized eigenspace of $1- (\cC_P\cF)^2$ corresponding to $1-\mu^2$ contains in the generalized eigenspace of $1- \cC_P \cF$ corresponding to $1 - \mu$ for any $\mu \in \C$, 
the identity $\rank \cP_{\cT_1}=2$ from \eqref{eq:rank_P_T_1} implies 
\begin{equation} \label{eq:rank_P_K_leq_2} 
\rank \PK \leq 2. 
\end{equation}
The following lemma provides the corresponding lower bound. 

\begin{lemma}[Eigenvalues of $\cK$ in $D_\eps(0)$] \label{lem:cK_small_eigenvalue} 
Let $\eps$ and $\rho_*$ be chosen as in Lemma~\ref{lem:small_eigenvalues}. 
If $\rho + \eta/\rho \leq \rho_*$ then $\spec(\cK) \cap D_\eps(0)$ consists of a unique eigenvalue $\kappa$ of $\cK$. This eigenvalue is positive, has algebraic and geometric multiplicity two 
and is a continuous function of $\eta$ for fixed $z \in D_\tau(0)$. 
\end{lemma} 

\begin{proof} 
Since $\rank \PK \leq 2$ by \eqref{eq:rank_P_K_leq_2}, the set 
$\spec(\cK) \cap D_\eps(0)$ contains (counted with algebraic multiplicity) at most two eigenvalues of $\cK$. 
We will now show that it contains one eigenvalue of (algebraic and geometric) multiplicity two.  
As $\ran \cC_P\cF \subset \dM$ it suffices to study the corresponding eigenvalue problem on $\dM$. 
Let $r_1, r_2 \in \C^{n}$ be vectors. We apply $\cC_P \cF$ to the diagonal matrix $R = \diag(r_1,r_2) \in \dM$ and obtain 
\begin{equation} \label{eq:cC_P_cF_applied_to_diagonal} 
  \cC_P \cF[R] = \cC_P \cF\bigg[ \begin{pmatrix} r_1 & 0 \\ 0 & r_2 \end{pmatrix}\bigg] = \begin{pmatrix} \smallF^t r_1 & 0 \\ 0& \smallF r_2 \end{pmatrix}, 
\end{equation}
where $\smallF$ denotes the $n\times n$-matrix defined in \eqref{eq:def_smallF} in the proof of Lemma~\ref{lem:spec_cF}. 
The spectral radii of the matrix $\smallF$ and its transpose $\smallF^t$ agree.  
We denote this common spectral radius by $1- \kappa$.  Since $\normtwoop{\cC_P\cF} = \normtwoop{\cF} < 1$ by Lemma~\ref{lem:spec_cF} we have $\kappa>0$. 
The entries of the matrices $\smallF$ and $\smallF^t$ are strictly positive. 
Hence, by the Perron-Frobenius theorem, there are $r_1, r_2 \in (0,\infty)^n$ such that $\smallF^t r_1 = (1-\kappa)r_1$ and $\smallF r_2 = (1-\kappa)r_2$. 
Thus, $\cK[R] = \kappa R$, where we used \eqref{eq:cC_P_cF_applied_to_diagonal} and introduced $R \defeq \diag(r_1, r_2) \in \C^{2n \times 2n}$. Since $r_1, r_2 >0$, 
$E_- R$ and $R$ are linearly independent. Moreover, $\cK[E_- R] = (1- \cC_P \cF)[E_-R] = E_-\cK[R]=\kappa E_- R$ due to \eqref{eq:cF_Emin} 
and \eqref{eq:P_Emin}. 
Therefore, $\spec(\cK) \cap D_\eps(0) = \{ \kappa \}$ and $R$ and $E_-R$ span the eigenspace of $\cK$ associated to $\kappa$, i.e., $\rank \PK = 2$. 
Since $\cK$ and $\PK$ are continuous functions of $\eta$, the eigenvalue $\kappa = \Tr(\cK \PK)/2$ is also continuous with respect to $\eta$. 
This completes the proof of Lemma~\ref{lem:cK_small_eigenvalue}. 
\end{proof} 

Since $\rank \PK = 2$ by Lemma~\ref{lem:cK_small_eigenvalue}, 
this completes the proof of Lemma~\ref{lem:small_eigenvalues} for $\cT = \cK$. 

Owing to \eqref{eq:norm_cK_minus_cL}, we have $\normtwoop{\cL - \cK} \lesssim \rho$. 
Hence, possibly shrinking $\eps \sim 1$ and $\rho_* \sim 1$ and a simple perturbation theory argument 
show the estimates in \ref{item:small_eigenvalues_resolvent_bound} and \eqref{eq:QT} for $\cT = \cL$. 
Moreover, viewing $\cL$ as perturbation of $\cK$ and using $\rank \PK =2$ yield $\rank \cP_{\cL} =2$ 
for sufficiently small $\rho_* \sim 1$. 
This completes the proof of Lemma~\ref{lem:small_eigenvalues}. 
\end{proof}

Using the spectral properties of $\cK$ established in Lemma~\ref{lem:small_eigenvalues}, we show in the following lemma 
that $E_\pm F_U$ are approximate eigenvectors of $\cK$ associated to its small eigenvalue $\kappa$ from Lemma~\ref{lem:cK_small_eigenvalue}.  

\begin{lemma}[Eigenvectors of $\cK$ associated to $\kappa$] \label{lem:eigenvectors_C_S_cF} 
Let $\eps$ and $\rho_*$ be chosen as in Lemma~\ref{lem:small_eigenvalues} as well as $\cP_\cK$, $\cQ_\cK$ defined in Lemma~\ref{lem:small_eigenvalues} for $\cK = 1- \cC_P \cF$. 
If $\rho + \eta/\rho \leq \rho_*$ then the following holds. 
\begin{enumerate}[label=(\roman*)]
\item  \label{item:K_plus} 
There are left and right eigenvectors $\wh{K}_+$ and $K_+$ of $\cK$ corresponding to $\kappa$ such that 
\begin{equation} \label{eq:eigenvectors_1_C_S_F} 
K_+ = F_U - \frac{\eta}{\rho}\cK^{-1}\cQ_\cK\cC_P[Q^2],\qquad \qquad \wh{K}_+ = F_U - \frac{\eta}{\rho} (\cK^*)^{-1} \cQ_\cK^*[Q^2].
\end{equation}
They are elements of $\dM$, continuous functions of $\eta$ for fixed $z \in D_\tau(0)$ and satisfy 
\begin{equation} \label{eq:scalar_L_S_R_S} 
\scalar{\wh{K}_+}{K_+} = \avg{F_U^2} + \ord(\eta^2/\rho^2). 
\end{equation}
Moreover, we have 
\begin{equation} \label{eq:expansion_lambda_S} 
 \kappa = \frac{\eta}{\rho} \frac{\pi}{\avg{F_U^2}} + \ord(\eta^2/\rho^2) . 
\end{equation}
\item \label{item:K_minus} Furthermore, $K_- \defeq E_-K_+$ and $\wh{K}_- \defeq E_- \wh{K}_+$ are also right and left eigenvectors of $\cK$ corresponding to $\kappa$ that are linearly independent of $K_+$ and $\wh{K}_+$, respectively. 
\item \label{item:K_projections} The projections $\PK$ and $\PK^*$ have the representation 
\[ \PK = \frac{\scalar{\wh{K}_+}{\genarg}}{\scalar{\wh{K}_+}{K_+}} K_+ + \frac{\scalar{\wh{K}_-}{\genarg}}{\scalar{\wh{K}_-}{K_-}} K_-, \qquad  
\PK^* = \frac{\scalar{K_+}{\genarg}}{\scalar{K_+}{\wh{K}_+}} \wh{K}_+ + \frac{\scalar{K_-}{\genarg}}{\scalar{K_-}{\wh{K}_-}} \wh{K}_-. \] 
In particular, $\ran \PK\subset \dM$ and $\ran \PK^*\subset \dM$ as well as $\PK \oM = \PK^* \oM = \{ 0 \}$. 
\end{enumerate} 
\end{lemma} 

For the proof, we note that the definition of $F_U$ in \eqref{eq:def_F_U}, \eqref{eq:balanced_polar_decomposition} and the definition of $\rho$ in \eqref{eq:def_rho} 
imply 
\begin{equation} \label{eq:scalar_Q_2_F_U} 
 \scalar{Q^2}{F_U} = \rho^{-1} \avg{Q \Im U Q } = \pi. 
\end{equation}

\begin{proof} 
We start the proof of \ref{item:K_plus} by remarking that 
the eigenspace of $\cK$ associated to $\kappa$ is contained in $\dM$ since $\ran \cC_P \cF \subset \dM$. 
Next, we apply $\QK\cC_P$ to \eqref{eq:F_U_eigenvector_cF}, use $\cC_P[F_U] = F_U$ and $\cK = 1 - \cC_P \cF$ and obtain 
\[ \cK \QK[F_U] = \frac{\eta}{\rho} \QK\cC_P[Q^2].  \] 
Hence, setting $K_+ \defeq \cP_{\cK}[F_U]$ yields 
\[ K_+ = \cP_{\cK}[F_U] = F_U - \QK[F_U] = F_U - \frac{\eta}{\rho}\cK^{-1}\QK\cC_P[Q^2]. \] 
This proves the expansion of $K_+$ in \eqref{eq:eigenvectors_1_C_S_F}. 
For the proof of the expansion of $\wh{K}_+$, we use $\cC_P [F_U] = F_U$ in \eqref{eq:F_U_eigenvector_cF}, apply $\QK^*$ to the result and set $\wh{K}_+ \defeq \cP_{\cK}^*[F_U]$. 
Then the expansion of $\wh{K}_+$ follows similarly as the one of $K_+$. 
The continuity of $F_U = \rho^{-1} \Im U$ due to Lemma~\ref{lem:scaling_relation_v_u} (iii) 
and the continuity of $\PK$ and $\PK^*$ due to 
 Lemma~\ref{lem:small_eigenvalues} (iii) imply that $K_+$ and $\wh{K}_+$ are also continuous. 

The relation in \eqref{eq:scalar_L_S_R_S} follows directly from \eqref{eq:eigenvectors_1_C_S_F} since $\scalar{\wh{K}_+}{K_+ - F_U} =  \scalar{\wh{K}_+ - F_U}{K_+} = 0$ due to $\QK \PK = 0$ and, thus, 
\[ \scalar{\wh{K}_+}{K_+} = \scalar{\wh{K}_+}{F_U} + \scalar{\wh{K}_+}{K_+-F_U} = \avg{F_U^2} + \scalar{\wh{K}_+ - F_U}{K_+} + \ord(\eta^2/\rho^2). \] 
For the proof of \eqref{eq:expansion_lambda_S}, we deduce from \eqref{eq:eigenvectors_1_C_S_F} and $\QK[K_+] = 0$ that
\[ \kappa \scalar{\wh{K}_+}{K_+} = \scalar{\cK^*[\wh{K}_+]}{K_+} = \scalar{\cK^*[F_U]}{K_+}  
= \frac{\eta}{\rho} \scalar{Q^2}{F_U}  + \ord(\eta^2/\rho^2) 
= \frac{\eta\pi}{\rho} + \ord(\eta^2/\rho^2),  \] 
where we used $\cC_P [F_U] = F_U$, \eqref{eq:F_U_eigenvector_cF} and  \eqref{eq:scalar_Q_2_F_U}. 
Therefore, we obtain \eqref{eq:expansion_lambda_S} due to \eqref{eq:scalar_L_S_R_S}. 

We now prove \ref{item:K_minus}. 
From \eqref{eq:cF_Emin} and \eqref{eq:P_Emin}, we deduce that $\cK$ commutes with $\cC_{E_-,1}$. Hence, 
$E_- K_+$ and $E_-\wh{K}_+$ are right and left eigenvectors of $\cK$ corresponding to $\kappa$ as well. 
For sufficiently small $\rho_* \sim 1$ and $\eta/\rho \leq \rho_*$, $K_+$ and $\wh{K}_+$ are strictly positive definite. Hence, $K_-$ and $\wh{K}_-$ are linearly independent of $K_+$ and $\wh{K}_+$, respectively. 

Part \ref{item:K_projections} follows directly from Lemma~\ref{lem:small_eigenvalues}, Lemma~\ref{lem:cK_small_eigenvalue} and Lemma~\ref{lem:eigenvectors_C_S_cF} \ref{item:K_plus}, \ref{item:K_minus}. This completes the proof of Lemma~\ref{lem:eigenvectors_C_S_cF}. 
\end{proof}

\subsection{Eigenvalues of $\cL$ in $D_\eps(0)$} \label{subsec:eigenvalues_cL} 

In this section, we study the small eigenvalues of $\cL$ as perturbations of the small eigenvalue $\kappa$ of $\cK$ (see Lemma~\ref{lem:cK_small_eigenvalue}). 

\begin{lemma}[Eigenvalues of $\cL$ in $D_\eps(0)$]  \label{lem:eigenvalues_cL}
There are $\rho_* \sim 1$ and $\eps \sim 1$ such that if $\rho + \eta/\rho \leq \rho_*$ then 
 $\spec(\cL) \cap D_\eps(0)$ consists of two eigenvalues $\beta$ and $\beta_*$. 
Each of these eigenvalues has algebraic and geometric multiplicity one. 
Moreover, they satisfy $\abs{\beta_*} < \abs{\beta}$ and 
\begin{equation} \label{eq:eigenvalues_cL_expansion} 
\beta_*  = \kappa + \ord (\rho^3 + \eta \rho), \qquad \qquad  
\beta  = \kappa + 2 \rho^2 \frac{\avg{F_U^4}}{\avg{F_U^2}} + \ord(\rho^3 + \eta \rho). 
\end{equation}  
Furthermore, $\beta$ and $\beta_*$ are continuous functions of $\eta$ for fixed $z \in D_\tau(0)$. 
\end{lemma} 

We remark that the eigenvalues of $\cL$ are denoted by $\beta$ and $\beta_*$ since the spectra of $\cL$ 
and $\cB$ agree. Indeed, $\cB$ and $\cL$ are related through the similarity transform $\cB = \cC_Q\cL \cC_Q^{-1}$ due to \eqref{eq:cB_rep_C_F}. 

To lighten the notation in the following, we denote the difference between $\cL$ and $\cK$ by 
\begin{equation} \label{eq:def_cD} 
 \cD \defeq \cL - \cK. 
\end{equation}

\begin{proof} 
We decompose $\cL$ according to the splitting $\PK + \QK = 1$, i.e. we write
\begin{equation} \label{eq:decomposition_cL} 
 \cL =  \begin{pmatrix} \PK \cL \PK &  \PK \cL \QK \\ \QK \cL \PK & \QK \cL \QK \end{pmatrix}.
\end{equation}
More precisely, by \eqref{eq:decomposition_cL} we mean that
 we consider the (not necessarily orthogonal) decomposition $\C^{2n} = \ran \PK + \ran \QK$
into two complementary subspaces and the operators in the right-hand side of \eqref{eq:decomposition_cL}  act among the appropriate
subspaces in this decomposition, e.g. $\QK \cL \PK$ is a linear operator from $\ran \PK$ to $\ran \QK$.
 Notice that  $\QK (\cL -\omega) \QK$
 (viewed as a linear map on $\ran \QK$) is invertible  if $|\omega| \le \eps$, where
 $\eps$ was chosen as in Lemma~\ref{lem:small_eigenvalues}. To see this,  we use the identity
 \begin{equation}\label{eq:triv}
 (A+B) \big[ I + A^{-1} B\big]^{- 1}A^{-1} = I
\end{equation}
for $A= \QK (\cK-\omega) \QK$, $B= \QK (\cL-\cK) \QK$ and $I$ being the identity map on $\ran \QK$
and notice that $A$ (viewed as a map on $\ran \QK$)
 is invertible by Lemma~\ref{lem:small_eigenvalues}, in fact $A^{-1} = \QK (\cK-\omega)^{-1} \QK$
 with $\| A^{-1}\|\lesssim1$ by \eqref{eq:QT}.
 Moreover $\| A^{-1} B \|\le \| A^{-1}\| \| \QK\|^2 \| \cL -\cK\| \lesssim \rho$, where we used 
$\norm{\cL - \cK} \lesssim \rho$ by \eqref{eq:norm_cK_minus_cL} and $\|\QK\| \le 1 + \|\PK\| \lesssim 1$ by \eqref{eq:QT}. 
Therefore $I+A^{-1} B$ is also invertible if $\rho$ is sufficiently small, yielding
 the invertibility of $A+B=  \QK (\cL-\omega) \QK$ from \eqref{eq:triv} and that
 \begin{equation}\label{eq:QLbound}
\big\| \big[ \QK (\cL-\omega) \QK\big]^{-1}\big\| \lesssim 1.
\end{equation}
 
Moreover, we use \eqref{eq:decomposition_cL} and Schur's determinant identity to compute the determinant of $\cL -\omega$ and obtain 
\begin{equation}\label{eq:proof_eigenvalues_cL_aux1}
 \det(\cL - \omega) = \det(\QK(\cL-\omega) \QK)  \det( \PK ( \cL - \omega) \PK - \PK \cL\QK ( \QK (\cL - \omega)\QK)^{-1} \QK \cL \PK). 
\end{equation}
Since the first determinant on the right-hand side is not zero for $|\omega|\le \eps$, the small eigenvalues of $\cL$
are exactly those $\omega$'s for which the second determinant vanishes. Note that this is a $2\times 2$ determinant
since $\ran \PK$ is two dimensional. Now we write this determinant in a convenient basis.
In the basis $(K_+,\, K_-)$ of $\ran \PK$ (cf.\ Lemma~\ref{lem:eigenvectors_C_S_cF}), we have 
\begin{equation} \label{eq:PK_cL_minus_omega_PK_aux1} 
 \PK ( \cL - \omega) \PK = \PK ( \cK - \omega) \PK  + \PK \cD \PK = \begin{pmatrix} \kappa - \omega & 0 \\ 0 &\kappa - \omega \end{pmatrix} + \Lambda, 
\end{equation}
where we introduce the $2 \times 2$-matrix $\Lambda$ defined through 
\begin{equation} \label{eq:def_Lambda}
 \Lambda \defeq \begin{pmatrix} 
\frac{\scalar{\hat{K}_+}{\cD[K_+]}}{\scalar{\hat{K}_+}{K_+}}  & \frac{\scalar{\hat{K}_+}{\cD[K_-]} }{\scalar{\hat{K}_+}{K_+}} \\ 
\frac{\scalar{\hat{K}_-}{\cD[K_+]}}{\scalar{\hat{K}_-}{K_-}}  & \frac{\scalar{\hat{K}_-}{\cD[K_-]} }{\scalar{\hat{K}_-}{K_-}} 
\end{pmatrix}. 
\end{equation} 
The following lemma which will be shown in Subsection~\ref{subsec:proofs_aux_results} provides a precise expansion of $\Lambda$ in the small $\rho$ regime. 
\begin{lemma}[Expansion of $\Lambda$] \label{lem:expansion_Lambda} 
For $\Lambda$ defined in \eqref{eq:def_Lambda}, we have the expansion
\[ \Lambda = 2\rho^2 \begin{pmatrix} \frac{\avg{F_U^4}}{\avg{F_U^2}} & 0 \\ 0 &0 \end{pmatrix} + \ord(\eta \rho ). \] 
\end{lemma} 
Lemma~\ref{lem:expansion_Lambda} and \eqref{eq:PK_cL_minus_omega_PK_aux1} imply 
\begin{equation} \label{eq:PK_cL_minus_omega_PK_aux2} 
 \PK (\cL -\omega) \PK = \begin{pmatrix} \kappa + 2 \rho^2 \frac{\avg{F_U^4}}{\avg{F_U^2}}  - \omega & 0 \\ 0 & \kappa - \omega \end{pmatrix} + \ord(\eta \rho).  
\end{equation}

What remains in order to compute the $2 \times 2$-determinant in \eqref{eq:proof_eigenvalues_cL_aux1} 
is estimating $\PK \cL\QK$ and $\QK \cL \PK$. To that end we use the following lemma 
which will be proven in Subsection~\ref{subsec:proofs_aux_results} below.

\begin{lemma}[Expansion of $\cD[K_\pm{]}$ and $\cD^*[\wh{K}_\pm{]}$] \label{lem:expansion_cD_K} 
Let $\cD$ be defined as in \eqref{eq:def_cD}. Let $K_\pm$ and $\wh{K}_\pm$ be the eigenvectors of $\cK$
introduced in Lemma~\ref{lem:eigenvectors_C_S_cF}. Then we have 
\[ \cD[K_+] = \ord(\rho + \eta), \qquad \cD[K_-] = \ord(\eta), \qquad \cD^*[\wh{K}_+] = \ord(\rho^2 + \eta), 
\qquad \cD^*[\wh{K}_-] = \ord(\eta). \] 
\end{lemma} 

As $\QK \cL \PK = \QK \cK \PK + \QK \cD \PK = \QK \cD \PK$ and $\PK \cL \QK = \PK \cK \QK + \PK \cD \QK = \PK \cD \QK$, the representation of $\PK$ 
in Lemma~\ref{lem:eigenvectors_C_S_cF} \ref{item:K_projections}  yields 
\begin{subequations} 
\begin{align} 
 \norm{\QK \cL \PK} & \lesssim \max\{ \norm{\QK \cD[K_+]}, \norm{\QK \cD[K_-]} \} \lesssim \rho + \eta,  
\label{eq:proof_eigenvalues_cL_aux2}\\ 
 \norm{\PK \cL \QK} & \lesssim \sup_{\norm{R}=1} \max\big \{ \abs{\scalar{\wh{K}_+}{\cD \QK[R]}}, \abs{\scalar{\wh{K}_-}{\cD \QK[R]}} \big\} \lesssim \rho^2 + \eta, 
\label{eq:proof_eigenvalues_cL_aux3}
\end{align} 
\end{subequations} 
where the last steps follow from Lemma~\ref{lem:expansion_cD_K} and \eqref{eq:QT}.  

Therefore, we combine \eqref{eq:QLbound}, \eqref{eq:PK_cL_minus_omega_PK_aux2}, 
\eqref{eq:proof_eigenvalues_cL_aux2} and \eqref{eq:proof_eigenvalues_cL_aux3}, use $\eta \lesssim \rho$ and obtain 
\begin{equation} \label{eq:proof_eigenvalues_cL_aux4} 
\PK ( \cL - \omega) \PK - \PK \cL\QK ( \QK (\cL - \omega)\QK)^{-1} \QK \cL \PK = 
\begin{pmatrix} \kappa + 2 \rho^2 \frac{\avg{F_U^4}}{\avg{F_U^2}}  - \omega & 0 \\ 0 & \kappa - \omega \end{pmatrix} + \ord(\rho^3 + \eta \rho)  
\end{equation}
with respect to the basis vectors $K_+$ and $K_-$. 

We now analyse the small eigenvalues of $\cL$. 
We have seen after \eqref{eq:proof_eigenvalues_cL_aux1} that, for any $|\omega|\le \eps$, we have $\det(\cL -\omega)=0$ if and only if 
\[ \det\Big(\PK ( \cL - \omega) \PK - \PK \cL\QK ( \QK (\cL - \omega)\QK)^{-1} \QK \cL \PK \Big) = 0. \]
Owing to \eqref{eq:proof_eigenvalues_cL_aux4}, the latter relation is equivalent to  
\begin{equation} \label{eq:proof_eigenvalues_cL_aux5} 
 0 = ( \kappa - \omega + \rho^2 \gamma +\delta_{11}) (\kappa - \omega  + \delta_{22}) - \delta_{12} \delta_{21}, 
\end{equation}
where $\gamma \defeq 2 \avg{F_U^4}/\avg{F_U^2}$ and $\delta_{ij}$ are the entries of the error term on the right-hand side of \eqref{eq:proof_eigenvalues_cL_aux4}. In particular, $\gamma \sim 1$ by \eqref{eq:F_U_sim_1} and $\delta_{ij} = \ord(\rho^3 + \eta \rho)$. 
The quadratic equation in \eqref{eq:proof_eigenvalues_cL_aux5} has the solutions 
\[ \omega_\pm = \kappa + \frac{\gamma}{2} \rho^2 ( 1 \pm 1) + \ord(\delta + \delta^2/\rho^2),\] 
where $\delta \defeq \sup_{i,j} \abs{\delta_{ij}}$. 
As $\kappa$, $\rho$, $\gamma$ and $\delta_{ij}$ are continuous in $\eta$ (the continuity of $\delta_{ij}$ follows from the continuity of $\PK$ and $\cL$), $\omega_\pm$ are continuous in $\eta$. 
Since $\delta = \ord(\rho^3 + \eta \rho)$ and $\rho \gtrsim \eta$, that is 
\[ 
 \omega_+   = \kappa + 2 \rho^2\frac{\avg{F_U^4}}{\avg{F_U^2}} + \ord(\rho^3 + \eta \rho), \qquad  
 \omega_-  = \kappa + \ord(\rho^3 + \eta \rho). 
\] 
Clearly, $\omega_+$ and $\omega_-$ are different from each other if $\rho + \eta/\rho \leq \rho_*$ 
and $\rho_* \sim 1$ is chosen small enough. 
Hence, $\omega_+$ and $\omega_-$ are two small eigenvalues of $\cL$. 
Lemma~\ref{lem:small_eigenvalues} implies that $\cL$ possesses at most two small eigenvalues,
 thus we have fully described the spectrum of $\cL$ close to zero. 
\end{proof} 

\subsection{Eigenvectors of $\cL$ and proof of Proposition~\ref{pro:stability_operator}}

By Lemma~\ref{lem:eigenvalues_cL}, there are two eigenvalues of $\cL$ in $D_\eps(0)$. The following lemma relates the corresponding eigenvectors to $F_U$ via the eigenvectors of $\cK$ from Lemma~\ref{lem:eigenvectors_C_S_cF}. The eigenvectors of $\cL$ will be perturbations of those of $\cK$. The main mechanism is that the two small eigenvalues of $\cL$ are sufficiently separated, $\abs{\beta- \beta_*} \sim \rho^2$ (cf.\ \eqref{eq:eigenvalues_cL_expansion}). We will use that this separation is much larger than $\rho^3 + \eta\rho$, the effect of the 
perturbation $\cD$ between the unperturbed spectral subspaces $\ran \PK$ and $\ran \QK$ (see \eqref{eq:proof_eigenvectors_cL_aux3} below). 
Hence, regular perturbation theory applies. 
 
Owing to $\cB = \cC_Q \cL \cC_Q^{-1}$ by \eqref{eq:cB_rep_C_F}, we will conclude Proposition~\ref{pro:stability_operator} immediately from this lemma. 

\begin{lemma}[Eigenvectors of $\cL$] \label{lem:eigenvectors_cL} 
There is $\rho_* \sim 1$ such that if $\rho + \eta/\rho \leq \rho_*$ then there are right (left) eigenvectors $L$ and $L_*$ ($\wh{L}$ and $\wh{L}_*$) of $\cL$ 
associated to the eigenvalues $\beta$ and $\beta_*$ from Lemma~\ref{lem:eigenvalues_cL}, respectively, satisfying
\begin{subequations} 
\begin{alignat}{2} 
L & = F_U + 2 \rho \ii F_U^2 (\Re U) + \ord(\rho^2 + \eta/\rho), \qquad \qquad & L_* & = E_-F_U + \ord(\rho^2 + \eta/\rho), \label{eq:expansion_right_eigen_cL} \\ 
\wh{L} & = F_U + \ord(\rho^2 + \eta/\rho), \qquad \qquad & \wh{L}_* & = E_- F_U + \ord(\rho^2 + \eta/\rho).\label{eq:expansion_left_eigen_cL}  
\end{alignat} 
\end{subequations} 
Moreover, $L$, $L_*$, $\wh{L}$ and $\wh{L}_*$ are continuous functions of $\eta$. 
For their scalar products, we have the expansions 
\begin{equation} \label{eq:scalar_products_eigenvectors_cL} 
 \scalar{\wh{L}}{L} = \avg{F_U^2} + \ord(\rho^2 + \eta/\rho), \qquad \qquad \scalar{\wh{L}_*}{L_*} = \avg{F_U^2} + \ord(\rho^2 + \eta/\rho). 
\end{equation}
\end{lemma} 

Before the proof of Lemma~\ref{lem:eigenvectors_cL}, we first conclude Proposition~\ref{pro:stability_operator} from Lemma~\ref{lem:small_eigenvalues}, Lemma~\ref{lem:eigenvalues_cL} and Lemma~\ref{lem:eigenvectors_cL}. 

\begin{proof}[Proof of Proposition~\ref{pro:stability_operator}]
We choose $\eps \sim 1$ as in Lemma~\ref{lem:eigenvalues_cL} and $\rho_* \sim 1$ as in Lemma~\ref{lem:eigenvectors_cL}. 
Since $\cB = \cC_Q \cL \cC_Q^{-1}$ due to \eqref{eq:cB_rep_C_F}, the spectra of $\cB$ and $\cL$ agree. Hence, $\spec(\cB) \cap D_\eps(0) = \{ \beta, \beta_*\}$, 
with $\beta$ and $\beta_*$ as introduced in Lemma~\ref{lem:eigenvalues_cL}. 
From Lemma~\ref{lem:eigenvalues_cL}, \eqref{eq:expansion_lambda_S} and $\avg{F_U^2} \sim 1$ by \eqref{eq:F_U_sim_1}, we obtain
the scaling relations in \eqref{eq:beta_beta_star_scaling} by shrinking $\rho_* \sim 1$ if needed.

We now derive \eqref{eq:expansion_beta} and \eqref{eq:expansion_beta_star}. 
From \eqref{eq:balanced_polar_decomposition}, \eqref{eq:def_F_U} and $\Im U = - \Im U^* = - \Im U^{-1}$ for the unitary operator $U$ (cf.\ Lemma~\ref{lem:scaling_relation_v_u} (i)), we conclude $\psi = \avg{F_U^4}$. 
Moreover, $\psi \sim 1$ due to \eqref{eq:F_U_sim_1}. 
The identity $\cB = \cC_Q \cL \cC_Q^{-1}$ and $Q = Q^*$ also imply $B =\cC_Q[L]$, $B_* = \cC_Q[L_*]$, $\wh{B} = \cC_Q^{-1}[\wh{L}]$ and $\wh{B}_* = \cC_Q^{-1}[\wh{L}_*]$. 
Hence, $\scalar{\wh{B}}{B} = \scalar{\wh{L}}{L}$ and $\scalar{\wh{B}_*}{B_*} = \scalar{\wh{L}_*}{L_*}$ as $Q = Q^*$ (In particular, $\abs{\scalar{\wh{B}}{B}} \sim 1$ and $\abs{\scalar{\wh{B}_*}{B_*}} \sim 1$).
Therefore, the expansions of $\beta$ and $\beta_*$ in Lemma~\ref{lem:eigenvalues_cL}, the expansion of $\kappa$ in \eqref{eq:expansion_lambda_S}, $\avg{F_U^2} \sim 1$ due to \eqref{eq:F_U_sim_1} 
and \eqref{eq:scalar_products_eigenvectors_cL} yield \eqref{eq:expansion_beta} and \eqref{eq:expansion_beta_star}.

The balanced polar decomposition, \eqref{eq:balanced_polar_decomposition}, the definition of \eqref{eq:def_F_U} and $\Im U = - \Im U^* = - \Im U^{-1}$ yield that 
\begin{equation} \label{eq:proof_stab_op_aux1} 
\begin{aligned}
 Q E_\pm F_U Q & = \rho^{-1} E_\pm \Im M, \qquad\qquad  \qquad Q^{-1} E_\pm F_U Q^{-1} = - \rho^{-1} E_\pm \Im M^{-1}, \\ 
 Q F_U^2 (\Re U) Q & = Q F_U QQ^{-1} F_U Q^{-1} Q (\Re U) Q = -\rho^{-2} (\Im M)(\Im M^{-1})(\Re M) . 
\end{aligned} 
\end{equation} 
Since $B =\cC_Q[L]$, $B_* = \cC_Q[L_*]$, $\wh{B} = \cC_Q^{-1}[\wh{L}]$ and $\wh{B}_* = \cC_Q^{-1}[\wh{L}_*]$, 
the expansions in \eqref{eq:expansions_eigenvectors_cB}, thus, follow from Lemma~\ref{lem:eigenvectors_cL}, \eqref{eq:proof_stab_op_aux1} and $Q \sim 1$ in \eqref{eq:scaling_U_Q_M}. 
Moreover, the continuity of $Q$, $Q^{-1}$ and the eigenvectors of $\cL$ from Lemma~\ref{lem:eigenvectors_cL} yield the continuity of the eigenvectors of $\cB$. 

The identity $\cB = \cC_Q \cL\cC_Q^{-1}$ also implies that $\cB^{-1} \cQ = \cC_Q \cL^{-1} \cQ_\cL \cC_Q^{-1}$. Similarly, $(\cB^*)^{-1} \cQ^* = \cC_Q^{-1} (\cL^*)^{-1} \cQ_{\cL}^* \cC_Q$ 
as $Q = Q^*$. 
Hence, the bounds in \eqref{eq:cB_inverse_Q_lesssim_one} follow from \eqref{eq:QT} in Lemma~\ref{lem:small_eigenvalues}. 
This completes the proof of Proposition~\ref{pro:stability_operator}. 
\end{proof} 

The remainder of this subsection is devoted to the proof of Lemma~\ref{lem:eigenvectors_cL}. 

\begin{proof}[Proof of Lemma~\ref{lem:eigenvectors_cL}]
We fix $\lambda \in \{ \beta, \beta_*\}$. 
Since $\beta$ and $\beta_*$ have multiplicity one and together with $\cL$ they are continuous functions of $\eta$ due to Lemma~\ref{lem:eigenvalues_cL} and Lemma~\ref{lem:scaling_relation_v_u} (iii), respectively, 
we find an eigenvector $L'$ of $\cL$ associated to $\lambda$ such that $L'$ is a continuous function of $\eta$ and $\norm{L'} = 1$. 

We apply $\PK$ to the eigenvector relation $\lambda L' = \cL[L']$, use $\cL = \cK + \cD$ and $\PK \cK = \cK \PK = \kappa \PK$ to obtain 
\begin{equation} \label{eq:proof_eigenvectors_cL_aux1} 
 \lambda \PK[L'] = \PK (\cK + \cD)[L'] = \kappa \PK[L'] + \PK \cD\PK[L'] + \PK \cD \QK[L'].  
\end{equation}

We express \eqref{eq:proof_eigenvectors_cL_aux1} in the basis $(K_+,\,K_-)$ of $\ran \PK$ (cf.\ Lemma~\ref{lem:eigenvectors_C_S_cF} \ref{item:K_projections}).
We use that $\PK \cD\PK = \Lambda$ in this basis, where $\Lambda$ is defined as in \eqref{eq:def_Lambda}, and decompose $\PK[L'] = \gamma_+ K_+ + \gamma_- K_-$ for some $\gamma_+, \gamma_- \in \C$. 
This yields 
\begin{equation} \label{eq:proof_eigenvectors_cL_aux3} 
 (\Lambda - \delta)\begin{pmatrix} \gamma_+ \\ \gamma_- \end{pmatrix} = - \begin{pmatrix} \frac{\scalar{\wh{K}_+}{\cD \QK [L']}}{\scalar{\wh{K}_+}{K_+}} \\ 
\frac{\scalar{\wh{K}_-}{\cD \QK [L']}}{\scalar{\wh{K}_-}{K_-}}  \end{pmatrix} 
= -  \begin{pmatrix} \frac{\scalar{\cD^*[\wh{K}_+]}{\QK [L']}}{\scalar{\wh{K}_+}{K_+}} \\ 
\frac{\scalar{\cD^*[\wh{K}_-]}{\QK [L']}}{\scalar{\wh{K}_-}{K_-}}  \end{pmatrix} = \begin{pmatrix} \ord((\rho^3 + \eta \rho) \norm{\PK[L']}) \\ \ord(\eta \rho \norm{\PK[L']}) \end{pmatrix}, 
\end{equation} 
where $\delta \defeq \lambda - \kappa$. 
Here, in the last step, we used Lemma~\ref{lem:expansion_cD_K} to estimate $\cD^*[\wh{K}_\pm]$. 
For the other factor, $\QK[L']$, we use the general eigenvector perturbation result, Lemma~\ref{lem:perturbation_theory_2} in Appendix~\ref{app:perturbation_th}. More precisely, applying $\QK$ to \eqref{eq:pert_th_eigenvector_right}
and using $\QK[K] = \QK \PK[L'] = 0$, we obtain 
$\norm{\QK[L']} \lesssim \norm{\cD} \norm{\PK[L']} \lesssim \rho \norm{\PK[L']}$ since $\norm{\cD} \lesssim \rho$ by \eqref{eq:norm_cK_minus_cL}. For the denominators in \eqref{eq:proof_eigenvectors_cL_aux3}, we use 
that $\abs{\scalar{\wh{K}_s}{K_s}} \sim 1$ for $s \in \{ \pm\}$ by Lemma~\ref{lem:eigenvectors_C_S_cF} and \eqref{eq:F_U_sim_1}.

From $\norm{L'}  = 1$ and \eqref{eq:QT}, we conclude $\norm{\PK[L']} \lesssim 1$. Thus,  $\gamma_\pm = \frac{\scalar{\wh{K}_\pm}{L'}}{\scalar{\wh{K}_\pm}{K_\pm}} = \ord(1)$ as $\abs{\scalar{\wh{K}_\pm}{K_\pm}} \sim 1$ by \eqref{eq:scalar_L_S_R_S} and \eqref{eq:F_U_sim_1}. 
Consequently,  \eqref{eq:proof_eigenvectors_cL_aux3} and Lemma~\ref{lem:expansion_Lambda} imply 
\begin{equation} \label{eq:proof_eigenvectors_cL_aux2} 
\begin{pmatrix} 2 \rho^2 \frac{\avg{F_U^4}}{\avg{F_U^2}} - \delta & 0 \\ 0 & -\delta \end{pmatrix} \begin{pmatrix} \gamma_ + \\ \gamma_- \end{pmatrix} 
= \begin{pmatrix} \ord(\rho^3 + \eta \rho) \\ \ord(\eta \rho) \end{pmatrix}. 
\end{equation} 
In order to compute $\gamma_+$ and $\gamma_-$, we now distinguish the two cases $\lambda = \beta$ and $\lambda=\beta_*$ and apply Lemma~\ref{lem:eigenvalues_cL} to estimate $\delta$. 
If $\lambda = \beta$ then $\abs{\delta} \sim \rho^2$ by Lemma~\ref{lem:eigenvalues_cL} and \eqref{eq:F_U_sim_1}. 
Hence, \eqref{eq:proof_eigenvectors_cL_aux2} implies $\abs{\gamma_-} \lesssim \eta/\rho$. 
Thus, $\abs{\gamma_+} \sim 1$ as $\abs{\gamma_+}\norm{K_+} \geq \norm{L'} -\abs{\gamma_-} \norm{K_-} - \norm{\QK[L']}$, $\abs{\gamma_-} \lesssim \eta/\rho $, $\norm{\QK[L']} \lesssim \rho$ and $\norm{L'} 
 =  1 \sim \norm{K_\pm}$. 
In particular, $L \defeq L' \scalar{\wh{K}_+}{K_+}/\scalar{\wh{K}_+}{L'} = L' /\gamma_+$ is continuous in $\eta$ and 
\[ \PK[L] = K_+ + \ord(\eta/\rho) = F_U + \ord(\eta/\rho), \] 
where we used \eqref{eq:eigenvectors_1_C_S_F} in the last step. 
We now apply Lemma~\ref{lem:perturbation_theory_2} to compute $\QK[L]$ with $K = \PK[L] = F_U + \ord(\eta/\rho)$. 
From \eqref{eq:cD_F_U} below, we obtain $(\cK - \kappa)^{-1}\QK \cD[F_U] = -2 \rho \ii F_U^2 (\Re U) + \ord(\rho^2 + \eta)$ since $\QK$ and $\cK$ agree with the identity map on $\oM$ and $F_U^2 (\Re U) \in \oM$. 
Hence, Lemma~\ref{lem:perturbation_theory_2} and $\norm{\cD} \lesssim \rho$ directly imply the expansion of $L$ in \eqref{eq:expansion_right_eigen_cL}.

We now consider the case $\lambda = \beta_*$. 
Lemma~\ref{lem:eigenvalues_cL} with $\lambda = \beta_*$ implies $\abs{\delta} \lesssim \rho^3 + \eta \rho$ and, thus, $\abs{2\rho^2 \frac{\avg{F_U^4}}{\avg{F_U^2}} - \delta} \sim \rho^2$. Hence, $\abs{\gamma_+} \lesssim \rho + \eta/\rho$ 
and, similarly to the other case, we set $L_* \defeq L'/\gamma_-$ and conclude 
\[ \PK[L_*] = E_- F_U + \ord(\rho + \eta/\rho).  \] 

Owing to \eqref{eq:pert_th_eigenvector_right} in Lemma~\ref{lem:perturbation_theory_2} with $K = \PK[L_*]$, we have 
\begin{equation} \label{eq:proof_eigenvectors_cL_aux4} 
 \QK[L_*] = - (\cK-\kappa)^{-1} \QK \cD \PK[L_*] + \ord(\rho^2) = - (\cK -\kappa)^{-1} \QK \cD[E_-F_U] + \ord(\rho^2 + \eta) = \ord(\rho^2 + \eta), 
\end{equation}
where the last step follows from \eqref{eq:cD_Emin_F_U} below. 
Therefore, the second identity in \eqref{eq:proof_eigenvectors_cL_aux3}, Lemma~\ref{lem:expansion_Lambda}, Lemma~\ref{lem:expansion_cD_K} and $\rho\gtrsim \eta$ by \eqref{eq:scaling_rho} imply
\[  
\begin{pmatrix} 2 \rho^2 \frac{\avg{F_U^4}}{\avg{F_U^2}} - \delta & 0 \\ 0 & -\delta \end{pmatrix} \begin{pmatrix} \gamma_ + \\ \gamma_- \end{pmatrix} 
= \begin{pmatrix} \ord(\rho^4 + \eta \rho) \\ \ord(\eta \rho) \end{pmatrix}. 
\] 
As $\abs{2\rho^2 \frac{\avg{F_U^4}}{\avg{F_U^2}} - \delta} \sim \rho^2$, we conclude $\abs{\gamma_+} \lesssim \rho^2 + \eta/\rho$. 
Hence, $\PK[L_*] = E_-F_U + \ord(\rho^2 + \eta/\rho)$ and the expansion of $L_*$ in \eqref{eq:expansion_right_eigen_cL} follows from \eqref{eq:proof_eigenvectors_cL_aux4}. 

A completely analogous argument starting from $\cL^* [\wh{L}] = \bar \lambda \wh{L}$ yields the expansions of $\wh{L}$ and $\wh{L}_*$ in \eqref{eq:expansion_left_eigen_cL}. 
We leave the details to the reader. 
From \eqref{eq:expansion_right_eigen_cL} and \eqref{eq:expansion_left_eigen_cL}, we directly obtain \eqref{eq:scalar_products_eigenvectors_cL} since $F_U^3(\Re U) \in \oM$ implies
$\avg{F_U^3 (\Re U)} = 0$. 
This completes the proof of Lemma~\ref{lem:eigenvectors_cL}. 
\end{proof}

\subsection{Proofs of the auxiliary Lemmas~\ref{lem:expansion_Lambda} and \ref{lem:expansion_cD_K}}
\label{subsec:proofs_aux_results} 

In this section, we show Lemma~\ref{lem:expansion_Lambda} and Lemma~\ref{lem:expansion_cD_K} which were both stated in Subsection~\ref{subsec:eigenvalues_cL}. 

\begin{proof}[Proof of Lemma~\ref{lem:expansion_cD_K}] 
Lemma~\ref{lem:expansion_cD_K} follows directly from Lemma~\ref{lem:eigenvectors_C_S_cF} and 
the precise expansions of $\cD[E_\pm F_U]$ and $\cD^*[E_\pm F_U]$ established in 
\eqref{eq:cD_F_U_all_cases} below.  

In the following computations, we constantly use that $P$, $U$, $U^*$ and $F_U$ commute with each other.  
From \eqref{eq:F_U_eigenvector_cF}, \eqref{eq:P_minus_U_lesssim_rho} and 
$1-U^2 = -2 \ii \Re U \Im U + 2 (\Im U)^2$, we obtain 
\begin{subequations} \label{eq:cD_F_U_all_cases}
\begin{equation} \label{eq:cD_F_U}
 \cD[F_U] = (\cC_P - \cC_{U}) F_U + \ord(\eta) = F_U (1-U^2) + \ord(\eta)= - 2 \rho \ii F_U^2 (\Re U) + 2\rho^2 F_U^3 + \ord(\eta). 
\end{equation}
Similarly, \eqref{eq:F_U_eigenvector_cF}, \eqref{eq:P_minus_U_lesssim_rho}, \eqref{eq:P_Emin} and $UE_-F_UU = - E_- U ^* F_U U = - E_- F_U$ by \eqref{eq:U_Emin} imply 
\begin{equation} \label{eq:cD_Emin_F_U}
 \cD[E_-F_U] = - (\cC_P - \cC_U)E_-F_U + \ord(\eta)= E_-F_U  - E_-F_U + \ord(\eta) =\ord(\eta). 
\end{equation}
Since $1-(U^*)^2 = 2 \ii \Re U \Im U + 2 (\Im U)^2$, $F_U^2 (\Re U) \in \oM$ by Lemma~\ref{lem:scaling_relation_v_u} (i) and $\cF$ vanishes on $\oM$, we get  
\begin{equation} \label{eq:cD_star_F_U}
 \cD^*[F_U] = \cF( \cC_P - \cC_{U^*})[F_U] = \cF[F_U ( 1 - (U^*)^2)] = 2 \ii \rho 
\cF[F_U^2 (\Re U)] + 2 \rho^2 \cF[F_U^3]  = 2 \rho^2 \cF[F_U^3].  
\end{equation} 
From $P E_- F_U P = - E_-F_U = U^* E_- F_U U^*$ by \eqref{eq:P_Emin} and \eqref{eq:U_Emin}, we deduce 
\begin{equation} \label{eq:cD_star_Emin_F_U}
 \cD^*[E_- F_U] = \cF[-E_- F_U + E_-F_U] = 0. 
\end{equation}
\end{subequations} 
This completes the proof of Lemma~\ref{lem:expansion_cD_K}. 
\end{proof} 

\begin{proof}[Proof of Lemma~\ref{lem:expansion_Lambda}]
We first show that, for all $s_1, s_2 \in \{ \pm \}$, we have 
\begin{equation} \label{eq:proof_expansion_Lambda_reduction_to_leading_term} 
 \scalar{\wh{K}_{s_1}}{\cD[K_{s_2}]} = \scalar{E_{s_1} F_U}{\cD[E_{s_2} F_U]} + \ord(\eta\rho + \eta^2).  
\end{equation}
In fact, it is easy to see that \eqref{eq:proof_expansion_Lambda_reduction_to_leading_term} follows from 
$K_\pm, \wh{K}_\pm \in \dM$ and \eqref{eq:eigenvectors_1_C_S_F} in Lemma~\ref{lem:eigenvectors_C_S_cF} as well as 
\begin{equation} \label{eq:expansion_Lambda_error_term} 
 \scalar{R_1}{(\cC_P - \cC_U)[R_2]} = \ord(\rho^2) 
\end{equation}
for all $R_1, R_2 \in \dM$ satisfying $\norm{R_1}, \norm{R_2} \lesssim 1$.  
For the proof of \eqref{eq:expansion_Lambda_error_term}, we expand $U = \Re U + \ii \Im U$ and obtain 
\[ \scalar{R_1}{(\cC_P- \cC_U)[R_2]} = \scalar{R_1}{(\cC_P - \cC_{\Re U})[R_2]} - \ii \scalar{R_1}{\Im U R_2 \Re U + \Re U R_2 \Im U} + \ord(\rho^2)  = \ord(\rho^2). \] 
Here, we used in the first step that $\Im U = \ord(\rho)$. For the second step, we noted that the first term is $\ord(\rho^2)$ due to \eqref{eq:P_minus_U_lesssim_rho} and the second term vanishes as $\Re U \in \oM$ while 
$R_1, R_2, \Im U \in \dM$ due to Lemma~\ref{lem:scaling_relation_v_u} (i). 

What remains is computing $\scalar{E_{s_1}F_U}{\cD[E_{s_2} F_U]} = \scalar{\cD^*[E_{s_1}F_U]}{E_{s_2}F_U}$. 
From \eqref{eq:cD_star_F_U} and \eqref{eq:F_U_eigenvector_cF}, we obtain 
\[ \scalar{\cD^*[F_U]}{F_U} = 2 \rho^2 \avg{F_U^4} + \ord(\eta \rho), \qquad 
\scalar{\cD^*[E_-F_U]}{F_U} = - 2 \rho^2 \avg{F_U^4 E_-} + \ord(\eta \rho) = \ord(\eta \rho), \] 
where we used in the very last step that $\avg{F_U^4 E_-} = 0$ since the diagonal $n$-vector components of $F_U$ 
are identical due to \eqref{eq:def_F_U}. 
Moreover, \eqref{eq:cD_star_Emin_F_U} directly implies that $\scalar{\cD^*[E_-F_U]}{E_\pm F_U} = 0$. 

Hence, owing to \eqref{eq:proof_expansion_Lambda_reduction_to_leading_term}, we deduce  
\[ \Lambda = 2\rho^2 \begin{pmatrix} \frac{\avg{F_U^4}}{\avg{F_U^2}} & 0 \\ 0 &0 \end{pmatrix} + \ord(\eta \rho + \eta^2). \] 
Using that $\rho \gtrsim \eta$ due to \eqref{eq:scaling_rho} completes the proof of Lemma~\ref{lem:expansion_Lambda}. 
\end{proof}

\subsection{Derivatives of $M$} 

As a first application of our analysis of the stability operator $\cB$ we show the following bound on the derivatives of $M$, the solution to the MDE, \eqref{eq:mde}, with respect to $\eta$, $z$ and $\bar z$.

\begin{lemma}[Bounds on derivatives of $M$] There is $\rho_* \sim 1$ such that $\rho + \eta/\rho \leq \rho_*$ implies
\begin{equation} \label{eq:pt_eta_M_bound} 
 \normb{\pt_\eta M} + \normb{\pt_z M} + \normb{\pt_{\bar z} M} \lesssim \frac{1}{\rho^2 + \eta/\rho}. 
\end{equation}
\end{lemma}

\begin{proof} 
We only show the bound on $\pt_\eta M$. The estimates on $\pt_z M$ and $\pt_{\bar z} M$ are shown analogously. 

If $\rho_* \sim 1$ is chosen small enough and $\rho + \eta/\rho \leq \rho_*$ then $\cB$ is invertible due to Proposition~\ref{pro:stability_operator}.  
Thus, applying the implicit function theorem to \eqref{eq:mde} yields that $M$ is differentiable with respect to $\eta$ and $\pt_\eta M = \ii \cB^{-1}[M^2]$. 
Hence, by Proposition~\ref{pro:stability_operator}, we have 
\begin{equation} \label{eq:aux_identity_derivative_M} 
 -\ii \pt_\eta M = \frac{\scalar{\wh{B}}{M^2}}{\beta\scalar{\wh{B}}{B}} B + \frac{\scalar{\wh{B}_*}{M^2}}{\beta_*\scalar{\wh{B}_*}{B_*}} B_* + \cB^{-1} \cQ [M^2]. 
\end{equation}
Moreover, differentiating $\scalar{E_-}{M} = 0$, which holds due to \eqref{eq:scalar_E_minus_M_equals_0}, with respect to $\eta$ yields 
\[ \scalar{E_-}{\pt_\eta M} = 0. \] 
Hence, we apply $\scalar{E_-}{\genarg}$ to \eqref{eq:aux_identity_derivative_M}
and obtain 
\begin{equation} \label{eq:proof_derivative_M_aux1} 
  \frac{\scalar{\wh{B}_*}{M^2}}{\beta_*\scalar{\wh{B}_*}{B_*}} \scalar{E_-}{B_*} = - \frac{\scalar{\wh{B}}{M^2}}{\beta\scalar{\wh{B}}{B}} \scalar{E_-}{B} + \scalar{E_-}{\cB^{-1}\cQ[M^2]}. 
\end{equation}
The right-hand side of \eqref{eq:proof_derivative_M_aux1} is $\ord(1)$ since $\wh{B}$, $M$ and $\cB^{-1} \cQ [M^2]$ are bounded, $\abs{\scalar{\wh{B}}{B}} \sim 1$ and 
$\abs{\scalar{E_-}{B}} \lesssim \abs{\beta}$ by \eqref{eq:beta_beta_star_scaling} and \eqref{eq:bound_scalar_E_minus_B}.   
Since $\abs{\scalar{E_-}{B_*}} \sim 1$ and $\norm{B_*} \lesssim 1$ by \eqref{eq:expansion_B_star}, 
we conclude that the second term on the right-hand side of \eqref{eq:aux_identity_derivative_M} is $\ord(1)$. 
Thus, the leading term to $\pt_\eta M$ comes from the first term on the right-hand side of \eqref{eq:aux_identity_derivative_M} 
which can be bounded by $(\rho^2 + \eta/\rho)^{-1}$ due to the boundedness of $\wh{B}$, $M$ and $B$ 
as well as $\abs{\scalar{\wh{B}}{B}} \sim 1$ and \eqref{eq:beta_beta_star_scaling}. 
\end{proof}

\subsection{Existence and properties of $\sigma$} 

The next proposition shows the existence of a probability density $\sigma$ satisfying \eqref{eq:sigma_equal_laplace_potential}. 
In particular, the logarithmic potential of the probability measure $\sigma(z) \dd^2 z$ is given 
by the function $-2\pi L$, where, for any $z \in \C$, $L(z)$ is defined through   
\begin{equation} \label{eq:def_L} 
 L(z) \defeq -\frac{1}{2\pi} \int_0^\infty \bigg( \avg{\Im M(z,\eta)} - \frac{1}{1 + \eta} \bigg) \, \dd \eta. 
\end{equation}

Throughout this subsection, we use the normalization $\varrho(\smallS) = 1$ (see \eqref{eq:normalization_spectral_radius}). 

\begin{proposition}[Properties of $\sigma$] \label{pro:properties_sigma} 
There exists an integrable function $\sigma \colon \C \to [0,\infty)$ such that $\sigma = \Delta_z L$ in the
sense of distributions. 
Moreover, $\sigma$ can be chosen to satisfy the following properties:  
\begin{enumerate}[label=(\roman*)]
\item The function $\sigma$ is a probability density on $\C$ with respect to $\dd^2z$. 
\item For all $z\in \C$ with $\abs{z} >1$, we have $\sigma(z) = 0$. 
\item The restriction $\sigma|_{D(0,1)}$ is infinitely often continuously differentiable and 
$\sigma(z) \sim 1$ uniformly for $z \in D(0,1)$. 
\end{enumerate} 

\end{proposition} 

The next lemma will directly imply Proposition~\ref{pro:properties_sigma} and be proved after the proof of this proposition. 

\begin{lemma}[Properties of $L$] \label{lem:properties_L} 
The function $L(z)$ from \eqref{eq:def_L} has the following properties: 
\begin{enumerate}[label=(\roman*)]
\item The integrand on the right-hand side of \eqref{eq:def_L} is Lebesgue-integrable for every $z \in \C$. 
\item $L$ is a rotationally symmetric function and 
continuously differentiable with respect to $z$ and $\bar z$ on $\C$. 
\item $\Delta_z L(z)$ exists for all $z \in\C$ satisfying $\abs{z} \neq 1$. 
\end{enumerate} 
\end{lemma} 

\begin{proof}[Proof of Proposition~\ref{pro:properties_sigma}]  
By Lemma~\ref{lem:properties_L} (iii), we know that $\Delta_z L$ defines a function on $\C \setminus \{ z \in \C \colon \abs{z} = 1 \}$. 
We set $\mathbf v(z,\eta)=(v_1(z,\eta),v_2(z,\eta))$ and remark that $\avg{\Im M} = \avg{\mathbf v}$. 
Moreover, \cite[Eq.~(4.2)]{Altcirc} implies that $\mathbf v$ has a smooth extension to $\C \setminus \overline{D(0,1)} \times [0,\infty)$ and 
$\lim_{\eta \downarrow 0} \mathbf v(z,\eta) = 0$ for $\abs{z} > 1$ due to \eqref{eq:scaling_rho}. 
Therefore, we can follow the proof of \cite[Eq.~(4.10)]{Altcirc} and obtain  
%\begin{equation} \label{eq:integral_sigma_equal_0} 
 \[ \int_{\tau_0 \leq \abs{z}^2 \leq \tau_1} \Delta_z L(z) \dd^2 z = 0  \] 
%\end{equation}
for all $\tau_1 > \tau_0 >1$. As $\tau_1$ and $\tau_0$ are arbitrary, we conclude $\sigma(z) = \Delta_z L(z) =0$ if $\abs{z}> 1$. 

Let $f \in C_0^\infty(\C)$ be a smooth function with compact support. We compute 
\begin{equation} \label{eq:integration_by_parts_L_Delta_f} 
 \int_\C L(z) \Delta_z f(z) \dd^2 z = \int_{D(0,1)} L(z) \Delta_z f(z) \dd^2 z + \int_{\C \setminus \overline{D(0,1)}} L(z) \Delta_z f(z) \dd^2 z = \int_{\C} \sigma(z) f(z) \dd^2 z. 
\end{equation}
Here, we moved $\Delta_z$ to $L$ in the second step and used that $\sigma(z) = \Delta_z L(z) = 0$ for $\abs{z}>1$ and that the boundary terms cancel each other due to the continuity of $L$,  $\pt_z L$ and 
$\pt_{\bar z} L$ from Lemma~\ref{lem:properties_L} (ii). 
Since $\sigma(z) = 0$ for $\abs{z}>1$, setting $\sigma(z) = 0$ if $\abs{z}=  1$ and using \cite[Proposition~2.4]{Altcirc} for the remaining properties complete the proof of Proposition~\ref{pro:properties_sigma} .  
\end{proof} 

In the proof of Lemma~\ref{lem:properties_L}, we will make use of the following lemma whose proof we postpone until the end of this section. 

\begin{lemma}\label{lem:derivatives_M_controlled_by_eta} 
Uniformly for $z \in \C$ and $\eta >0$, we have 
\begin{align} 
\normbb{M(z,\eta) - \frac{\ii}{1 + \eta}} & \lesssim \frac{1}{1+\eta^{2}}, \label{eq:approximating_M_large_eta} \\  
 \norm{\pt_z M(z,\eta)} + \norm{\pt_{\bar z} M(z,\eta)} & \lesssim \frac{1}{\eta^{2/3} +  \eta^{2}}. 
\label{eq:control_derivative_M} 
\end{align}
\end{lemma} 

\begin{proof}[Proof of Lemma~\ref{lem:properties_L}] 
The assertion in (i) follows immediately from \eqref{eq:approximating_M_large_eta}. 
Moreover, since $M(z,\eta)$ is continuously differentiable with respect to $z$ and $\bar z$, the bound \eqref{eq:control_derivative_M} implies (ii). 
In \cite[Proposition~2.4(i)]{Altcirc}, it was shown that $\Delta_z \avg{\Im M(z,\eta)} = \Delta_z \avg{\mathbf v(z,\eta)} = 4 \avg{\tau \pt_\tau^2 \boldsymbol{v}^\tau + \pt_\tau \boldsymbol{v}^\tau}|_{\tau = \abs{z}^2}$ (the last equality uses the notation of \cite{Altcirc}) 
is integrable in $\eta>0$ on $[0,\infty)$ for all $z \in \C$ with $\abs{z} < 1$. 
Completely analogously, the integrability can be shown if $\abs{z} > 1$.  
This shows that $\Delta_z$ and the $\eta$-integral can be interchanged which proves (iii).
\end{proof}

\begin{proof}[Proof of Lemma~\ref{lem:derivatives_M_controlled_by_eta}] 
From \eqref{eq:mde}, it is easy to get $\norm{M} \leq \eta^{-1}$ for all $z \in \C$ and $\eta >0$. As in the proof of Lemma~\ref{lem:scaling_relation_v_u}, we see that, 
uniformly for $\abs{z} \geq 2$ or $\eta \geq 1$, we have 
\begin{equation} \label{eq:scaling_relations_large_eta_z} 
\rho \sim \frac{\eta}{\abs{z}^2 + \eta^2}, \qquad v_1 \sim v_2 \sim \rho, \qquad u \sim \frac{1}{\abs{z}^2 + \eta^2}, \qquad \Im U \sim \frac{\eta}{\abs{z} + \eta } , \qquad Q \sim \frac{1}{\abs{z}^{1/2} + \eta^{1/2}}. 
\end{equation}
In particular, $\norm{M} \lesssim (1 + \eta)^{-1}$ and \eqref{eq:approximating_M_large_eta} follows from multiplying \eqref{eq:mde} by $\ii \eta^{-1} M$ 
and using $\norm{\cal{S}} \lesssim 1$. 

We remark that \eqref{eq:control_derivative_M} follows from \eqref{eq:pt_eta_M_bound} and \eqref{eq:scaling_rho} as well as similarly to the proof of \cite[Eq.~(4.2)]{Altcirc} 
if $\abs{z} \leq 2$ and $\eta \leq 1$. 
Hence, we assume $\abs{z} \geq 2$ or $\eta \geq 1$ in the remainder of the proof. 
In this regime, we obtain $1 - \normtwoop{\cF} \sim 1$ by following the proof of \eqref{eq:normtwo_cF} and using 
\eqref{eq:scaling_relations_large_eta_z}.  
Therefore, \eqref{eq:cB_rep_C_F}, $\normtwoinf{\cS} \lesssim 1$ (cf.\ the proof of Lemma~\ref{lem:small_eigenvalues}) and \eqref{eq:scaling_relations_large_eta_z} imply $\norm{\cB^{-1}} \lesssim 1$. 
We differentiate \eqref{eq:mde} with respect to $z$ and, thus, obtain 
\[ \norm{\pt_z M(z,\eta)} = \normbb{\cB^{-1}\bigg[M \begin{pmatrix} 0 & 1 \\ 0 & 0 \end{pmatrix} M\bigg] } \lesssim \eta^{-2}. \] 
Together with a similar argument for $\pt_{\bar z} M$ this completes the proof of Lemma~\ref{lem:derivatives_M_controlled_by_eta}. 
\end{proof}

\section{Cubic equation associated to stability equation of the MDE}  \label{sec:cubic_equation} 

In this section, we study specific perturbations to the Matrix Dyson equation (MDE). 
Throughout this section, $M$ is the solution to the unperturbed MDE, \eqref{eq:mde}. We consider solutions $G$ of the perturbed MDE, \eqref{eq:perturbed_mde}, for some $D \in\C^{2n \times 2n}$ with the additional constraint $\scalar{E_-}{G} = 0$, keeping in mind that in our application the resolvent $G = (H_z - \ii\eta)^{-1}$ (see \eqref{eq:scalar_E_minus_G_equals_0} below) satisfies this constraint. 
Since $D$ is small, we need to study the stability of the MDE, \eqref{eq:mde}, under a small perturbation. 
The linear stability operator of this perturbation problem is $\cB=1-\cC_M \cS$ (see \eqref{eq:def_cB}).
When $\rho$ is small, the inverse of $\cB$  blows up, hence we need to expand to the next order, i.e. study the quadratic stability equation, \eqref{eq:stability_equation}. 
The following proposition describes the stability properties 
of \eqref{eq:stability_equation} in this regime. 
In fact, the difference $G-M$ is dominated by the contribution $\Theta \defeq \scalar{\wh{B}}{G-M}/\scalar{\wh{B}}{B}$ of $G-M$ in the unstable direction $B$ of $\cB$ (cf.\ Proposition~\ref{pro:stability_operator}). 
The scalar quantity $\Theta$ satisfies a cubic equation. In order to control $G-M$, we will control this cubic equation via a bootstrapping argument 
in $\eta$ in Section~\ref{sec:local_law_H} below.

In the following proposition and the rest of the paper, we use a special matrix norm to estimate the distance between $G$ and $M$. We denote this norm by $\normstar{\genarg}$. 
It is slightly modified compared of those defined in \cite{AltEdge,Cusp1}
in order to account for the additional unstable direction of $\cB$. 
Such norms are tailored to local law proofs by the cumulant method and have first appeared in \cite{Erdos2017Correlated}. 
We need some auxiliary notations in order to define $\normstar{\genarg}$. 
For $a,b,c,d \in [2n]$, we set  
$\kappa_c(ab, cd) \defeq \delta_{ad}\delta_{bc} s_{ab}$ and $\kappa_d(ab,cd) \defeq \delta_{ac} \delta_{bd} t_{ab}$, 
where $s_{ab} \defeq \E \abs{w_{ab}}^2$ and $t_{ab} \defeq \E w_{ab}^2$ with $W = (w_{ab})_{a,b \in [2n]}\defeq H_z - \E H_z$ (compare \eqref{eq:def_H_z} for the definition of $H_z$). 
Moreover, for a vector $\bx =(x_a)_{a \in [2n]}$, we write 
$\kappa_c(\bx b, cd) = \sum_a x_a \kappa_c(ab,cd)$ for a weighted version of $\kappa_c$. We use an analogous convention for $\kappa_d$. 
If we replace an index of a scalar quantity by a dot ($\cdot$) then this denotes the corresponding vector, where the omitted index runs through $[2n]$, e.g. $R_{a \cdot}$ denotes the vector $(R_{ab})_{b \in [2n]}$. 
For an $2n\times 2n$ matrix $R$ and vectors $\bx, \by\in \C^{2n}$, we use the short-hand notation 
$R_{\bx\by}$ to denote the quadratic form $\scalar{\bx}{R\by}$ and $R_{\bx a}$ to denote $\scalar{\bx}{R\mathbf e_a}$,  where $\mathbf e_a$ is the $a$-th normalized standard basis vector. 
With these conventions, we define some sets of testvectors. 
For fixed vectors $\bx, \by \in \C^{2n}$, we define 
\[ \begin{aligned} 
I_0 & \defeq \{ \bx, \by \} \cup \{ \mathbf e_a, (\wh{B}^*)_{a \cdot}, ((\wh{B}_*)^*)_{a \cdot} \colon a \in [2n] \}, \\ 
I_{k+1}  & \defeq I_k \cup \{ M \bu \colon \bu \in I_k \} \cup \{ \kappa_c((M\bu) a, b \cdot ),\kappa_d( (M\bu)a, b \cdot) \colon \bu \in I_k, a,b \in [2n] \} . 
\end{aligned} \] 
We now introduce the $\normstar{\genarg}$-norm defined by 
\[ \normstar{R} \defeq \normstar{R}^{K,\bx,\by} \defeq \sum_{0 \leq k < K} (2n)^{-k/2K} \norm{R}_{I_k}+ (2n)^{-1/2} \max_{\bu \in I_K} \frac{\norm{R_{\cdot \bu}}}{\norm{\bu}}, \qquad \norm{R}_I \defeq \max_{\bu,\bv \in I} \frac{\abs{R_{\bu \bv}}}{\norm{\bu}\norm{\bv}}. \] 
We remark that the norm $\normstar{\genarg}$ depends on $\eta$ and $z$ via $M = M(z,\eta)$. 
However, this will not play any important role in our arguments. 

In this section, the model parameters for the comparison relation $\lesssim$ are given by $s_*$ and $s^*$ from \eqref{eq:condition_flatness} as well as $\tau$ from the upper bound on $\abs{z}$. 

\begin{proposition}[Cubic equation for $\Theta$] \label{pro:cubic_equation} 
There is $\rho_* \sim 1$ such that if $\rho + \eta/\rho \leq \rho_*$ for some fixed $z \in D_\tau(0)$ and $\eta \in (0,1]$ then the following holds. 
We fix $K \in \N$, $\bx, \by \in \C^{2n}$ and set $\normstar{\genarg} \defeq \normstar{\genarg}^{K,\bx, \by}$. 
If $G$ and $D$ satisfy \eqref{eq:perturbed_mde}, $\scalar{E_-}{G} = 0$ and $\normstar{G-M} + \normstar{D} \lesssim n^{-30/K}$ then 
\begin{equation} \label{eq:G_minus_M_expansion_via_Theta_and_D} 
G - M = \Theta B - \cB^{-1} \cQ[MD] +  \Theta^2 \cB^{-1} \cQ [M \cS[B]B] + E,
\end{equation}
where $\Theta \defeq \scalar{\wh{B}}{G-M}/\scalar{\wh{B}}{B}$ and the error matrix $E$ has the upper bound 
\begin{equation} \label{eq:normstar_E} 
\normstar{E} \lesssim \abs{\Theta}^2 (\rho + \eta/\rho) + n^{16/K}\big(\abs{\Theta}^3 + \abs{\Theta}(\normstar{D} + \rho^2 + \eta/\rho) + \normstar{D}^2 + \abs{\scalar{R_1}{D}} \big) 
\end{equation}
with $R_1 \defeq M^* (\cB^{-1} \cQ)^*[E_-]$,  thus, the $2n\times 2n$-matrix $R_1$ does not depend on  $G$ and $D$ and satisfies $\norm{R_1} \lesssim 1$. 
 
Moreover, $\Theta$ fulfils the approximate cubic equation 
\begin{equation} \label{eq:cubic_equation} 
 \Theta^3 + \xi_2 \Theta^2 + \xi_1 \Theta = \eps_* 
\end{equation}
whose coefficients $\xi_2$ and $\xi_1$ satisfy the scaling relations 
\begin{equation} \label{eq:scaling_xi_2_xi_1} 
 \abs{\xi_2} \sim \rho, \qquad \abs{\xi_1} \sim \eta/\rho + \rho^2 
\end{equation} 
and the error term $\eps_*$ is bounded by 
\begin{equation} \label{eq:bound_eps_star} 
 \abs{\eps_*} \lesssim n^{62/K} \big( \normstar{D}^3 + \abs{\scalar{R_1}{D}}^{3/2} + \abs{\scalar{R_2}{D}}^{3/2} \big) + \abs{\scalar{\wh{B}}{MD}} + \abs{\scalar{\wh{B}}{M (\cS \cB^{-1} \cQ[MD])(\cB^{-1} \cQ [MD])}}. 
\end{equation} 
Here, the matrix $R_2\in \C^{2n \times 2n}$  does not depend on  $G$ and $D$ and satisfies $\norm{R_2} \lesssim 1$. 
\end{proposition} 

We note that $R_2$ has an explicit definition (see \eqref{eq:def_R_2} below) but its exact form will not be important.

\begin{proof} 
The proof follows from an application of Lemma~\ref{lem:cubic_equation_abstract} to \eqref{eq:stability_equation} with the choices 
$\cA[R,T]\defeq \frac{1}{2} M( \cS[R] T + \cS[T]R)$, $\cB=1-\cC_M \cS$ as in \eqref{eq:def_cB}, $Y = G - M$ and ${ Z} = MD$ in Lemma~\ref{lem:cubic_equation_abstract}. 
Note that $\scalar{E_-}{G-M} = 0$ by assumption and \eqref{eq:scalar_E_minus_M_equals_0}. 
We first check the conditions of Lemma~\ref{lem:cubic_equation_abstract} in \eqref{eq:conditions_abstract_cubic} with $\norm{\genarg} \equiv \normstar{\genarg}$ and $\lambda \defeq n^{1/2K}$. Partly, they will be a consequence of the bounds 
\begin{subequations}  \label{eq:aux_bounds_star_norm} 
\begin{align} 
&\normstar{M \cS[R]T} \lesssim n^{1/2K} \normstar{R} \normstar{T}, \quad \normstar{MR} \lesssim n^{1/2K} \normstar{R}, \quad \normstarop{\cQ} \lesssim 1, \quad \normstarop{\cB^{-1} \cQ} \lesssim 1, \\ 
&\abs{\scalar{\wh{B}}{R}} + \abs{\scalar{\wh{B}_*}{R}} + \abs{\scalar{E_-}{R}} \lesssim \normstar{R}, \label{eq:bounds_star_norm_scalar_products} 
\end{align} 
\end{subequations}
for all $R, T \in \C^{2n\times 2n}$. The proof of \eqref{eq:aux_bounds_star_norm} is very similar to the one of \cite[Lemma~3.4]{AltEdge}. 
 Since \eqref{eq:bounds_star_norm_scalar_products} does not have a counterpart in \cite[Lemma~3.4]{AltEdge}, we provide the details at the end of this section.  

Owing to \eqref{eq:expansion_B} and \eqref{eq:expansion_B_star}, we have 
$\normstar{B} + \normstar{B_*}\lesssim \norm{B} + \norm{B_*} \lesssim 1$.  
The third, sixth and ninth term in \eqref{eq:conditions_abstract_cubic} are $\sim 1$ by Proposition~\ref{pro:stability_operator}.  
This completes the proof of \eqref{eq:conditions_abstract_cubic} with $\lambda \defeq n^{1/2K}$. 

Therefore, Lemma~\ref{lem:cubic_equation_abstract} with $\delta \defeq n^{-81/4K}$, $\abs{\Theta} \lesssim \normstar{G-M} \lesssim n^{-30/K}$ and $\normstar{M D} \lesssim n^{1/2K} \normstar{D}$ imply 
\begin{equation} \label{eq:proof_cubic_stability_aux1}  
\begin{aligned} 
& \mu_3 \Theta^3 + \mu_2 \Theta^2 - \beta \scalar{\wh{B}}{B} \Theta = -\mu_0 + \scalar{R_2}{D} \Theta  \\ 
 & \quad \qquad + \ord\big( n^{-1/4K} \abs{\Theta}^3 + n^{62/K}( \normstar{D}^3 + \abs{\scalar{R_1}{D}}^{3/2}) 
 + n^{20/K}\abs{\Theta}^2( \abs{\scalar{E_-}{B}}^2 + \abs{\scalar{\wh{B}}{\cA[B,B_*]}}^2) \big), 
\end{aligned} 
\end{equation}  
where $\mu_3$, $\mu_2$ and $\mu_0$ are defined as in \eqref{eq:coefficients_abstract_cubic}, $R_1 = M^* (\cB^{-1} \cQ)^*[E_-]$
and we introduced 
\begin{equation} \label{eq:def_R_2} 
 R_2 \defeq  M^* (\cB^{-1} \cQ)^* \bigg[ \cS[B^*]M^* \wh{B} + \cS[ M^* \wh{B} B^*] - \frac{\scalar{B}{E_-}}{\scalar{B_*}{E_-}} \Big( \cS[(B_*)^*] M^* \wh{B} + \cS[ M^* \wh{B} (B_*)^*] \Big) \bigg].  
\end{equation}
Note that $R_1$ and $R_2$ are independent of $G$ and $D$ and satisfy $\norm{R_1} \lesssim 1$ and $\norm{R_2} \lesssim 1$ due to Proposition~\ref{pro:stability_operator} and \eqref{eq:scaling_U_Q_M}.

The remaining task is expanding $\mu_3$, $\mu_2$, $-\beta\scalar{\wh{B}}{B}$ and $\scalar{\wh{B}}{\cA[B,B_*]}$ on the right-hand side of \eqref{eq:proof_cubic_stability_aux1} with the help of 
Proposition~\ref{pro:stability_operator}. 
The coefficient $-\beta\scalar{\wh{B}}{B}$ has already been identified in \eqref{eq:expansion_beta}. 
For the others, we will rewrite the expansions in Proposition~\ref{pro:stability_operator} in terms of $U$, $Q$ and $F_U=\rho^{-1} \Im U$ defined in \eqref{eq:def_U_Q} and \eqref{eq:def_F_U} 
via $M = QUQ$ by \eqref{eq:balanced_polar_decomposition}. In particular, 
\begin{equation} \label{eq:Im_M_in_terms_of_Q_U} 
 \hspace*{-0.2cm} \Im M = Q (\Im U) Q = \rho Q F_U Q, \qquad - \Im M^{-1} = Q^{-1} (\Im U) Q^{-1} = \rho Q^{-1} F_U Q^{-1}, \qquad \Re M = Q (\Re U) Q. 
\end{equation}
Note that $U$, $U^*$ and $F_U$ commute. Moreover, since $U$ is unitary (cf.\ Lemma~\ref{lem:scaling_relation_v_u} (i)), the estimate \eqref{eq:scaling_U_Q_M} implies 
\begin{equation} \label{eq:Re_U_squared_equal_1} 
 (\Re U)^2 = 1- (\Im U)^2 = 1 + \ord(\rho^2). 
\end{equation}
We recall that $\psi$ defined in Proposition~\ref{pro:stability_operator} satisfies $\psi = \avg{F_U^4}$ (cf.\ the proof of Proposition~\ref{pro:stability_operator}). 
In the following, we will frequently use that $\avg{R} = 0$ if $R \in \oM$.

We now compute the coefficients from \eqref{eq:coefficients_abstract_cubic}. Indeed, we now show that 
\begin{subequations}  \label{eq:expansion_mu} 
 \begin{align} 
\mu_3 & = \psi + \ord(\rho + \eta/\rho) , \label{eq:mu_3} \\ 
\mu_2 & = 3 \ii \rho \psi + \ord(\rho^2 + \eta/\rho).\label{eq:mu_2}  
\end{align} 
\end{subequations} 
As a preparation of the proof of \eqref{eq:expansion_mu}, we expand $\cA[B,B]$. Proposition~\ref{pro:stability_operator}, \eqref{eq:Im_M_in_terms_of_Q_U} and the definition $\cF=\cC_Q \cS \cC_Q$ from \eqref{eq:def_cF} yield 
\begin{equation} \label{eq:cA_B_B} 
\begin{aligned} 
\cA[B,B] = M\cS[B]B & = Q U  \cF [ F_U + 2 \ii \rho F_U^2 \Re U] (F_U + 2 \ii \rho F_U^2 \Re U) Q  + \ord(\rho^2 + \eta/\rho)\\ 
& = Q U F_U ( F_U + 2 \ii \rho F_U^2 \Re U) Q + \ord(\rho^2 + \eta/\rho) \\ 
& = Q ( F_U^2 \Re U + 3 \ii \rho F_U^3) Q + \ord(\rho^2 + \eta/\rho). 
\end{aligned} 
\end{equation}
Here, we used that $\cF$ vanishes on $\oM$ and $F_U^2 \Re U \in \oM$ as well as \eqref{eq:F_U_eigenvector_cF} in the second step and $U = \Re U + \ii \Im U = \Re U + \ii \rho F_U = \Re U + \ord(\rho)$ by \eqref{eq:scaling_U_Q_M} in the last step. 

We recall the definitions $\cL=1-\cC_U \cF$ and $\cK=1-\cC_P \cF$ from \eqref{eq:def_cK_cL}.
Since $\cB^{-1} \cQ = \cC_Q \cL^{-1} \cQ_\cL \cC_Q^{-1} = \cC_Q \cK^{-1} \QK \cC_Q^{-1} + \ord(\rho)$ by Lemma~\ref{lem:small_eigenvalues} and \eqref{eq:norm_cK_minus_cL} we deduce from \eqref{eq:cA_B_B} 
that 
\begin{equation} \label{eq:cB_inverse_Q_cA_B_B} 
 \cB^{-1} \cQ \cA[B, B] = \cC_Q \cK^{-1} \QK[F_U^2\Re U] + \ord(\rho + \eta/\rho) = \cC_Q [F_U^2\Re U] + \ord(\rho + \eta/\rho). 
\end{equation}
The last step follows since $\cK^{-1} \QK$ acts as the identity map on $\oM$ and $F_U^2 \Re U \in \oM$. 

\begin{proof}[Proof of \eqref{eq:mu_3}]
For the first term in the definition of $\mu_3$ of \eqref{eq:coefficients_abstract_cubic}, we use Proposition~\ref{pro:stability_operator}, \eqref{eq:Im_M_in_terms_of_Q_U} and \eqref{eq:cB_inverse_Q_cA_B_B} to obtain 
\begin{equation} \label{eq:first_term_mu3} 
 2\scalar{\wh{B}}{\cA[B,\cB^{-1} \cQ \cA[B,B]} = \avg{F_U U\cF[F_U] F_U^2 (\Re U)} + \ord(\rho + \eta/\rho) = \psi + \ord(\rho + \eta/\rho).  
\end{equation}
In the first step, we also employed that $\cS$ vanishes on $\oM$ and  $\cC_Q[F_U^2 \Re U] \in \oM$, which follows from $Q\in \dM$ and $F_U^2 \Re U \in \oM$.   
The second step is a consequence of \eqref{eq:F_U_eigenvector_cF}, \eqref{eq:Re_U_squared_equal_1} and $\psi = \avg{F_U^4}$.  

To estimate the second term in the definition of $\mu_3$, we now estimate $\scalar{\wh{B}}{\cA[B,B_*]}$. Proposition~\ref{pro:stability_operator} and \eqref{eq:Im_M_in_terms_of_Q_U} imply
\begin{equation} \label{eq:scalar_L_cA_B_B_star} 
\begin{aligned} 
 2 \scalar{\wh{B}}{\cA[B,B_*]} & = \avg{F_U^2 U \cF[F_U] E_-} + \avg{U \cF[E_-F_U] F_U^2 (1 + 2 \ii \rho F_U \Re U)} + \ord(\rho^2 + \eta/\rho) \\ 
 &  = \avg{F_U^3 \Re U E_-} + \ii \rho \avg{F_U^4 E_-} - \avg{E_- \Re U F_U^3} - 2 \ii \rho \avg{E_- F_U^4(\Re U)^2} +\ord(\rho^2 + \eta/\rho) \\ 
 & = \ord(\rho^2 + \eta/\rho). 
\end{aligned} 
\end{equation}
Here, we used that $\cF$ vanishes on $F_U^2 \Re U \in \oM$ in the first step. The second step follows from \eqref{eq:F_U_eigenvector_cF} and \eqref{eq:cF_Emin}. 
In the last step, after cancelling the first and third terms, we employed \eqref{eq:Re_U_squared_equal_1} and $\avg{F_U^4 E_-} = 0$ by \eqref{eq:def_F_U}. 

The expansion in \eqref{eq:cB_inverse_Q_cA_B_B} and $E_- Q (\Re U) F_U^2 Q \in \oM$ imply
\begin{equation} \label{eq:scalar_Emin_cB_inverse_Q_cA_B_B} 
\scalar{E_-}{\cB^{-1} \cQ \cA[B,B]} = \avg{E_- Q (\Re U) F_U^2 Q} + \ord(\rho +\eta/\rho) = \ord(\rho + \eta/\rho). 
\end{equation}
From \eqref{eq:first_term_mu3}, \eqref{eq:scalar_L_cA_B_B_star}, $\abs{\scalar{E_-}{B_*}} \sim 1$ by \eqref{eq:expansion_B_star} and \eqref{eq:scalar_Emin_cB_inverse_Q_cA_B_B} 
we conclude that $\mu_3$ defined in \eqref{eq:coefficients_abstract_cubic} satisfies \eqref{eq:mu_3}. 
\end{proof} 

\begin{proof}[Proof of \eqref{eq:mu_2}]
We now turn to the expansion of $\mu_2$.  
From Proposition~\ref{pro:stability_operator}, \eqref{eq:Im_M_in_terms_of_Q_U} and \eqref{eq:cA_B_B}, we conclude 
\begin{equation} \label{eq:expansion_mu_2} 
\mu_2= \scalar{\wh{B}}{\cA[B,B]} 
  = \avg{F_U^3 \Re U + 3 \ii \rho F_U^4} + \ord(\rho^2 + \eta/\rho) 
  = 3 \ii \rho \psi + \ord(\rho^2 + \eta/\rho). 
\end{equation}
Here, we used $F_U^3 (\Re U) \in \oM$ and $\psi = \avg{F_U^4}$ in the last step. 
This completes the proof of \eqref{eq:mu_2}. 
\end{proof}

We now continue to estimate the right-hand side of \eqref{eq:proof_cubic_stability_aux1}.  
Young's inequality  implies that $\abs{\scalar{R_1}{D} \Theta} \leq n^{-1/4K}\abs{\Theta}^3 + n^{1/8K} \abs{\scalar{R_1}{D}}^{3/2}$. 
Then, we incorporate the error terms on the right-hand side of \eqref{eq:proof_cubic_stability_aux1} bounded by $n^{-1/4K} \abs{\Theta}^3$ and introduce $\wt{\mu}_3$ such that $\wt{\mu}_3 \Theta^3 = \mu_3 \Theta^3 + \ord(n^{-1/4K}\abs{\Theta}^3)$. 
Hence, $\abs{\wt{\mu}_3} \sim 1$ by \eqref{eq:mu_3} and $\psi \sim 1$ by Proposition~\ref{pro:stability_operator}. 
After this rearrangement, we divide \eqref{eq:proof_cubic_stability_aux1} by $\wt{\mu}_3$ and set 
\[\xi_2 \defeq \mu_2 / \wt{\mu}_3, \qquad
\xi_1 \defeq \big( - \beta\scalar{\wh{B}}{B} + \ord(n^{20/K} \abs{\Theta}( \abs{\scalar{E_-}{B}}^2 + \abs{\scalar{\wh{B}}{\cA[B,B_*]}}^2)) \big) /\wt{\mu}_3. \]  
Since $\abs{\wt{\mu}_3} \sim 1$, we conclude $\abs{\xi_2} \sim \abs{\mu_2} \sim \rho$ due to \eqref{eq:mu_2} and $\psi \sim 1$. 
For the scaling relation of $\xi_1$, we note that $\abs{\beta \scalar{\wh{B}}{B}} \sim \rho^2 + \eta/\rho$ by \eqref{eq:beta_beta_star_scaling} and $\abs{\scalar{\wh{B}}{B}} \sim 1$ from Proposition~\ref{pro:stability_operator}. 
Moreover, from \eqref{eq:scalar_L_cA_B_B_star} and \eqref{eq:bound_scalar_E_minus_B}, we obtain 
$n^{20/K} \abs{\Theta}( \abs{\scalar{E_-}{B}}^2 + \abs{\scalar{\wh{B}}{\cA[B,B_*]}}^2)) \lesssim n^{-10/K} ( \rho^4 + \eta^2/\rho^2)$. 
Hence, $\abs{\xi_1} \sim \rho^2 + \eta/\rho$. 
This completes the proof of \eqref{eq:cubic_equation}, the scaling relations \eqref{eq:scaling_xi_2_xi_1} and the bound on $\eps_*$ in \eqref{eq:bound_eps_star}. 

Finally, the expansion of $G-M$ in \eqref{eq:G_minus_M_expansion_via_Theta_and_D} and the error estimate in \eqref{eq:normstar_E} follow 
from \eqref{eq:expansion_Y_abstract_cubic} in Lemma~\ref{lem:cubic_equation_abstract} together with \eqref{eq:scalar_Emin_cB_inverse_Q_cA_B_B}, $\normstar{MD} \lesssim n^{1/2K} \normstar{D}$, \eqref{eq:bound_scalar_E_minus_B} and $R_1=M^* (\cB^{-1} \cQ)^*[E_-]$. 
This completes the proof of Proposition~\ref{pro:cubic_equation}. 
\end{proof}

\begin{proof}[Proof of \eqref{eq:bounds_star_norm_scalar_products}] 
This proof is motivated by the proof of the bound on $\normstar{\cP[R]}$ in \cite[Lemma~3.4]{AltEdge}. 
We compute 
\begin{equation} \label{eq:scalar_wh_B_R} 
\scalar{\wh{B}}{R} = \frac{1}{2n} \Tr (\wh{B}^* R ) = \frac{1}{2n} \sum_{a=1}^{2n} \scalar{\wh{B} \mathbf e_a}{R \mathbf e_a} = \frac{1}{2n} \sum_{a=1}^{2n} R_{\wh{B}_{a \cdot}^*a}. 
\end{equation}
Since $\norm{\wh{B}} \lesssim 1$ by Proposition~\ref{pro:stability_operator}, we have $\norm{\wh{B}_{a\cdot}^*} \lesssim 1$. 
Thus, we obtain $\abs{R_{\wh{B}_{a\cdot}^*a}} \leq \norm{R}_{I_0} \norm{\wh{B}_{a\cdot}^*} \lesssim \normstar{R}$. Using this in \eqref{eq:scalar_wh_B_R} completes the proof of the bound on $\scalar{\wh{B}}{R}$ in \eqref{eq:bounds_star_norm_scalar_products}. 

The bound for $\scalar{\wh{B}_*}{R}$ is proven in the same way. The proof of the estimate on $\scalar{E_-}{R}$ is simpler.
\end{proof}

\section{Local law for $H_z$}  \label{sec:local_law_H}

The main result of this section, Theorem~\ref{thm:local_law_H}, is a precise expansion of the resolvent of $H_z$ at $\ii \eta$ when $\eta>0$ is sufficiently small and the modulus of $z\in \C$ is close to 1. 
We recall that we assume $\varrho(\smallS) = 1$ (cf.\ \eqref{eq:normalization_spectral_radius} and the associated explanations). 
For the formulation of Theorem~\ref{thm:local_law_H} as well as the subsequent statements and arguments, we use the following notion for high probability estimates. 

\begin{definition}[Stochastic domination] \label{def:stochastic_domination} 
Let $\Phi= (\Phi^{(n)})_n$ and $\Psi =(\Psi^{(n)})_n$ be two sequences of nonnegative random variables. 
We say $\Phi$ is \emph{stochastically dominated} by $\Psi$ and write $\Phi \prec \Psi$ if, for all (small) $\eps>0$ and (large) $D>0$, there is a constant $C_{\eps,D}>0$ such that 
\begin{equation}\label{eq:def_stochastic_domination} 
 \P \Big( \Phi^{(n)} > n^\eps \Psi^{(n)} \Big) \leq \frac{ C_{\eps,D}}{n^D} 
\end{equation}
for all $n \in \N$. If $\Phi^{(n)}$ and $\Psi^{(n)}$ depend on some parameter family $U^{(n)}$ and \eqref{eq:def_stochastic_domination} holds for all $u \in U^{(n)}$ then we say that $\Phi \prec \Psi$ uniformly for all $u \in U^{(n)}$. 
\end{definition} 

In the following, let $\rho = \avg{\Im M}/\pi$ (cf.~\eqref{eq:def_rho}), $H_z$ be defined as in \eqref{eq:def_H_z} and $G \defeq (H_z - \ii \eta)^{-1}$. 
Moreover, $M$ denotes the solution of the MDE, \eqref{eq:mde}. 
For each $z\in \C$, we now introduce the \emph{fluctuation scale} $\etaf=\etaf(z)$ of eigenvalues of $H_z$ around zero: 
We set 
\begin{equation} \label{eq:def_etaf} 
 \etaf(z) \defeq 
\begin{cases} (1-\abs{z}^2)^{-1/2} n^{-1}, & \text{ if } \abs{z}^2 \leq 1 - n^{-1/2}, \\ 
n^{-3/4}, & \text{ if } 1 - n^{-1/2} < \abs{z}^2 \leq 1 + n^{-1/2}, \\ 
(\abs{z}^2 - 1)^{1/6} n^{-2/3}, & \text{ if } 1 + n^{-1/2} < \abs{z}^2 \leq 2,\\ 
n^{-2/3}, & \text{ if } \abs{z}^2 > 2. 
\end{cases} 
\end{equation}
The fluctuation scale describes the typical eigenvalue spacing of $H_z$ at zero (first two cases) and at the spectral 
edges of the eigenvalue density of $H_z$ close to zero (last two cases). 
The definition of $\etaf$ in \eqref{eq:def_etaf} is motivated by the definition of the fluctuation scale in \cite{Cusp1} and the scaling relations of $\rho$ from \eqref{eq:scaling_rho}. 
For $\abs{z} > 1$, the eigenvalue density of $H_z$ has a gap of size $\Delta \sim (\abs{z}^2-1)^{3/2}$ around zero, hence, \eqref{eq:def_etaf} is analogous to \cite[Eq.~(2.7)]{Cusp1}. 
If $\abs{z} \leq 1$ then the eigenvalue density of $H_z$ has a small local minimum of height $\rho_0 \sim (1-\abs{z}^2)^{1/2}$ 
at zero. So \eqref{eq:def_etaf} should be compared to \cite[Eq.~(A.8a)]{Cusp1}.

\begin{theorem}[Local law for $H_z$] \label{thm:local_law_H} 
Let $X$ satisfy \ref{assum:flatness} and \ref{assum:bounded_moments}. Then there is $\tau_* \sim 1$ such that, for each $\zeta>0$, the estimates 
\begin{equation} \label{eq:G_minus_M_standard}  
 \abs{\scalar{\bx}{(G- M)\by}} \prec \norm{\bx}\norm{\by}\bigg(\sqrt{\frac{\rho}{n\eta}} + \frac{1}{n\eta} \bigg), \qquad \qquad 
 \abs{\avg{R(G-M)}}  \prec \frac{\norm{R}}{n \eta}  
\end{equation} 
hold uniformly for all $z \in \C$ satisfying $\abs{\abs{z} - 1} \leq \tau_*$, for all $\eta \in [n^{\zeta}\etaf(z),n^{100}]$,  
for any deterministic vectors $\bx, \by \in \C^{2n}$ and deterministic matrix $R \in \C^{2n\times 2n}$. 

Moreover, outside the spectrum, for each $\zeta>0$ and $\gamma>0$, we have the improved bound 
\begin{equation} \label{eq:G_minus_M_average_improved} 
 \abs{\avg{R (G-M)}} \prec \norm{R} \frac{n^{-\gamma/3}}{n\eta} 
\end{equation}
uniformly for all $z \in \C$ and $\eta \in \R$ satisfying $\abs{z}^2 \geq 1 + (n^\gamma \eta)^{2/3}$, $ \abs{z} \leq 1 + \tau_*$, $n^{\zeta} \etaf(z) \leq \eta \leq \tau_*$ and $R \in \C^{2n\times 2n}$. 
\end{theorem} 

We stress that the spectral parameter of the resolvent $G$ in the previous theorem and throughout the entire paper lies on the 
imaginary axis and is given by $\ii \eta$. 
With additional efforts our method can be extended to spectral parameters near the imaginary axis, but the Hermitization formula \eqref{eq:girko} requires to understand the resolvent of $H_z$ only on the imaginary axis, so we restrict ourselves to this case. 
After the proof of Theorem~\ref{thm:local_law_H}, we will establish the following corollary that will directly imply Corollary~\ref{thm:delocalization}. 

\begin{corollary}[Isotropic eigenvector delocalization]\label{cor:delocalization} 
Let $\tau_* \sim 1$ be chosen as in Theorem~\ref{thm:local_law_H}.  Let $\bx\in \C^{2n}$ be a fixed deterministic vector. 
If $\bu \in \C^{2n}$ is contained in the kernel of $H_z$ for some $z \in \C$ satisfying $\abs{\abs{z} - 1} \leq \tau_*$, 
i.e. $H_z \bu = 0$ then 
\[ \abs{\scalar{\bx}{\bu}}  \prec n^{-1/2} \norm{\bx} \norm{\bu}. \] 
\end{corollary}

We remark that the conclusion of Corollary~\ref{cor:delocalization} is also true if $\abs{\abs{z} - 1} >\tau_*$. This can easily be shown following the proof of \cite[Theorem~5.2]{Altcirc}, where certain steps of the proof of 
the local law from \cite{AjankiCorrelated} have been used except that now analogous inputs from \cite{Erdos2017Correlated} are needed instead of \cite{AjankiCorrelated}.

\begin{proof}[Proof of Corollary~\ref{thm:delocalization}]
Let $u\in \C^n$ be an eigenvector of $X$, i.e. $X u = \zeta u$ for some $\zeta \in \C$. If $\abs{\zeta} \leq 1 - \tau_*$ then the claim follows from \cite[Corollary~2.6]{Altcirc}. 
Otherwise, we can assume that $\abs{\abs{\zeta} - 1} \leq \tau_*$ by \cite[Theorem~2.5 (ii)]{Altcirc}. 
We set $\bu \defeq (0,u)^t$ and obtain $H_\zeta \bu = 0$. 
Hence, we choose $\bx = \mathbf{e}_i$ in Corollary~\ref{cor:delocalization} and obtain a bound on $\abs{u_i}$. 
Finally, taking the maximum over $i \in [n]$ and a simple union bound complete the proof of Corollary~\ref{thm:delocalization}.  
\end{proof} 

\begin{remark}[Choice of $\cS$ in MDE] \label{rem:convention_for_S} 
We warn the reader that in this paper the definition of the self-energy operator $\cS$ given in \eqref{eq:def_cS} coincides with the definition in \cite[Eq.~(3.3)]{Altcirc}. 
However, this convention differs from the more canonical choice, $R \mapsto \E[(H- \E H)R(H- \E H)]$, typically used for a Hermitian random matrix $H$ in several other works (e.g.~\cite{AjankiCorrelated,Erdos2017Correlated,AltEdge}). 
The present choice of $\cS$ substantially simplifies the analysis of the associated MDE. 
As a price for this, we will need a simple adjustment when estimating the error term $D$ 
in the perturbed Dyson equation, \eqref{eq:perturbed_mde}, since the convenient estimate on $D$ builds upon the canonical choice of $\cS$. As Proposition~\ref{pro:D_bounds} below shows, nevertheless the very same estimates on $D$ as for the canonical choice \cite{Cusp1} 
can be obtained for the current choice. 
\end{remark} 

We will establish Theorem~\ref{thm:local_law_H} in Subsection~\ref{subsec:proof_local_law_H} below. The proof will consist of a 
bootstrapping argument using the stability properties of the MDE in the previous sections, in particular, Proposition~\ref{pro:cubic_equation}, 
and the following bounds on the error term $D$ in the perturbed MDE for $G$, \eqref{eq:perturbed_mde}. 
To formulate these bounds, we now introduce some norms for random matrices and a spectral domain. 
For $p \geq 1$, a scalar-valued random variable $Z$ and a random matrices $Y\in \C^{2n\times 2n}$, we define the $p$th-moment norms 
\[ \norm{Z}_p \defeq \big( \E \abs{Z}^p \big)^{1/p}, \qquad \norm{Y}_p \defeq \sup_{\bx,\by} \frac{\norm{\scalar{\bx}{Y\by}}_p}{\norm{\bx}\norm{\by}}. \] 
For $\zeta>0$, we introduce the spectral domain 
\[ \DD_\zeta \defeq \{ (z,\eta) \in \C \times \R \colon n^{-1+ \zeta}\leq \eta \leq \tau_*, \, \abs{\abs{z} - 1} \leq \tau_* \}, \] 
where $\tau_*\sim 1$ is chosen such that \eqref{eq:scaling_rho} implies $\rho + \eta/\rho\leq \rho_*$ for all $(z,\eta) \in \DD_\zeta$ 
with $\rho_*$ from Proposition~\ref{pro:cubic_equation}. 

In the following, we will work with several quantities that will depend on $\eps$ and $p$. In order to simplify the notation, we will use the following notatoin. Let $f(\eps,p,a)$ and $g(\eps,p,a)$ are two quantities that depend on $\eps$ and $p$ as well as other parameters $a$. 
We write $f \leq_{\eps,p} g$ if there is a constant $C(\eps,p)$ such that $f(\eps,p,a) \leq C(\eps,p) g(\eps,p,a)$ for all $\eps$, $p$ and $a$. 

\begin{proposition} \label{pro:D_bounds} 
Let $D$ be the error matrix from \eqref{eq:perturbed_mde}. 
Under the assumptions of Theorem~\ref{thm:local_law_H}, there is a constant $C>0$ such that for any $p \geq 1$, $\eps >0$, $(z,\eta) \in \DD_0$,
and any deterministic  $\bx, \by \in \C^{2n}$ 
and $R \in \C^{2n \times 2n}$, we have the moment bounds
\begin{subequations} 
\begin{align} 
\norm{\scalar{\bx}{D\by}}_p & \leq_{\eps,p} \norm{\bx} \norm{\by} n^\eps  \psi_q' \Big( 1 + \norm{G}_q \Big)^C \bigg( 1 + \frac{\norm{G}_q}{\sqrt{ n}} \bigg)^{Cp}, \label{eq:D_bound_isotropic} \\ 
\norm{\avg{RD}}_p & \leq_{\eps,p} \norm{R}  n^\eps    \Big[\psi_q'\Big]^2 \Big( 1 + \norm{G}_q \Big)^C \bigg( 1 + \frac{\norm{G}_q}{\sqrt{ n}} \bigg)^{Cp}. \label{eq:D_bound_average} 
\end{align} 
Moreover, if $R \in \oM$ then we have the improved estimate 
\begin{equation} \label{eq:D_bound_cusp_FA} 
\norm{\avg{R D}}_p \leq_{\eps,p} \norm{R}  n^\eps  \sigma_q \Big[\psi + \psi_q'\Big]^2 \Big( 1 + \norm{G}_q \Big)^C \bigg( 1 + \frac{\norm{G}_q}{\sqrt{n}} \bigg)^{Cp}. 
\end{equation}
\end{subequations} 
Here, we used the $z$-dependent control parameters
\[ \psi \defeq \sqrt{\frac{\rho}{n\eta}}, \quad \psi_q' \defeq \sqrt{\frac{\norm{\Im G}_q}{n\eta}}, \quad \psi_q'' \defeq \norm{G-M}_q, \quad \sigma_q \defeq \rho + \psi + \sqrt{\frac{\eta}{\rho}} + \psi_q' + \psi_q'' \] 
with $q \defeq Cp^3/\eps$. 
\end{proposition}

\begin{remark}\label{rem:sigma}
This proposition is the exact counterpart of the cusp fluctuation averaging in \cite[Proposition 4.12]{Cusp1}
with  $\sigma=0$, hence the definition of $\sigma_q$ does not contain $\sigma$. 
Notice that  $\sigma=0$ in our case following from the fact that the spectral parameter $\ii\eta$ lies on the imaginary axis to which
the spectrum is symmetric.
\end{remark}

We remark that $\psi$ in Proposition~\ref{pro:D_bounds} is different from the $\psi$ defined in Proposition~\ref{pro:stability_operator}. 
This should not lead to any confusion since the latter notation is used only in Section~\ref{sec:stability} and \ref{sec:cubic_equation} while the former is used exclusively in Proposition~\ref{pro:D_bounds} 
and Section~\ref{subsec:proof_D_bounds}. 
We prefer to stick to these notations for compatibility with the publications \cite{Ajankirandommatrix,Altshape,Cusp1}. 
We will show Proposition~\ref{pro:D_bounds} in Subsection~\ref{subsec:proof_D_bounds} below.

\subsection{Proof of Theorem~\ref{thm:local_law_H}} \label{subsec:proof_local_law_H} 

This subsection is devoted to the proof of Theorem~\ref{thm:local_law_H}. To that end, we follow the arguments from \cite[Sections~3.2, 3.3]{Cusp1}, where the local law for a general Hermitian matrix close to a cusp regime was 
deduced from estimates on $D$ as provided in Proposition~\ref{pro:D_bounds}. We will present the main steps of the proof, focusing on the differences, but for arguments that require only simple (mostly notational) adjustments we will refer the reader to \cite{Cusp1}. 
When comparing to \cite{Cusp1}, the reader should think of the following cases described in the notation of \cite[Eq.~(3.7b)]{Cusp1}. 
For $\abs{z} \leq 1$, the eigenvalue density of $H_z$ has a local minimum of size $\rho(\tau_0) \sim (1-\abs{z}^2)^{1/2}$ at $\tau_0 = 0$ and $\omega = 0$. For $\abs{z} > 1$, the spectrum of $H_z$ has a symmetric gap of size $\Delta 
\sim (\abs{z}^2-1)^{3/2}$ around zero and we study the resolvent of $H_z$ at the middle of this gap, $\abs{\omega} = \Delta/2$. 
For a random matrix $Y \in \C^{2n\times 2n}$ and a deterministic control parameter $\Lambda = \Lambda(z)$, we define 
the notations $\abs{Y} \prec \Lambda$ and $\abs{Y}_\mathrm{av} \prec \Lambda$ as follows 
\begin{alignat*}{4} 
\abs{Y} & \prec \Lambda \qquad &\Longleftrightarrow&& \qquad  \abs{Y_{\bx\by}} &\prec \Lambda \norm{\bx}\norm{\by} \quad &&\text{ uniformly for all } \bx, \by \in \C^{2n}, \\ 
\abs{Y}_\mathrm{av} & \prec \Lambda \qquad & \Longleftrightarrow &&\qquad  \abs{\avg{RY}} &\prec \Lambda \norm{R} \quad &&\text{ uniformly for all } R \in \C^{2n\times 2n}. 
\end{alignat*}  
We recall that by definition $Y_{\bx\by}=\scalar{\bx}{Y \by}$ for $\bx, \by \in \C^{2n}$. 
The following lemma relates this notion of high probability bounds to the high moments estimates introduced above. We leave the simple adjustments of the proof of \cite[Lemma~3.7]{AltEdge} to the reader.

\begin{lemma} \label{lem:high_moment_bound_to_high_prob_bound} 
Let $Y$ be a random matrix in $\C^{2n\times 2n}$, $\Phi$ a deterministic control parameter such that $\Phi\geq n^{-C}$ and $\norm{Y} \leq n^C$ for some $C >0$. Let $K \in \N$ be fixed. Then we have 
\[ \normstar{Y}^{K,\bx,\by} \prec \Phi \text{ uniformly for } \bx, \by \in \C^{2n} \quad \Longleftrightarrow \quad \abs{Y} \prec \Phi \quad \Longleftrightarrow \quad \norm{Y}_p \leq_{\eps,p} n^\eps \Phi \text{ for all } \eps >0, p \geq 1. \] 
\end{lemma}

The next lemma adapts Proposition~\ref{pro:cubic_equation} to the random matrix setup with the help of Proposition~\ref{pro:D_bounds}. 
The lemma is the analog of \cite[Lemma~3.8]{Cusp1} in our setup.  

\begin{lemma} \label{lem:cubic_inequality_refined} 
Let $\zeta, c>0$ be fixed and sufficiently small. We assume that $\abs{G-M} \prec \Lambda$, $\abs{\Im(G-M)}\prec \Xi$ and $\abs{\Theta} \prec \theta$ at some fixed $(z,\eta) \in\DD_{\zeta}$ for some 
deterministic control parameters $\Lambda$, $\Xi$ and $\theta$ such that $\Lambda+ \Xi + \theta \lesssim n^{-c}$. Then, for any sufficiently small $\delta >0$, the estimates 
\begin{equation} \label{eq:cubic_equation_refined} 
\abs{ \Theta^3 + \xi_2 \Theta^2 + \xi_1 \Theta} \prec n^{2 \delta} \bigg( \rho + \frac{\eta^{1/2}}{\rho^{1/2}} + \bigg( \frac{\rho + \Xi}{n\eta} \bigg)^{1/2} \bigg) \frac{\rho + \Xi}{n\eta} + n^{-\delta} \theta^3
\end{equation}
and 
\begin{equation} \label{eq:G-M_theta}  
 \abs{G-M} \prec \theta + \sqrt{\frac{\rho + \Xi}{n \eta}}, \qquad \qquad \abs{G-M}_\mathrm{av} \prec \theta + \frac{\rho + \Xi}{n\eta} 
\end{equation}
hold, where $\xi_2$ and $\xi_1$ are chosen as in Proposition~\ref{pro:cubic_equation} and $\Theta = \scalar{\wh{B}}{G-M}/\scalar{\wh{B}}{B}$.  

Moreover, for fixed $z$, $\Theta$ is a continuous function of $\eta$ as long as $(z,\eta) \in \DD_\zeta$. 
\end{lemma} 

\begin{proof} 
Owing to Lemma~\ref{lem:high_moment_bound_to_high_prob_bound} and $\abs{G} \prec \norm{M} + \Lambda \lesssim 1$, the high-moment bounds in \eqref{eq:D_bound_isotropic} and \eqref{eq:D_bound_average} imply
\begin{equation} \label{eq:aux_bounds_D} 
 \abs{D} \prec \sqrt{\frac{\rho + \Xi}{n \eta}}, \qquad \qquad \abs{D}_\mathrm{av} \prec \frac{\rho + \Xi}{n\eta}.  
\end{equation}
We conclude that the assumption on $\normstar{G-M} + \normstar{D}$ in Proposition~\ref{pro:cubic_equation} is satisfied for sufficiently large $K$ depending on $c$ and $\eps$ in the definition of $\prec$ in Definition~\ref{def:stochastic_domination}. 
What remains to ensure the applicability of Proposition~\ref{pro:cubic_equation} is checking $\scalar{E_-}{G} = 0$. 
In fact, we now prove that, for each $z \in \C$ and $\eta >0$, the resolvent $G = (H_z - \ii \eta)^{-1}$ satisfies 
\begin{equation} \label{eq:scalar_E_minus_G_equals_0} 
\scalar{E_-}{G} = 0. 
\end{equation} 
For the proof of \eqref{eq:scalar_E_minus_G_equals_0}, we denote by $G_{11}, G_{22} \in \C^{n\times n}$ the upper-left and lower-right $n\times n$-minor of the resolvent $G = (H_z- \ii \eta)^{-1} \in \C^{2n\times 2n}$. 
Then the block structure of $H_z$ from \eqref{eq:def_H_z} yields 
\[ G_{11} = \frac{\ii \eta }{(X-z)(X-z)^* + \eta^2}, \qquad G_{22} = \frac{\ii \eta}{(X-z)^* (X-z) + \eta^2} .\] 
Since $(X-z)(X-z)^*$ and $(X-z)^*(X-z)$ have the same eigenvalues we obtain $(2n) \scalar{E_-}{G} = \Tr G_{11} - \Tr G_{22} = 0$. This shows \eqref{eq:scalar_E_minus_G_equals_0} and, thus, ensures the applicability of
Proposition~\ref{pro:cubic_equation}.

The first bound in \eqref{eq:aux_bounds_D}, the bounds on $\cB^{-1} \cQ$ and $MR$ in \eqref{eq:aux_bounds_star_norm} and Lemma~\ref{lem:high_moment_bound_to_high_prob_bound} 
yield 
\begin{equation} \label{eq:proof_cubic_refined_aux4} 
\abs{\cB^{-1} \cQ[MD]}\prec \sqrt{\frac{\rho + \Xi}{n \eta}} 
\end{equation} by choosing $K$ sufficiently large to absorb various $n^{1/K}$-factors into $\prec$.  
Similarly, we use \eqref{eq:aux_bounds_star_norm}, \eqref{eq:aux_bounds_D}, the assumption $\abs{\Theta} \prec \theta$ and Lemma~\ref{lem:high_moment_bound_to_high_prob_bound} to estimate the other terms in  
\eqref{eq:G_minus_M_expansion_via_Theta_and_D} and \eqref{eq:normstar_E} and deduce \eqref{eq:G-M_theta}. 

What remains is estimating the right-hand side of \eqref{eq:bound_eps_star} to obtain \eqref{eq:cubic_equation_refined}. 
Incorporating the $n^{1/K}$ factors into $\prec$, we see that $\norm{D}_*^3$, $\abs{\scalar{R_1}{D}}^{3/2}$ and $\abs{\scalar{R_2}{D}}^{3/2}$ are dominated by the right-hand side of \eqref{eq:cubic_equation_refined} 
due to \eqref{eq:D_bound_isotropic} and Lemma~\ref{lem:high_moment_bound_to_high_prob_bound}. 
Recall $M = \Re M + \ord(\rho)$ with $\Re M \in \oM$ and $\wh{B}^* = -\rho^{-1} \Im M^{-1} + \ord(\rho + \eta/\rho)$ (cf.~\eqref{eq:expansion_L}) 
with $-\rho^{-1} \Im M^{-1} \sim 1$ and $-\rho^{-1} \Im M^{-1} \in \dM$. Therefore, $M$ is almost off-diagonal while $\wh{B}^*$ is almost diagonal, so we find 
$B_1 \in \oM$, $B_2 \in \C^{2n\times 2n}$ such that $\wh{B}^*M = B_1 + B_2$ and $\norm{B_1} \lesssim 1$, $\norm{B_2} \lesssim \rho + \eta/\rho$. 
Hence, \eqref{eq:D_bound_average} and \eqref{eq:D_bound_cusp_FA} imply 
\[ \abs{\scalar{\wh{B}}{MD}} \prec \bigg( \rho + \frac{\eta^{1/2}}{\rho^{1/2}} + \theta + \bigg( \frac{\rho + \Xi}{n \eta} \bigg)^{1/2} \bigg) \frac{ \rho + \Xi}{n \eta} \prec n^{\delta} \bigg( \rho + \frac{\eta^{1/2}}{\rho^{1/2}} 
+ \bigg( \frac{\rho + \Xi}{n \eta} \bigg)^{1/2} \bigg) \frac{ \rho + \Xi}{n \eta} + n^{-{\delta}} \theta^3, 
\] 
where we used the bound on $\abs{G-M}$ from \eqref{eq:G-M_theta} in the second step and Young's inequality in the last step. 

We now conclude the proof of \eqref{eq:cubic_equation_refined} by showing that 
\begin{equation} \label{eq:proof_cubic_refined_aux2} 
 \abs{\scalar{\wh{B}}{M (\cS \cB^{-1} \cQ[MD])\cB^{-1} \cQ[MD]}} \prec \bigg( \frac{\rho + \Xi}{n\eta} \bigg)^{3/2}. 
\end{equation} 
Defining $A \defeq \cB^{-1}\cQ[MD]$, the inclusion $\ran \cS\subset \dM$ implies 
\begin{equation} \label{eq:proof_cubic_refined_aux3} 
\abs{\scalar{\wh{B}}{M (\cS \cB^{-1} \cQ[MD])\cB^{-1} \cQ[MD]}} = \frac{1}{2n} \absbb{\sum_{a \in [2n]} (A\wh{B}^* M)_{aa} \cS[A]_{aa} } \lesssim \max_{a \in [2n]} \abs{\cS[A]_{aa}} \max_{a \in [2n]} \abs{(A \wh{B}^* M)_{aa}}. 
\end{equation}
Owing to the second bound in \eqref{eq:aux_bounds_D}, the definition of $\abs{\genarg}_\mathrm{av}$ and $\norm{M^* (\cB^*)^{-1} \cQ^* \cS[{\be_a}\be_a^*]} \lesssim 1$, we obtain 
\begin{equation} \label{eq:proof_cubic_refined_aux1} 
\abs{\cS[A]_{aa}} = \abs{\scalar{\be_a}{\cS[A] \be_a}} = \abs{\scalar{\cS[\be_a\be_a^*]}{\cB^{-1} \cQ[MD]}} = \abs{\scalar{M^* (\cB^*)^{-1} \cQ^*\cS[{\be_a}\be_a^*]}{D}} \prec \frac{\rho + \Xi}{n\eta}. 
\end{equation}
To conclude \eqref{eq:proof_cubic_refined_aux2}, we use \eqref{eq:proof_cubic_refined_aux1} in \eqref{eq:proof_cubic_refined_aux3} 
and, to bound the second factor in \eqref{eq:proof_cubic_refined_aux3}, we deduce from \eqref{eq:proof_cubic_refined_aux4} 
and $\norm{\wh{B}^* M} \lesssim 1$   
that $\max_{a \in [2n]} \abs{(A\wh{B}^* M)_{aa}} = \max_{a \in [2n]} \abs{\scalar{\be_a}{A \wh{B}^*M \be_a}} = \max_{a \in [2n]} \abs{A_{a, \wh{B}^*M\be_a}}$ is stochastically dominated by the right-hand side of \eqref{eq:proof_cubic_refined_aux4}. 
This completes the proof of \eqref{eq:cubic_equation_refined}. 

Finally, we note that $\Theta$ is a continuous function of $\eta$ as $\wh{B}$, $B$, $G$ and $M$ are continuous with respect to $\eta$. This completes the proof of Lemma~\ref{lem:cubic_inequality_refined}. 
\end{proof}

We now introduce $\wt{\xi}_2$ and $\wt{\xi}_1$ which will turn out to be comparable versions of the coefficients $\xi_2$ and $\xi_1$, respectively, (see \eqref{eq:scaling_xi_2_xi_1} above and Lemma~\ref{lem:inequality_coefficients} \ref{item:xi_and_tilde_xi} below). 
Moreover, they depend explicitly and monotonically on $\eta$ which will be important for our arguments. 
We define  
\begin{equation} \label{eq:def_wt_xi_2_1} 
 \wt{\xi}_2 \defeq \absb{1-\abs{z}^2}^{1/2} + \eta^{1/3}, \qquad \wt{\xi}_1 \defeq \big(\wt{\xi}_2\big)^2. 
\end{equation}
These definitions are chosen in analogy to \cite[Eq.~(3.7e)]{Cusp1}, where in the first case we chose $\abs{\omega} \sim \Delta \sim (\abs{z}^2 - 1)^{3/2}$ and, in the second case, $\rho(\tau_0) \sim (1- \abs{z}^2)^{1/2}$ 
and $\omega = 0$.

\begin{lemma}[Properties of $\wt{\xi}_2$ and $\wt{\xi}_1$] \phantomsection \label{lem:inequality_coefficients} 
\begin{enumerate}[label=(\roman*)]
\item \label{item:xi_and_tilde_xi} 
For all $z \in D_\tau(0)$ and $\eta \in (0,1]$, we have $\rho^2 + \eta/\rho \sim \wt{\xi}_1$. 
For any $\eta \in (0,1]$, we have $\wt{\xi}_2 \sim \rho$ if $z \in \C$ satisfies $\abs{z} \leq 1$ and $\wt{\xi}_2 \gtrsim \rho$ if $z \in D_\tau(0)\setminus \overline{D_1(0)}$.  
\item \label{item:comparsion_tilde_xi_rho} 
Uniformly for all $z \in D_\tau(0)$ and $\eta \geq \etaf$, we have 
\[ \wt{\xi}_2 \gtrsim \frac{1}{n\eta} + \bigg( \frac{\rho}{n \eta} \bigg)^{1/2}, \qquad \wt{\xi}_1 \gtrsim \wt{\xi}_2 \bigg( \rho + \frac{1}{n \eta} \bigg). \] 
\end{enumerate}
\end{lemma} 

\begin{proof} 
The scaling relations in \ref{item:xi_and_tilde_xi} follow easily from the scaling relations for $\rho$ in \eqref{eq:scaling_rho} by distinguishing the regimes $\abs{z} \leq 1$ and $\abs{z} >1$. 

The first bound in \ref{item:comparsion_tilde_xi_rho} follows once $\wt{\xi}_2 \gtrsim 1/(n\eta)$ and $(\wt{\xi}_2)^2 \gtrsim \rho/(n\eta)$ are proven. 
For $\abs{z}^2 \leq 1 - n^{-1/2}$, we have $(1-\abs{z}^2)^{1/2} \gtrsim 1/(n\eta)$ if $\eta \geq \etaf(z)$. If $\abs{z}^2 > 1 - n^{-1/2}$ then $\etaf(z) \geq n^{-3/4}$ by definition and, hence, $\eta^{1/3} \gtrsim 1/(n\eta)$. 
This shows $\wt{\xi}_2 \gtrsim 1/(n\eta)$ in all regimes. If $\abs{z} \leq 1$ then $\wt{\xi}_2 \sim \rho$ by \ref{item:xi_and_tilde_xi} and $\rho^2 \gtrsim \rho/(n\eta)$ is easily verified due to \eqref{eq:scaling_rho}. 
For $\abs{z}>1$, $(\wt{\xi}_2)^2 \gtrsim \rho/(n\eta)$ is equivalent to $(\abs{z}^2- 1)^{1/2} + \eta^{1/3} \gtrsim n^{-1/4}$ which follows directly from $\etaf(z) \gtrsim n^{-3/4}$ in this regime. 
This shows the first bound in \ref{item:comparsion_tilde_xi_rho}. 

We note that, owing to $\wt{\xi}_1 = (\wt{\xi}_2)^2$, the second bound in \ref{item:comparsion_tilde_xi_rho} is equivalent to $\wt{\xi}_2 \gtrsim \rho + 1/(n \eta)$. But we know $\wt{\xi}_2 \gtrsim \rho$ from \ref{item:xi_and_tilde_xi}. This completes the proof of Lemma~\ref{lem:inequality_coefficients}. 
\end{proof} 

\begin{proof}[Proof of Theorem~\ref{thm:local_law_H}] 
We will only consider the bounds in \eqref{eq:G_minus_M_standard} for $\eta \leq \tau_*$ since the opposite regime is covered by \cite[Theorem~2.1]{Erdos2017Correlated} due to $\rho \sim \eta^{-1}$ for $\eta \geq \tau_*$ 
by \eqref{eq:scaling_rho} and \cite[Eq.~(3.9)]{Altcirc}. 
The bounds \eqref{eq:G_minus_M_standard} and \eqref{eq:G_minus_M_average_improved} are the analogs of (3.28) and (3.30) in \cite{Cusp1}, respectively. 
Given the preparations presented above, the proofs of \eqref{eq:G_minus_M_standard} and \eqref{eq:G_minus_M_average_improved}, thus, the one of Theorem~\ref{thm:local_law_H}, are identical to the proofs of \cite[Eq.'s~(3.28) and (3.30)]{Cusp1} in \cite[Section~3.3]{Cusp1}. 
Therefore, we only describe the main strategy here and explain the applicability of certain inputs. 

The proof of Theorem~\ref{thm:local_law_H} starts with the following (isotropic) rough bound on $G-M$. 
\begin{lemma}[Rough bound] \label{lem:rough_bound} 
For any $\zeta >0$, there exists a constant $c>0$ such that the rough bounds 
\[ \abs{G-M} \prec n^{-c} \] 
holds on the spectral domain $\DD_\zeta$. 
\end{lemma} 

\begin{proof} 
The proof of Lemma~\ref{lem:rough_bound} is identical to the one of \cite[Lemma~3.9]{Cusp1}. 
We explain the main idea. 
From \cite[Theorem~2.1]{Erdos2017Correlated}, an initial bound on $\abs{G-M}$ at $\eta = \tau_*$ is deduced. 
We remark that \cite[Theorem~2.1]{Erdos2017Correlated} is also applicable in our setup. 
Using the monotonicity of the map $\eta \mapsto \eta \norm{G(z,\eta)}_p$ (which is shown in the same way as in \cite[Eq.~(5.11)]{Erdos2017Correlated}), the bootstrapping result  
\cite[Lemma~3.10]{Cusp1} for cubic inequalities and Lemma~\ref{lem:cubic_inequality_refined}, this initial bound is strengthened and propagated down to all (small) values of $\eta$ in $\DD_\zeta$. 
Moreover, the assumptions in (ii) of \cite[Lemma~3.10]{Cusp1} are easily checked by using the definitions of $\wt{\xi}_2$ and $\wt{\xi}_1$ in \eqref{eq:def_wt_xi_2_1} 
and Lemma~\ref{lem:inequality_coefficients} \ref{item:xi_and_tilde_xi}. 
This completes the proof of Lemma~\ref{lem:rough_bound}. 
\end{proof} 

As in the proof of \cite[Theorem~2.5]{Cusp1}, we now record an intermediate local law in the next proposition.
It is obtained by following the proof of \cite[Proposition~3.11]{Cusp1} and employing Lemma~\ref{lem:cubic_inequality_refined} instead of \cite[Lemma~3.8]{Cusp1}. 

\begin{proposition}[Local law uniformly for $\eta \geq n^{-1+\zeta}$]\label{pro:local_law_all_scales} 
Let $\zeta>0$. On $\DD_\zeta$, we have the bounds 
\begin{equation} \label{eq:local_law_n_minus_1} 
\abs{G-M} \prec  \theta_* + \sqrt{\frac{\rho}{n \eta}} + \frac{1}{n\eta}, \qquad \abs{G-M}_\mathrm{av} \prec \theta_* + \frac{\rho}{n \eta} + \frac{1}{(n\eta)^2}, 
\end{equation}
where $\theta_*$ is defined through 
\[ \theta_* \defeq \min \{ d_*^{1/3}, d_*^{1/2}/\wt{\xi}_2^{1/2}, d_*/\wt{\xi}_1 \}, \qquad d_* \defeq \wt{\xi}_2 \bigg( \frac{\wt{\rho}}{n \eta} + \frac{1}{(n \eta)^2} \bigg) + \frac{1}{(n \eta)^3} + \bigg( \frac{\wt{\rho}}{n\eta} 
\bigg)^{3/2} \] 
and $\wt{\rho}$ denotes the right-hand side of \eqref{eq:scaling_rho}, i.e. $\rho \sim \wt{\rho}$. \qed
\end{proposition} 

Now, we follow the proof of \cite[Eq.~(3.28)]{Cusp1} and use Lemma~\ref{lem:inequality_coefficients} \ref{item:comparsion_tilde_xi_rho} instead of \cite[Lemma~3.3]{Cusp1} to obtain both bounds in \eqref{eq:G_minus_M_standard}. 

We now strengthen \eqref{eq:G_minus_M_standard} to \eqref{eq:G_minus_M_average_improved} outside of the spectrum. 
If $\abs{z}^2 > 1 + (n^\gamma \eta)^{2/3}$ then $\theta_* \leq d_*/\wt{\xi}_1$, Lemma~\ref{lem:inequality_coefficients} \ref{item:comparsion_tilde_xi_rho}, $(\wt{\xi}_2)^2 = \wt{\xi}_1$ and \eqref{eq:scaling_rho} imply
\begin{equation} \label{eq:proof_local_law_H_aux1} 
\theta_* + \frac{\wt{\rho}}{n \eta} + \frac{1}{(n\eta)^2} \lesssim \frac{\wt{\xi}_2}{\wt{\xi}_1} \bigg( \frac{\rho}{n\eta} + \frac{1}{(n\eta)^2} \bigg) 
\lesssim \frac{1}{(\abs{z}^2 - 1)^{1/2}} \bigg( \frac{\eta}{\abs{z}^2 - 1} + \frac{1}{n \eta} \bigg)\frac{1}{n\eta}  \lesssim \frac{n^{-\gamma/3}}{n\eta} 
\end{equation} 
for $\eta \geq \etaf(z) \gtrsim n^{-3/4}$ (cf.~\cite[Eq.~(3.30)]{Cusp1}).  
Applying \eqref{eq:proof_local_law_H_aux1} to the second bound in \eqref{eq:local_law_n_minus_1} yields the improved bound \eqref{eq:G_minus_M_average_improved} and, thus, completes the proof of Theorem~\ref{thm:local_law_H}. 
\end{proof} 

From Proposition~\ref{pro:local_law_all_scales}, we now conclude Corollary~\ref{cor:delocalization}. 

\begin{proof}[Proof of Corollary~\ref{cor:delocalization}]
The bound on $\abs{G-M}$ in \eqref{eq:local_law_n_minus_1} and the argument from \cite[Corollary~1.14]{Ajankirandommatrix} directly imply Corollary~\ref{cor:delocalization}. 
\end{proof}

We conclude this subsection by collecting two simple consequences of the previous results  and \cite{Altcirc}. For the remainder of Section~\ref{subsec:proof_local_law_H}, $\tau>0$ will be a parameter bounding the spectral parameter $z$ from above. The implicit constant in $\prec$-estimates is allowed to depend on $\tau$.

\begin{corollary} \label{cor:im_local_law_everywhere} 
Let $X$ satisfy \ref{assum:flatness} and \ref{assum:bounded_moments}. Let $\zeta>0$. Then we have 
\[ \abs{\Im \avg{G(z,\eta)-M(z,\eta)}} \prec \frac{1}{n\eta} \] 
uniformly for all $z \in D_{\tau}(0)$ and all $ \eta \in [n^\zeta \etaf(z), n^{100}]$.  
\end{corollary} 

\begin{proof} 
The corollary is a direct consequence of the local law near the edge, \eqref{eq:G_minus_M_standard}, the 
local law away from the edge, \cite[Eq.~(5.4)]{Altcirc}, and the definition of $\etaf$ in \eqref{eq:def_etaf}.  
\end{proof} 

We denote the eigenvalues of $H_z$ by $\lambda_1(z), \ldots, \lambda_{2n}(z)$. 
The following lemma provides a simple bound on the number of eigenvalues of $H_z$ in the interval $[-\eta,\eta]$. It is an extension of \cite[Eq.~(5.22)]{Altcirc} to the edge regime. 

\begin{lemma}[Eigenvalues of $H_z$ near 0] \label{lem:number_eigenvalues} 
Let $X$ satisfy \ref{assum:flatness} and \ref{assum:bounded_moments}. Let $\zeta>0$. Then we have 
\[ \abs{ \{ i \in [2n] \colon \abs{\lambda_i(z)} \leq \eta \}} \prec n \eta \rho + 1\] 
uniformly for all $z \in D_\tau(0)$ and for all $\eta \in [n^\zeta\etaf(z),n^{100}]$. 
\end{lemma}

\begin{proof} 
We define $\Lambda_\eta \defeq \{ i \in [2n] \colon \abs{\lambda_i(z)} \leq \eta \}$. 
For $\eta \geq n^{\zeta}\etaf$, we obtain from Corollary~\ref{cor:im_local_law_everywhere} and $\rho = \avg{\Im M}/\rho$ that 
\[ \frac{\abs{\Lambda_\eta}}{2\eta} \leq \sum_{i \in \Lambda_\eta} \frac{\eta}{\eta^2 + \lambda_i(z)^2} \leq 2n \Im \avg{G(z,\eta)} \prec n \bigg(\rho + \frac{1}{n \eta} \bigg) \lesssim n \rho + \frac{1}{\eta}. \] 
This completes the proof of Lemma~\ref{lem:number_eigenvalues}. 
\end{proof}

\subsection{Cusp fluctuation averaging -- Proof of Proposition~\ref{pro:D_bounds}} \label{subsec:proof_D_bounds} 

In this subsection we will provide the proof of Proposition~\ref{pro:D_bounds}. Since the self-consistent density of states of the Hermitian matrix $H=H_z$  develops a cusp singularity at the origin in the regime $\abs{\abs{z}-1}\ll 1$, this result is analogous to \cite[Theorem~3.7]{Cusp1}, which provides an improved bound for specific averages of the random error matrix in the MDE. This improved bound is called \emph{cusp fluctuation averaging} and takes the form \eqref{eq:D_bound_cusp_FA} in the current work. In \cite{Cusp1} the expectation $\E H$ was  diagonal and the self-energy operator was assumed to satisfy the flatness condition \cite[Eq.~(3.6)]{Cusp1}. Both conditions are violated in our current setup and thus the result from \cite{Cusp1} is not directly applicable. However, with minor modifications the proof of \cite[Theorem~3.7]{Cusp1} can be adjusted to yield  Proposition~\ref{pro:D_bounds}. In fact, the cancellation that underlies the cusp fluctuation averaging \eqref{eq:D_bound_cusp_FA} is simpler and more robust for $H$ with the bipartite structure \eqref{eq:def_H_z}. An indication of this fact is that \eqref{eq:D_bound_cusp_FA} holds for any $R \in \oM$ while the corresponding bound in \cite[Theorem~3.7]{Cusp1} only holds when the error matrix is averaged against a specific vector that depends on  
 $M$, the solution to the MDE.

For the purpose of being able to follow the strategy from \cite{Cusp1} very closely we define 
\begin{equation}\label{eq:quantities from Cusp1}
W\defeq H_z-\E H_z\,, \qquad \wt{\cal{S}}[R]\defeq \E W RW\,, \qquad \wt{D} \defeq WG+ \wt{\cal{S}}[G]G\,.
\end{equation}
The modified self-energy operator $\wt{\cal{S}}$  is introduced to match the convention of \cite{Cusp1} (cf.\ Remark~\ref{rem:convention_for_S}). 
 This differs from the self-energy operator $\cal{S}$ defined in this paper (see \eqref{eq:def_cS} and Remark~\ref{rem:convention_for_S}), which is
the  block diagonal part of $\wt{\cal S}$, consisting
of the blocks  $ \E  X R_{22} X^*$ and  $ \E  X^* R_{11} X$, both themselves being diagonal since $X$ has independent entries. 
The difference between the two versions of the self-energy is 
\begin{equation}\label{def of cal T}
\cal{T}[ R ] \defeq (\wt{\cal{S}}-\cal{S} )[R]=
\left(
\begin{array}{cc}
0 & \E X R_{21} X 
\\
\E X^* R_{12}X^*&0
\end{array}
\right)
=  \left(
\begin{array}{cc}
0 & \smallT  \odot R_{21}^t 
\\
 {\smallT}^*  \odot R_{12}^t &0
\end{array}
\right)
= T \odot R^t
\,,
\end{equation}
where $\odot$ indicates the entrywise (Hadamard) matrix product and we introduced the matrices $\smallT=(\E x_{ij}^2)_{i,j=1}^n \in \C^{n\times n}$ and $T = (t_{ij})_{i,j=1}^{2n}\in\C^{2n\times 2n}$ with entries 
$t_{ij} = \E w_{ij}^2$ being the second moments of the entries of $W =(w_{ij})_{i,j=1}^{2n}$ from \eqref{eq:quantities from Cusp1}.  
The modified error matrix $\wt{D}$ was labelled $D$ in \cite{Cusp1} and is the natural error when considering the MDE with self-energy $\wt{\cal{S}}$ and corresponding solution $\wt{M}$. In the current work we will stick to the convention from \cite{Altcirc} with respect to the definition of $\cal{S},D,M$ in order to keep the MDE and its solution simple and  thus we indicate the corresponding quantities  $\wt{\cal{S}},\wt{D},\wt{M}$ from \cite{Cusp1} by a tilde. 
 Another notational difference is that the dimension of $H$ was denoted by $N$ in \cite{Cusp1}, that corresponds to $N=2n$ in this paper.

We start the proof of  Proposition~\ref{pro:D_bounds} by showing that it suffices to establish its statement for $D$ replaced by $\wt{D}$. Let us therefore assume the following proposition whose proof is the main content of this subsection. 
\begin{proposition} \label{pro:wtD_bounds} 
The statement of Proposition~\ref{pro:D_bounds} holds with $D$ replaced by $\wt{D}$, i.e. under the same assumptions and with the same constants we have the estimates
\begin{subequations} 
\begin{align} 
\norm{\scalar{\bx}{ \wt{D}\by}}_p & \leq_{\eps,p} \norm{\bx} \norm{\by} n^\eps  \psi_q' \Big( 1 + \norm{G}_q \Big)^C \bigg( 1 + \frac{\norm{G}_q}{\sqrt{n}} \bigg)^{Cp}, \label{eq:wtD_bound_isotropic} \\ 
\norm{\avg{R \wt{D}}}_p & \leq_{\eps,p} \norm{R} n^\eps   \Big[\psi_q'\Big]^2 \Big( 1 + \norm{G}_q \Big)^C \bigg( 1 + \frac{\norm{G}_q}{\sqrt{n}} \bigg)^{Cp},\label{eq:wtD_bound_average} 
\end{align} 
and for $R \in \oM$ the improved estimate 
\begin{equation} \label{eq:wtD_bound_cusp_FA} 
\norm{\avg{R \wt{D}}}_p \leq_{\eps,p} \norm{R} n^\eps  \sigma_q \Big[\psi + \psi_q'\Big]^2 \Big( 1 + \norm{G}_q \Big)^C \bigg( 1 + \frac{\norm{G}_q}{\sqrt{n}} \bigg)^{Cp}. 
\end{equation}
Furthermore, with $\cal{T}$ from \eqref{def of cal T}  and for an arbitrary deterministic matrix $R$, 
\begin{equation}\label{eq:cal T GG}
\norm{\avg{R\cal{T}[G]G} }_p \leq_{\eps,p} \norm{R} n^\eps  \sigma_q \Big[\psi + \psi_q'\Big]^2 \Big( 1 + \norm{G}_q \Big)^C \bigg( 1 + \frac{\norm{G}_q}{\sqrt{n}} \bigg)^{Cp}. 
\end{equation}
\end{subequations} 
\end{proposition} 
Given the bounds from Proposition~\ref{pro:wtD_bounds} it suffices to  estimate the difference  $\wt{D}-D = \cal{T}[G]G$. First \eqref{eq:D_bound_isotropic} follows from \eqref{eq:wtD_bound_isotropic} because for normalized vectors $\bx, \by \in \C^{2n}$ we have 
\[
\norm{\scalar{\bx}{\cal{T}[G]G \by}}_p =   \normb{{\textstyle \sum_i} G_{i \bv_i }G_{i\by}}_p 
\lesssim  n^\eps \norm{G}_{1/\eps}\pbb{\frac{\norm{\im G}_p}{n \eta}}^{1/2}\le n^\eps \norm{G}_{q}\psi_q',
\]
where in the equality we introduced  the vectors $\bv_i = (t_{ji}\ol{x}_j)_j $ with $\norm{\bv_i}_\infty \lesssim \frac{1}{n} $ and in the first inequality we used the Ward identity in the second factor after applying  the general bound $\norm{\sum_i X_i Y_i}_p \le  n^\eps \sup_i\norm{X_i}_{1/\eps}\norm{\sum_i  \abs{Y_i}}_{2p}$ for any random variables $(X_i,Y_i)_{i=1}^{2n}$ and $\eps \in (0,1/2p)$.
Then \eqref{eq:wtD_bound_average} implies \eqref{eq:D_bound_average} by taking the $\norm{\2\cdot\2}_p$-norm on both sides of 
\[
\abs{\avg{R\cal{T}[G]G}} \lesssim \frac{\norm{R}}{n} \avg{G^*G} = \norm{R}\frac{\avg{\im G}}{n\eta} \,,
\]
where we used $\normtwoop{\cal{T}}\lesssim \frac{1}{n}$. 
Finally, \eqref{eq:D_bound_cusp_FA} immediately follows from \eqref{eq:wtD_bound_cusp_FA} and \eqref{eq:cal T GG}.

The remainder of this subsection is dedicated to proving Proposition~\ref{pro:wtD_bounds}. To avoid repetition we will only point out the necessary modifications to the proof of \cite[Theorem~3.7]{Cusp1}.

The proof of \eqref{eq:wtD_bound_isotropic} and \eqref{eq:wtD_bound_average} is exactly the same as the proof of \cite[Eq.'s (3.11a) and (3.11b)]{Cusp1}, which,
in turn directly follow from  \cite[Theorem 4.1]{Erdos2017Correlated}. Note that this latter theorem does not assume flatness, i.e.\ 
lower bound on $\wt{\cS}$, hence it is directly applicable to our $H$ as well. We also remark that the proof of \eqref{eq:wtD_bound_isotropic} and \eqref{eq:wtD_bound_average}
requires only double index graphs (in the sense of \cite{Erdos2017Correlated}) and their estimates rely only on the
power counting of Wardable edges. A self-contained summary of the necessary concepts can be found in 
\cite[Section~4.1-4.7]{Cusp1}, where the quite involved cumulant expansion from \cite{Erdos2017Correlated}, originally designed to handle
any correlation, is translated into the much simpler independent setup. This summary  in \cite{Cusp1} has the advantage that it also introduces
the single index graphs as a preparation for the more involved $\sigma$-cell estimates needed for the cusp fluctuation averaging. 

In the rest of the proof we focus on \eqref{eq:wtD_bound_cusp_FA} and \eqref{eq:cal T GG} and we assume that the reader is
 familiar with \cite[Section~4]{Cusp1}, but no familiarity with \cite{Erdos2017Correlated} is assumed. We will exclusively work
 with single index graphs 
as defined in \cite[Section~4.2]{Cusp1}. In the rest of this section 
we use $N=2n$ for easier comparison with \cite{Cusp1}.

We start with the proof of \eqref{eq:wtD_bound_cusp_FA}. 
We write $R \in \oM $ as $R = J\diag(\br)$ for some $\br \in \R^{2n}$, where we can without loss of generality assume that $R$ has real entries and the matrix $J$ that exchanges $\oM$ and $\dM$ is 
\[
J \defeq \left(
\begin{array}{cc}
0 & 1 
\\
1& 0
\end{array}
\right)\,.
\]
With this notation the left-hand side of \eqref{eq:wtD_bound_cusp_FA}  takes the form $\avg{\diag(\br) \wt{D}J}= \avg{\diag(\br)(W + \wt{\cal{S}}[G])GJ}$. This form exactly matches the left-hand side of \cite[Eq.~(3.11c)]{Cusp1}
with $\br = \mathbf{pf}$, 
 except that the last factor inside the trace is $GJ$ instead of just $G$. To accommodate this change we slightly extend the set of single index graphs $\Gamma \in \cal{G}$ defined  in \cite[Section~4.2]{Cusp1} by allowing two additional types of $G$-edges in $\rm{GE}=\rm{GE}(\Gamma)$. We call the original $G$-edges from  \cite{Cusp1}, that encode the entries of $G$ and $G^*$, \emph{straight $G$-edges} and add  new \emph{twisted $G$-edges} that represent the entries of $GJ$ and $(GJ)^* = JG^*$, respectively.
  Graphically $(GJ)_{ab}$  will be encoded by a solid directed line from vertex $a$ to vertex $b$ and with a superscript $J$ on the
 line. Similarly, $(GJ)^*_{ab}$ is a dashed line from $a$ to $b$ with a superscript $J$. Hence, the new twisted $G$-edges are represented by 
\[ GJ = \ssGraph{a --[g,J] b;}\, , \qquad (GJ)^* = \ssGraph{a --[s,J] b;}\,. \] 
  The terminology $G$-edge will refer to all four types of edges. 
In particular, all of them are taken into account for the $G$-edge degree of vertices.

According to the single index graph expansion (cf.\ \cite[Eq.~(4.10)]{Cusp1}) the $p$-th moment of  $\avg{R \wt{D}}$ can now be written as a sum over the values $\rm{Val}(\Gamma)$ associated to the graphs $\Gamma$ within the subset $\cal{G}(p) \subset \cal{G}$ of single index graphs. This subset originates from the single index resolution (cf.\ \cite[Definition~4.2]{Cusp1}) of double index graphs, i.e.
\begin{equation}
\label{eq:graph expansion}
\E \abs{\avg{\diag(\br) \wt{D}J}}^p = N^{-p} \sum_{\Gamma \in \cal{G}(p)}\rm{Val}(\Gamma)+ \ord(N^{-p})\,.
\end{equation}
The twisted $G$-edges enter into the graphs $\cal{G}(p)$ through the following simple modification (iv)' of \cite[(iv) from Definition~4.2]{Cusp1} that originates from the fact that a wiggled $G$-edge in  double index graphs is now associated to the matrix $GJ\diag(\br)$ and its adjoint instead of $G\diag(\br)$ with $\br = \mathbf{pf}$ as in \cite{Cusp1}:
\begin{itemize}
\item[(iv)'] If a wiggled $G$-edge is mapped to an edge $e$ from $u$ to $v$, then $v$ is equipped with a weight of $\br$ and $e$ is twisted. If a  wiggled  $G^*$-edge is mapped to an edge $\wt{e}$ from $u$ to $v$, then  $u$ is equipped with weight $\br$ and $\wt{e}$ is twisted. All vertices with no weight specified in this way  are equipped with constant weight $\mathbf 1$. 
\end{itemize}
The above changes reveal a one-to-one correspondence between the set of graphs $\cal{G}(p)$ in \cite{Cusp1} and its modification in the current work. 
This correspondence shows that the single index graph expansion is entirely unaffected by the presence of the off-diagonal matrix $J$ apart from replacing each 
weight $\mathbf{pf}$ in $\cal{G}(p)$ from \cite{Cusp1} by a weight $\br$ and replacing $p$  straight $G$-edges by twisted ones. 
More precisely, if in a graph from \cite{Cusp1}  a vertex $v$ had a weight $(\mathbf{pf})^k$ (see \cite[Fact 2 at the end of Section~4.3]{Cusp1} for the introduction of the concept of a general weight and \cite[Eq.~(4.12)]{Cusp1} for an example), then in its corresponding graph the vertex $v$ is adjacent to exactly $k_1$ twisted $G$-edges  that end at $v$ and $k_2$ twisted $G^*$-edges that start from $v$ such that $k_1 +k_2=k$.

Since the  graphs contained in the sets $\cal{G}$ and $\cal{G}(p)$ do not differ between the current work and \cite{Cusp1} once the distinction between straight and twisted edges is dropped and the exact form of the weight $\br$ is irrelevant, any result from \cite[Section~4]{Cusp1} that is insensitive to these distinctions can be directly applied here. 
When determining whether a subset $\rm{GE}_W \subset \rm{GE}$ is classified as Wardable (cf.\ \cite[Definition~4.6]{Cusp1}), 
we take into account all $G$-edges, i.e.\ straight as well as twisted $G$-edges. This is justified since the Ward estimates (cf.\ \cite[Eq.~(4.14b)]{Cusp1})
\bels{Ward estimates}{
\sum_a \abs{(GJ)_{ab}}^2 = \frac{\big(J(\im G)J\big)_{bb}}{\eta}\lesssim N \psi^2\,, \quad \sum_b \abs{(GJ)_{ab}}^2 = \frac{\im G_{aa}}{\eta}\lesssim N \psi^2
\,, \quad \psi= \bigg(\frac{\rho}{N \eta}\bigg)^{1/2}
}
are valid  for twisted $G$-edges as well.  As in \cite{Cusp1} the inequalites \eqref{Ward estimates} are meant in a high moment sense. 
In particular, \cite[Lemmas~4.7, 4.8 and 4.11]{Cusp1} 
remain valid. 
Note that \cite[Lemma 4.11]{Cusp1} involves the concept of a $\sigma$-cell that we explain next. 

The most relevant difference between our setup and \cite{Cusp1} concerns the specific mechanism behind  the cusp fluctuation averaging. This mechanism is revealed by exploiting a local cancellation within the graphs $\cal{G}$ appearing along the expansion that is associated to specific edges, called $\sigma$-cells (cf.\ \cite[Definition~4.10]{Cusp1}). For the following discussion we recall the definition of $\sigma$-cells and rephrase it so that it fits our setup.

\begin{definition} A $\sigma$-cell is  an interaction edge $e=(a,b)$ inside a graph $\Gamma \in \cal{G}$ such that
there are exactly two $G$-edges adjacent to each endpoint of $e$, loops are not allowed, and to precisely one
of the endpoints (say $a$) an additional weight   $\br$ is attached. In a typical $\sigma$-cell there are four $G$-edges that
connect external   vertices $x,y,u,v$ with $(a,b)$ and the adjacent $G$-edges encode the expression 
\begin{equation}\label{eq:sigma cell}
\E \2\sum  r_a (\wt{G}J)_{xa}G^{(1)}_{ay}K_{ba}G^{(2)}_{ub}G^{(3)}_{vb}f_{xyuv}\,.
\end{equation}
Here $f_{xyuv}$ represents the rest of the graph and is independent of $a,b$; the sum runs over all vertex indices and $K_{ba}$ is either $\E \abs{ w_{ba}}^2$, 
$\E  w_{ba}^2$ or $\E {w}_{ab}^2$. Furthermore, $\wt{G} \in \{G, \ol{G}\}$ and $G^{(i)} \in \{G,G^t,G^*,\ol{G}\}$ 
(here $G^t$ and $\ol{G}$  just denote a $G$- or a $G^*$-edge with opposite orientation to save
us from  writing out all possibilities in \eqref{eq:sigma cell}).
Some $G$'s in \eqref{eq:sigma cell} may coincide giving rise to  two additional options for a $\sigma$-cell, the first one with two
external indices, and the second one with no external index:
\begin{equation}\label{eq:sigma cell1}
\E \2\sum  r_a (\wt{G}J)_{xa}G^{(1)}_{ab}K_{ba}G^{(2)}_{ub}f_{xu}\,, \qquad \mbox{and}\qquad
 \E \2\sum  r_a (\wt{G}J)_{xa}G^{(1)}_{ab}K_{ba}f\,.
\end{equation}
\end{definition}
The graphical representation of these three types of  $\sigma$-cells is the same as drawn in \cite[Definition~4.10]{Cusp1}
with weight $\br =\mathbf{pf}$, except that one $G$-edge adjacent to $a$ is twisted.  For example, the $\sigma$-cell with four external
indices \eqref{eq:sigma cell} is represented by 
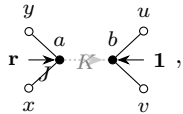
\[ 
\ssGraph{ a[label=$a$,lbr] --[K] b[label=$b$,r1]; x[label=below:$x$,o] --[J] a -- y[label=$y$,o]; v[label=below:$v$,o] -- b --u[label=$u$,o]; }, 
\] 
 where the solid lines are $G$-edges, exactly one of them twisted (indicated by $J$), and without indicating their orientation.  The interaction
 edge $K$ is depicted by the dotted line, while the weights $\br$ and $\mathbf{1}$ attached to vertices are indicated 
 by arrows pointing to these vertices. The weight $\mathbf{1}$ could be ignored, it plays no specific role in the current paper; we drew it 
 only for consistency with the picture in \cite{Cusp1} where it was essential that exactly one edge of the $\sigma$-cell receives a specific weight.
 The graphical picture of the other two types of $\sigma$-cells are analogous.

Exactly as in  \cite{Cusp1} the cusp fluctuation mechanism will allow us to gain a factor $\sigma_q$ as defined in Proposition~\ref{pro:D_bounds} for every $\sigma$-cell inside each graph. In \cite{Cusp1} this gain is stated as \cite[Proposition~4.12]{Cusp1}. 
In our setup, its analog is the next proposition.

\begin{proposition} \label{pro:proposition_4_12} 
Let $c>0$ be any constant and $\Gamma \in \mathcal G$ be a single index graph with at most $cp$ vertices and $cp^2$ edges with a $\sigma$-cell $(u,v) = e \in \mathrm{IE}(\Gamma)$. Then there exists a finite collection of graphs $\mathcal G_\Gamma$ with at most one additional vertex and at most $6p$ additional $G$-edges such that 
\[ 
 \mathrm{Val}(\Gamma) = \sum_{\Gamma' \in \mathcal G_{\Gamma}} 
\mathrm{Val}(\Gamma') + \ord(N^{-p}), 
\qquad \qquad 
\mathrm{W}\text{-}\mathrm{Est}(\Gamma') \leq_p \sigma_q \mathrm{W}\text{-}\mathrm{Est}(\Gamma'), \quad \Gamma' \in \mathcal G_\Gamma
\] 
and all graphs $\Gamma' \in \mathcal G_\Gamma$ have exactly one $\sigma$-cell less than $\Gamma$. 
\end{proposition} 

The statement of Proposition~\ref{pro:proposition_4_12} coincides with \cite[Proposition~4.12]{Cusp1} 
except for the missing term $\sigma \mathrm{Val}(\Gamma_\sigma)$ in the expansion of $\mathrm{Val}(\Gamma)$ originating from $\sigma = 0$ (cf.\ Remark~\ref{rem:sigma})
and 
 the modified definition of $\sigma_q$ (see Proposition~\ref{pro:D_bounds}). 
Up to \cite[Proposition~4.12]{Cusp1} (which is replaced by Proposition~\ref{pro:proposition_4_12}), we have now verified  all  ingredients in the proof of \cite[Theorem 3.7]{Cusp1} and thus its analog Proposition~\ref{pro:wtD_bounds}. Therefore, we will finish this subsection by pointing out the necessary modifications to the proof of \cite[Proposition~4.12]{Cusp1} in order to show Proposition~\ref{pro:proposition_4_12}.

\begin{proof}[Proof of Proposition~\ref{pro:proposition_4_12}]  
We follow the proof of \cite[Proposition~4.12]{Cusp1} and explain the necessary changes. 
The proof of \cite[Proposition~4.12]{Cusp1}  
has two ingredients.  The first is an explicit computation that involves the projections on stable and unstable directions of the stability operator $B$ of the MDE
(cf.\ \cite[Eq.~(4.30)]{Cusp1}).  This computation is extremely delicate and involves the precise choice for $K_{ba}$, $r_a$,  $G^{(i)}$ and their relation in the $\sigma$-cell \eqref{eq:sigma cell}.
Its outcome is that up to a sufficiently small error term it is possible to act with the stability operator on any vertex $a$ of the $\sigma$-cell. This action of $B$ on $a$ leads to an improvement of the bound on the corresponding graph that is stated as \cite[Lemma~4.13]{Cusp1}. 

In our current setup the gain $\sigma_q$ for every $\sigma$-cell inside a graph $\Gamma$ is much more robust than in \cite{Cusp1},
it is basically a consequence of the almost off-diagonality of $M$. There is no need to act with the stability operator
and also the specific weights attached to the vertices of the $\sigma$-cells are not important.
 Instead, the value of any graph containing a $\sigma$-cell 
 can be estimated directly by $\sigma_q$ times the  sum of values of graphs with one sigma cell locally resolved (removed).
For this gain the concrete choice of $K_{ab}$, $r_a$  and $G^{(i)}$  does not matter as  
 long as $\abs{K_{ba}} \lesssim \frac{1}{N}$ and $r_a \lesssim 1$. 
Furthermore, we will not make use of resolvents $G^{(2)}$ and $G^{(3)}$ in the corresponding calculations. 
Thus in the following, instead of \eqref{eq:sigma cell}, we will only consider the simplified expression 
\begin{equation}\label{halfcell}
\E\sum (\wt{G}J K^{(b)} G^{(1)})_{xy}f_{xyb} =\E\sum_a (\wt{G}J)_{xa} k^{(b)}_a G^{(1)}_{ay}f_{xyb}\, ,
\end{equation}
 where $K^{(b)} \defeq \diag (\bk^{(b)})$ is a diagonal matrix whose diagonal $\bk^{(b)}$ has components $k^{(b)}_a\defeq r_a K_{ba}$,
and $f = f_{xyb}$ encodes the rest of the graph.
 This is exactly the reference graph $\Gamma$ 
at the beginning of the proof of \cite[Lemma~4.13]{Cusp1}, but now the left edge is twisted and a weight $\bk^{(b)}$
is attached to the vertex  $a$.  With the choice $\wt{G}=G$, $G^{(1)} = G^*$ we have  
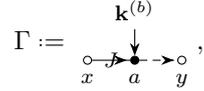
\[ \Gamma \defeq \ssGraph{x[label=below:$x$,o] --[g,J] a[label=below:$a$,tbk]; a -- [s] y[label=below:$y$,o];}, \] 
which corresponds to the case $GJK^{(b)}G^*$ in \eqref{halfcell}.  Since for all possible choices $\wt{G} \in \{G, \ol{G}\}$ and $G^{(1)} \in \{G,G^t,G^*,\ol{G}\}$ the discussion is analogous, we will restrict ourselves to  the case $GJK^{(b)}G^*$.
 
In complete analogy to \cite[Eq.~(4.24)]{Cusp1}, but  using simply the identity operator instead of the stability operator $B$
and  inserting the identity 
\[
G=M -G\cal{S}[M]M-GWM
\]
for the  first   resolvent factor on the left into \eqref{halfcell},  we  find by \eqref{def of cal T} that 
\begin{equation}
\label{eq:sigma cell expansion}
 \E (GJK^{(b)}G^*)_{xy}f = \E \Big( \sb{1+G\cal{S}[G-M]-G(\wt{\cal{S}}[G]+W)+G\cal{T}[G-M]+G\cal{T}[M]}(MJK^{(b)}G^*)\Big)_{xy} f\,.
\end{equation}
Notice that the twisted $G$-edge corresponding to $GJ$ disappeared and $J$ now appears  only together with $M$ in the form $MJK^{(b)}$.

We estimate the  five summands inside the square brackets of \eqref{eq:sigma cell expansion}.
This means to show that their  power counting estimate (defined as W-Est in \cite[Lemma~4.8]{Cusp1}) is smaller than 
the W-Est of the left-hand side of \eqref{eq:sigma cell expansion}, $\mbox{W-Est}(\Gamma)$, at least by a factor $\sigma_q$,  i.e. we have
\bels{graph expansion}{
\rm{Val}(\Gamma) = \sum_{\Gamma' \in \cal{G}_\Gamma}\rm{Val}(\Gamma') + \ord(N^{-p}) \quad \text{with} \quad  \mbox{W-Est}(\Gamma')\le_p \sigma_q\mbox{W-Est}(\Gamma)\,,
}
where all graphs $\Gamma' \in \cal{G}_\Gamma$ have one $\sigma$-cell less than $\Gamma$.  
Note that in contrast to  \cite[Lemma~4.13]{Cusp1} no insertion of the stability operator $B$ is needed for \eqref{graph expansion} to hold and that in contrast to   \cite[Proposition~4.12]{Cusp1}   the additional  graph $\Gamma_\sigma$ is absent 
  from the right-hand side. In this sense \eqref{graph expansion} combines these 
  two statements from \cite{Cusp1} in a simplified fashion.

We remark that here the notion of differential edge (defined around \cite[Eq.~(4.27)]{Cusp1}) is understood to include twisted $G$-edges as well. 
The derivatives of the twisted edges follow the same rules as the derivatives of the untwisted edges with respect to
the matrix elements of $W$, for example
$$
    \frac{\partial}{\partial w_{ab}} (GJ) = G \Delta^{ab} (GJ),  \qquad \mbox{with}\quad (\Delta^{ab})_{ij} = \delta_{ia}\delta_{jb},
$$
i.e., simply one of the resulting two $G$-edges remains twisted.  In particular, the number of twisted edges remains unchanged.

 The  first, second and fourth  summands in \eqref{eq:sigma cell expansion} correspond to the graphs treated in parts (a), (b) and (c) inside the proof of \cite[Lemma~4.13]{Cusp1} and their estimates follow completely analogously.  We illustrate this with the simplest first and the more complex fourth term.
The first term gives 
$   \E (MJK^{(b)}G^*)_{xy} f$, 
which would exactly be case (a) in the proof of \cite[Lemma~4.13]{Cusp1} if $MJK^{(b)}$ were diagonal. 
However, even if it has an offdiagonal part 
(in the sense of  $\oM$), the same bound holds, i.e. we still have 
\begin{equation}\label{aterm}
  \mbox{W-Est}\Big( \sum_a (MJK^{(b)})_{xa} G^*_{ay} f\Big) \le \frac{1}{N\psi^2}  \mbox{W-Est}\Big( \sum_a (GJK^{(b)} G^*)_{xy}f \Big) = 
  \sigma_q  \mbox{W-Est}(\Gamma)\,,
\end{equation}
where $1/N$ comes from the fact that the summation over $a$ collapses to two values, $a=x$ and $a=\hat x$
 and $\psi^2$ accounts for the two Wardable edges in $\Gamma$. Here we defined $\hat x\defeq x+n \mbox{(mod $2n$)}$
 to be the complementary index of $x$.  In order to see \eqref{aterm} more systematically, 
 set $\bm_d \defeq \diag(M)$ and $\bm_o\defeq\diag(MJ)$
 to be the vectors representing the diagonal and offdiagonal parts of $M$. Then
 \begin{equation}\label{MJ}
   MJK^{(b)} = \diag( \bm_o \bk^{(b)}) + J\diag( \widetilde {\bm_d} \bk^{(b)})
   = \diag( \bm_o \bk^{(b)}) + \diag(  \bm_d  \widetilde {\bk^{(b)} })J\,,
 \end{equation}
 where for any $2n$ vector $\bv =(\bv_1, \bv_2)$ with $\bv_i\in \C^n$
 we define $\widetilde{\bv}\defeq (\bv_2, \bv_1)$. Thus, graphically, the factor $MJK^{(b)}$ 
 can be represented as a sum of two graphs with a weight assigned to one vertex and for one of them there is an additional twist operator $J$
 which one may put on either side. 
 Therefore, the graph on the left-hand side of \eqref{aterm} can be represented
 by 
\begin{equation}\label{Mres}
\ssGraph{x[label=below:$x$,o] -- [eq] a[label=below:$a$,tmok] -- [s] y[label=below:$y$,o];} + \ssGraph{x[label=below:$x$,o] -- [eq] a[label=below:$a$,mdtk] -- [s,J] y[label=below:$y$,o];}.
\end{equation}
Here the double  solid line depicts the identity operator as in \cite{Cusp1}. Both graphs  in \eqref{Mres}  are exactly the same as the one in 
   case (a) within the proof of \cite[Lemma~4.13]{Cusp1} with the only changes being the modified  vertex weights  $\bm_o \bk^{(b)}$ and $ \bm_d\widetilde{\bk}^{(b)}$ instead of just $\bm$ and the twisted $G^*$-edge in the second graph of \eqref{Mres}.
 To justify the Ward estimate \eqref{aterm} we use only the fact that $\norm{\bm_o}_\infty \lesssim 1$
 and $\norm{\bm_d}_\infty \lesssim 1$, although the latter estimate can be improved to $\norm{\bm_d}_\infty \lesssim \rho$. Here we denote by  $\norm{\bx}_\infty :=\max_{a \in [2n]}\abs{x_a}$ the maximum norm of a vector $\bx$.

The  decomposition~\eqref{MJ}  of $MJK^{(b)}$ into a sum of two terms, each with usual  weights  $\bm$ and one of them with 
a twisted edge, can be reinterpreted in all other cases. Thus, the second and fourth term 
in \eqref{eq:sigma cell expansion} can be treated analogously to the
 cases (b) and (c) within  the proof of \cite[Lemma~4.13]{Cusp1}.  In particular, the fourth term is split as 
 \begin{equation}\label{spl}
 \E (G \cal{T}[G-M]MJK^{(b)}G^*)_{xy} f = \E \sum  \pb{t_{ba}G_{xb} (G-M)_{ab}u_a G^*_{ay} + t_{ba}G_{xb} (G-M)_{ab}v_a (GJ)^*_{ay}}  f \,,
 \end{equation}
 for some bounded vectors $\bu,\bv$ with $\norm{\bu}_\infty+\norm{\bv}_\infty\lesssim 1$. Thus the corresponding two graphs exactly match the first graph depicted in (c) of the proof of \cite[Lemma~4.13]{Cusp1} with locally modified weights and the second one having a twisted $G^*$-edge.

For the third term in \eqref{eq:sigma cell expansion} we use a cumulant expansion  and find
\begin{equation}\label{eq:cumulant expansion}
\begin{split}
\E \2(G(\wt{\cal{S}}[G]+W)MJK^{(b)}G^*)_{xy} f &=  - \E \2(G\wt{\cal{S}}[MJK^{(b)}G^*]G^*)_{xy} f 
\\
&\qquad+ \sum_{a,c}\sum_{k=2}^{6p} \sum_{\mathbf{\beta} \in I^k} \kappa(ac,\underline{\beta}) \E \partial_{\underline{\beta}}[G_{xa}(MJK^{(b)}G^*)_{cy}f]
\\
&\qquad+ \sum_{a,c}  \E \, G_{xa}(MJK^{(b)}G^*)_{cy}(\E \abs{w_{ac}}^2\partial_{ac}+\E w_{ac}^2\partial_{ca}) f+ \ord(N^{-p})\,,
\end{split}
\end{equation}
where $\mathbf{\beta}$ is a $k$-tuple of double indices from $I=[N]\times[N]$,
 $\underline{\beta}$ is the multiset formed out of the entries  of $\mathbf{\beta}$
 (multisets allow repetitions) and $\partial_{\underline{\beta}} =  \prod_{(ij)\in\underline{\beta}} \partial_{w_{ij}}$.  The notation 
 $\kappa(ac,\underline{\beta})$ denotes the higher order cumulants of $w_{ab}$ and $\{ w_\beta\; : \; \beta\in \underline{\beta}\}$.   
The second and third summands on the right-hand side of \eqref{eq:cumulant expansion} correspond to the graphs treated in (e) and (d) of the proof of \cite[Lemma~4.13]{Cusp1}, 
respectively, with the factor $MJK^{(b)}$ reinterpreted as sum of two terms with 
some weight $\bm$   as explained in \eqref{MJ}.
  Thus we focus our attention on the first summand which reflects the the main difference between our setup and that in \cite{Cusp1}.

This term is expanded further  using
\begin{equation}\label{SM}
\wt{\cal{S}}[MJK^{(b)}G^*]   = \cal{S}[MJK^{(b)}M^*]+ \cal{S}[MJK^{(b)}(G-M)^*]+ {\cal{T}}[MJK^{(b)}G^*]\,.
\end{equation}
At this point the main mechanism behind the cusp fluctuation averaging is revealed by the fact that the leading 
term $\cal{S}[MJK^{(b)}M^*] = \diag(\bx)$ for some vector $\bx$ with $\abs{x_i} \lesssim \frac{\rho}{N}$, because the diagonal elements of $M$ are of order $\ord(\rho)$. 
This is the only place where the smallness of the diagonal elements of $M$ is used, in all other estimates
we used only the block diagonal structure of $M$ and the boundedness of  its matrix elements. 
The other two terms in \eqref{SM} are smaller order.
In fact, the second term exactly corresponds to the term encoded by the fourth graph on the right-hand side of \cite[Eq.~(4.29)]{Cusp1}, taking into account that this graph now splits into two due to   the decomposition~\eqref{MJ}  
with an extra twisted edge on one of the resulting graphs  similarly to \eqref{spl}. 
Therefore  this  term is estimated as explained in (b) of the proof of \cite[Lemma~4.13]{Cusp1}. The  term corresponding to the last summand in \eqref{SM} is analogous to the sixth  graph in \cite[Eq.~(4.29)]{Cusp1} 
 and thus treated as explained in (c) within the proof of \cite[Lemma~4.13]{Cusp1}.

Finally, we treat the last $G\cal{T}[M]$ term in \eqref{eq:sigma cell expansion} which  
was absent in \cite{Cusp1} and stems from the difference between the two self-energy operators $\cal{S}$ and $\wt{\cal{S}}$. This term leads to a contribution of the form $\E (GKG^*)_{xy} f $ with the block diagonal matrix $ L^{(b)}  \defeq \cal{T}[M]MJ K^{(b)}$ that satisfies $\norm{ L^{(b)}  }\lesssim N^{-1}$ in \eqref{eq:sigma cell expansion}  as $\norm{\mathcal T[M]} \lesssim N^{-1}$.  Thus the ensuing two graphs $\Gamma' \in \cal{G}_\Gamma$ have the same Ward estimates as $\Gamma$ but an additional $N^{-1}$-edge weight. 
This completes the proof of Proposition~\ref{pro:wtD_bounds}. 
\end{proof}

\section{Local law for $X$}\label{sec:Xlaw}

In this section, we provide the proofs of Theorem~\ref{thm:spectral_radius_X} and Theorem~\ref{thm:local_law_X}. We start with the proof of Theorem~\ref{thm:local_law_X} which, given Theorem~\ref{thm:local_law_H}, 
will follow a similar argument as the proof of \cite[Theorem~2.5(i)]{Altcirc}. 

In this section, the model parameters consist of $s_*$, $s^*$ from \ref{assum:flatness}, the sequence $(\mu_m)_m$ from \ref{assum:bounded_moments} and $\alpha$, $\beta$ from \ref{assum:bounded_density} 
as well as $a$ and $\varphi$ from Theorem~\ref{thm:local_law_X}. Therefore, the implicit constants in the comparison relation and the stochastic domination are allowed to depend on these parameters. 

\begin{proof}[Proof of Theorem~\ref{thm:local_law_X}] 
Let $T>0$. From \cite[Eq.~(2.15)]{Altcirc} and \eqref{eq:integration_by_parts_L_Delta_f} in the 
proof of Proposition~\ref{pro:properties_sigma}, we get that 
\begin{equation} \label{eq:main_decomposition} 
\begin{aligned} 
 \frac{1}{n} \sum_{i=1}^n f_{z_0,a}(\zeta_i) - \int_\C f_{z_0,a}(z) \sigma(z) \di^2 z = \, & \frac{1}{4\pi n} \int_\C \Delta f_{z_0,a}(z) \log \abs{\det(H_z - \ii T)} \di^2 z \\ 
 &  - \frac{1}{2\pi} \int_\C \Delta f_{z_0,a}(z) \int_0^T \Im \avg{G(z,\eta)- M(z,\eta)} \di \eta \, \di^2 z \\ 
 & + \frac{1}{2\pi} \int_\C \Delta f_{z_0,a}(z) \int_T^\infty \bigg( \Im \avg{M(z,\eta)} - \frac{1}{1 + \eta } \bigg) \di \eta \, \di^2 z. 
\end{aligned} 
\end{equation}
Here, we used that $\avg{v_1^\tau(\eta)|_{\tau = \abs{z}^2}} = \Im \avg{M(z,\eta)}$ and $m^z(\ii \eta) = \avg{G(z,\eta)}$, where $m^z$ is the Stieltjes transform of the empirical spectral measure of $H_z$ (see \cite[Eq.~(2.12)]{Altcirc}). 
We also employed $\int_{\C} \Delta f_{z_0,a}(z)\int_0^T (1+ \eta)^{-1} \dd \eta \dd^2 z = 0$ as $f \in C_0^2(\C)$.  

We remark that $\supp f_{z_0,a} \subset z_0 + \supp f \subset D_{2\varphi}(0)$. For the remainder of the proof, we choose $T \defeq n^{100}$. 
The same arguments used in the proof of \cite[Theorem~2.5]{Altcirc} to control the first and third term on the right-hand side of \eqref{eq:main_decomposition} for $\abs{z_0} \leq 1$ imply that 
those terms are stochastically dominated by $n^{-1 + 2a} \norm{\Delta f}_1$ for all $z_0 \in \C$ such that $\abs{z_0} \leq \varphi$. 

What remains is bounding the second term on right-hand side of \eqref{eq:main_decomposition}. To that end, we fix $z$ and estimate the $\di \eta$-integral by 
\[ I(z) \defeq \int_0^T \Im \avg{G(z,\eta)- M(z,\eta)} \di \eta  \] 
for $z \in D_{2\varphi}(0)$ via the following lemma which is an extension of \cite[Lemma~5.8]{Altcirc}. 

\begin{lemma} \label{lem:moment_bound_I} 
For every $\delta >0$ and $p \in \N$, there is a positive constant $C$, depending only on $\delta$ and $p$ in addition to the model parameters, such that 
\[ \sup_{z \in D_{2\varphi}(0)} \E \abs{I(z)}^p \leq C \frac{n^{\delta p}}{n^p}. \] 
\end{lemma} 

We postpone the proof of Lemma~\ref{lem:moment_bound_I} to the end of this section. Now, Lemma~\ref{lem:moment_bound_I} implies Theorem~\ref{thm:local_law_X} along the same steps 
used to conclude \cite[Theorem~2.5(i)]{Altcirc} from \cite[Lemma~5.8]{Altcirc}. This completes the proof of Theorem~\ref{thm:local_law_X}.  
\end{proof} 

Before we prove Lemma~\ref{lem:moment_bound_I}, we first conclude Theorem~\ref{thm:spectral_radius_X} from Theorem~\ref{thm:local_law_X} and the improved bound on $G-M$ in \eqref{eq:G_minus_M_average_improved} 
in Theorem~\ref{thm:local_law_H}. 

\begin{proof}[Proof of Theorem~\ref{thm:spectral_radius_X}] 
We first show the upper bound on $\varrho(X)$ as detailed in Remark~\ref{rem:upper_bound_spectral_radius}. 
Throughout the proof, we say that an event $\Xi =\Xi_n$ (in the probability space of $n\times n$ random matrices $X$) 
 holds with \emph{very high probability} if for each $D>0$, 
there is $C>0$ such that $\P ( \Xi) \geq 1 - C n^{-D}$ for all $n \in \N$. 

Let $\tau_* \sim 1$ be chosen as in Theorem~\ref{thm:local_law_H}. From \cite[Theorem~2.5(ii)]{Altcirc}, we conclude that the event 
\[ \Xi_* \defeq \big\{ \spec (X) \subset D_{1 + \tau_*} (0) \big \} \] 
holds with very high probability. 
Fix $\eps >0$. We set $\eta \defeq \etaf n^{\eps/12}$. We use \eqref{eq:G_minus_M_average_improved}
for a sufficiently fine grid of values of $z$, 
 $\abs{\avg{\Im M}}\lesssim n^{-\eps/6} /(n\eta)$ by \eqref{eq:scaling_rho} if $\abs{z} -1 \geq n^{-1/2+\eps}$, a union bound 
 over the grid elements and the Lipschitz-continuity of $G(z,\eta)$ as a function of $z$ 
to obtain that the event
\[\Xi_1 \defeq \big\{\avg{\Im G(z,\eta)} \leq n^{-\eps/8}/(n\eta) \text{ for all } z \in \C \text{ satisfying } n^{-1/2 + \eps} \leq \abs{z} -1 \leq \tau_* \big \} \] 
holds with very high probability. 
Let $\lambda_1, \ldots, \lambda_{2n}$ denote the eigenvalues of $H_z$. 
If $H_z$ has a nontrivial kernel, i.e. $\lambda_i = 0$ for some $i \in \{1, \ldots, 2n\}$ then we conclude $\avg{\Im G(z,\eta)} \geq 1/(2n\eta)$ for any $\eta >0$.  
Hence, $\Xi_1 \subset \Xi_2$, where we defined 
\[ \Xi_2 \defeq \big \{  0 \notin \spec(H_z) \text{ for all } z \in \C \text{ satisfying } n^{-1/2 + \eps} \leq \abs{z} -1 \leq \tau_* \big \}. \] 
On the other hand, $0 \in \spec (H_z)$ if and only if $z \in \spec(X)$. 
Since $\Xi_*$ and $\Xi_2$ hold with very high probability, this completes the proof of the upper bound on $\varrho(X)$, i.e. the one of Remark~\ref{rem:upper_bound_spectral_radius}. 

The corresponding lower bound on $\varrho(X)$ follows directly from Theorem~\ref{thm:local_law_X} with $\abs{z_0} = 1 - n^{-1/2 + \eps}$, $a = 1/2 - \eps$ and a suitably chosen test function $f$ which completes the proof of Theorem~\ref{thm:spectral_radius_X}.  
\end{proof}

\begin{proof}[Proof of Lemma~\ref{lem:moment_bound_I}] 
The proof proceeds analogously to the proof of \cite[Lemma~5.8]{Altcirc}. However, we have to replace the fluctuation scale in the bulk, $n^{-1}$, by the $z$-dependent fluctuation scale $\etaf$ defined in \eqref{eq:def_etaf}. 
Throughout the proof, we will omit the dependence of $G(z,\eta)$ and $M(z,\eta)$ on $z$ and $\eta$
from our notation and write $G= G(z,\eta)$ and $M = M(z,\eta)$. 
Similarly, we denote the $2n$ eigenvalues of $H_z$ by $\lambda_1, \ldots, \lambda_{2n}$. 

We will have to convert a few $\prec$-bounds into moment bounds. In order to do that, we will use the following straightforward estimate. 
Let $c>0$. Then, for each $\delta>0$ and $p \in \N$, there is $C$, depending on $c$, $\delta$ and $p$, such that any random variable $Y \geq 0$ satisfies 
\begin{equation} \label{eq:estimate_prec_to_moment}
Y \prec n^{-1}, ~ Y \leq n^c \qquad \Longrightarrow \qquad \E Y^p \leq C n^{p(-1 + \delta)}.  
\end{equation} 

We fix $\eps >0$, choose $l >0$ sufficiently large and decompose the integral in the definition of $I$ to obtain 
\begin{equation} \label{eq:decomposition_I} 
I(z) = \frac{1}{n} \sum_{\abs{\lambda_i} < n^{-l}} \log \bigg(1 + \frac{\etaf^2 n^{2\eps}}{\lambda_i^2}\bigg) + \frac{1}{n} \sum_{\abs{\lambda_i} \geq n^{-l} } \log \bigg( 1 + \frac{\etaf^2 n^{2\eps}}{\lambda_i^2} \bigg)
- \int_0^{\etaf n^\eps} \avg{\Im M} \di \eta + \int_{\etaf n^\eps}^T  \Im \avg{G- M}  \di \eta.  
\end{equation} 

We now estimate the $p$th moment of each term on the right-hand side of \eqref{eq:decomposition_I} individually.  
Exactly as in the proof of \cite[Lemma~5.8]{Altcirc}, we choose $l>0$ sufficiently large, depending on $\alpha$, $\beta$ from \ref{assum:bounded_density}
and  on $p$ such that the estimate on the smallest singular value of $X-z$ in \cite[Proposition~5.7]{Altcirc}, which holds uniformly for $z \in D_{2\varphi}(0)$, implies 
\[ \E \absbb{\frac{1}{n} \sum_{\abs{\lambda_i} < n^{-l}} \log \bigg( 1 + \frac{\etaf^2 n^{2\eps}}{\lambda_i^2} \bigg)}^p \leq n^{-p}. \] 
For the second term on the right-hand side of \eqref{eq:decomposition_I}, we distinguish the three regimes, $\abs{\lambda_i} \in [n^{-l}, \etaf n^{\eps}]$, $\abs{\lambda_i} \in [\etaf n^{\eps},\etaf n^{1/2}]$ and $\abs{\lambda_i} > \etaf n^{1/2}$. 
In the first regime, we use $\etaf^2 n^{2 \eps} \lambda_i^{-2}  \leq n^{2l +2}$ and Lemma~\ref{lem:number_eigenvalues} with $\eta = \etaf n^\eps$ and obtain 
\[ \frac{1}{n} \sum_{\abs{\lambda_i} \in [n^{-l}, \etaf n^\eps]} \log \bigg(1 + \frac{\etaf^2 n^{2\eps}}{\lambda_i^2} \bigg) \leq \frac{C \log n}{n} \abs{\{ i \colon \abs{\lambda_i} \leq \etaf n^\eps \}} \prec \frac{n^{2\eps}}{n}. \] 
We decompose the second regime, $\abs{\lambda_i} \in [\etaf n^{\eps}, \etaf n^{1/2}]$, into the union of the intervals 
$[\eta_k, \eta_{k+1}]$, where $\eta_k \defeq \etaf n^\eps 2^k$, $k=0, \ldots, N$ and $N \lesssim \log n$. 
Hence, $\log(1+ x) \leq x$ for $x >0$ yields 
\[ \frac{1}{n} \sum_{\abs{\lambda_i} \in [\etaf n^{\eps}, \etaf n^{1/2}]} \log \bigg( 1 + \frac{\etaf^2 n^{2\eps}}{\lambda_i^2} \bigg) \leq \frac{2}{n} \sum_{k=0}^N \sum_{\lambda_i \in [\eta_k, \eta_{k+1}]} \etaf^2 n^{2\eps} \lambda_i^{-2} \prec \frac{n^{3\eps}}{n}.  
\] 
In the third regime, $\abs{\lambda_i} > \etaf n^{1/2}$, we conclude from $\abs{\lambda_i} > \etaf n^{1/2}$ that 
\[ \frac{1}{n}\sum_{\abs{\lambda_i} >\etaf n^{1/2}} \log \bigg( 1 + \frac{\etaf^2 n^{2\eps}}{\lambda_i^2} \bigg) \leq \frac{1}{n} \sum_{\abs{\lambda_i} > \etaf n^{1/2}} \log \big( 1 + n^{-1 + 2\eps}\big) \leq \frac{2n^{2\eps}}{n}. 
\] 
Therefore, \eqref{eq:estimate_prec_to_moment} implies that the second term on the right-hand side of \eqref{eq:decomposition_I} satisfies the bound in Lemma~\ref{lem:moment_bound_I}. 

For the third term in \eqref{eq:decomposition_I}, we use \eqref{eq:scaling_rho} and distinguish the regimes in the definition of $\etaf$ in \eqref{eq:def_etaf} and obtain 
\[ \int_0^{\etaf n^\eps}  \avg{\Im M} \,\di \eta \sim \int_0^{\etaf n^\eps} \rho\, \di \eta \lesssim n^{-1 + 4\eps/3}. \]

Before estimating the last term in \eqref{eq:decomposition_I}, we conclude 
$\sup_{\eta \in [\etaf n^{\eps},T]} \eta \abs{\Im \avg{G-M}} \prec n^{-1}$ 
from Corollary~\ref{cor:im_local_law_everywhere}.
Here, we also used a union bound and the Lipschitz-continuity of $G$ and $M$ as functions of $\eta$ in the following sense: 
There is $c\geq 2$ such that $\norm{G(z,\eta_1) - G(z,\eta_2)} + \norm{M(z,\eta_1) - M(z,\eta_2)} \lesssim n^{c} \abs{\eta_1-\eta_2}$ 
for all $\eta_1, \eta_2 \geq n^{-1}$ and $z \in D_{2\varphi}(0)$. 
For $G$ this follows from resolvent identities. For $M$ this was shown in 
\cite[Corollary~3.8]{Kronecker}\footnote{Note that the operator norm on $\C^{2n\times 2n}$ induced by the Euclidean norm on $\C^{2n}$ was denoted by $\norm{\genarg}_2$ in \cite{Kronecker}.}. 
The bound $\sup_{\eta \in [\etaf n^\eps, T]} \eta \abs{\Im \avg{G- M}} \prec n^{-1}$ implies 
\[ \int_{\etaf n^\eps}^T \abs{\Im \avg{G-M}} \di \eta \prec n^{-1}. \] 
Owing to \eqref{eq:estimate_prec_to_moment} this implies the desired estimate on the $p$th moment of the last term on the right-hand side of \eqref{eq:decomposition_I}. 
This completes the proof of Lemma~\ref{lem:moment_bound_I}. 
\end{proof} 

\begin{remark}[Alternative to Assumption~\ref{assum:bounded_density}, modifications in the proof] \label{rem:alternative_changes_proof} 
Instead of Assumption~\ref{assum:bounded_density} we now assume the condition $\max_{i,j} \cal{L}(\sqrt{n}\, x_{ij}, t) \leq b$ from Remark~\ref{rem:alternative_A3} and we explain the necessary modifications in the proof of Theorem~\ref{thm:local_law_X}. 
The only place where Assumption~\ref{assum:bounded_density} is used in our entire argument was in estimating the first term in~\eqref{eq:decomposition_I}.
Under this new condition, it is a simple consequence of \cite[Theorem~1.1]{LTV_singular_value2019} that 
\begin{equation} \label{eq:abs_log_lambda_1_prec_1} 
  \abs{\log \lambda_1(z)} \prec 1,
\end{equation}
where we denoted by $\lambda_1(z)$ the smallest, nonnegative eigenvalue of $H_z$. 
With this bound at hand, the sampling method from \cite[Lemma~36]{tao2015} can be used to approximate the first $z$-integral 
in \eqref{eq:girko} (see also \eqref{eq:main_decomposition})  
by an average over $n^C$ many evaluations of the integrand in $z$ (the sampling method requires 
the additional condition $\|\Delta f\|_{L^{2+\epsilon}}\le n^C\| \Delta f\|_{L^1}$). 
Applying a union bound reduces Theorem~\ref{thm:local_law_X} to the local law for $H_z$, Theorem~\ref{thm:local_law_H}, with one fixed $z$.  Note that $I(z)$ does not have finite expectation in general, e.g. if the  distribution of the matrix elements of $X$ has atoms. Thus,
Lemma~\ref{lem:moment_bound_I}  as it is stated cannot be correct in general. Instead, we
need to  split the $\eta$-integration in the definition of $I(z)$
and we use \eqref{eq:abs_log_lambda_1_prec_1}  to control 
the regime $\eta \in [0,n^{-l}]$.  For the remaining  regime the very small singular values do not play any role and 
the above proof directly applies.
\end{remark}

\appendix 

\section{Derivation of cubic equations} \label{app:cubic_equation} 

The following general lemma determines the first few terms in the perturbative expansion of 
 the solution $Y$ to the matrix-valued quadratic equation $\cB[Y] - \cA[Y,Y] + { Z} =0$ in the regime, where ${ Z}$ is small. 
 Here, $\cB$ is a linear map and $\cA$ is a bilinear map on the space of matrices.  As explained in Section~\ref{sec:outline},
 the unstable directions of $\cB$ play a particular role. 
  The case of one unstable direction was treated in \cite[Lemma~A.1]{Cusp1} which was sufficient for analysing Wigner-type matrices with a flat variance matrix. 
In our current situation, a second unstable direction is present due to the specific block structure of the matrix $H_z$. Accordingly, in the next lemma, we need to treat both unstable directions, $B$ and $B_*$, separately. 
We are, however, in the special situation, where $Y$ is orthogonal to $E_-$ and 
$B_*$ is far from being orthogonal to $E_-$. This allows us to neglect the component of $Y$ in the 
direction of $B_*$ and arrive at a single cubic equation for $\Theta$, the coefficient of $Y$ in the $B$
direction, $Y=\Theta B+ \mathrm{error}$. The main result of this lemma is to determine
the coefficients of the cubic equation for $\Theta$.

\begin{lemma} \label{lem:cubic_equation_abstract} 
Let $\scalar{\genarg}{\genarg}$ be the Hilbert-Schmidt scalar product on $\C^{2n \times 2n}$ and 
$\norm{\genarg}$ an arbitrary norm on $\C^{2n\times 2n}$.  
Let $\cA \colon \C^{2n\times 2n}\times \C^{2n \times 2n} \to \C^{2n \times 2n}$ be a bilinear map such that 
 $\cA[R,T] = \cA[T,R]$ for all $R, T \in \C^{2n \times 2n}$. 
Let $\cB \colon \C^{2n \times 2n} \to \C^{2n \times 2n}$ be a linear operator with two simple eigenvalues $\beta$ and $\beta_*$ with associated left and right eigenvectors $\wh{B}$, $B$ and $\wh{B}_*$, $B_*$, respectively. 

For some $\lambda \geq 1$, we assume that 
\begin{equation} \label{eq:conditions_abstract_cubic} 
\norm{\cA} + \norm{\cB^{-1} \cQ} + \frac{1}{\abs{\scalar{\wh{B}}{B}}} + \norm{\scalar{\wh{B}}{\genarg}} + \norm{B} + 
\frac{1}{\abs{\scalar{\wh{B}_*}{B_*}}} + \norm{\scalar{\wh{B}_*}{\genarg}} + \norm{B_*} + \frac{1}{\abs{\scalar{E_-}{B_*}}} + \norm{\scalar{E_-}{\genarg}} \leq \lambda, 
\end{equation}
where $\norm{\genarg}$ denotes the norm on bilinear maps, linear operators and linear forms induced by the norm on $\C^{2n \times 2n}$. 

Then there is a universal constant $c>0$ such that for any $Y, { Z} \in \C^{2n \times 2n}$ with $\norm{Y} + \norm{{ Z}} \leq c\lambda^{-12}$ that satisfy the quadratic equation 
\begin{equation} \label{eq:quadratic_equation} 
 \cB[Y] - \cA[Y,Y] + { Z} = 0 
\end{equation}
with the constraint $\scalar{E_-}{Y} = 0$ the following holds: For any $\delta \in (0,1)$, the coefficient 
\[ \Theta \defeq \frac{\scalar{\wh{B}}{Y}}{\scalar{\wh{B}}{B}} \] 
fulfills the cubic equation 
\begin{equation} \label{eq:abstract_cubic} 
\begin{aligned} 
& \mu_3 \Theta^3 + \mu_2 \Theta^2 + \mu_1 \Theta + \mu_0 \\ 
 & \qquad \quad = \lambda^{40} 
 \ord\big(\delta \abs{\Theta}^3 + \abs{\Theta}^4 + \delta^{-2} ( \norm{{ Z}}^3 + \abs{\scalar{E_-}{\cB^{-1} \cQ[{ Z}]}}^{3/2}) + \abs{\Theta}^2( \abs{\scalar{E_-}{B}}^2 + \abs{\scalar{\wh{B}}{\cA[B,B_*]}}^2) \big)
\end{aligned} 
\end{equation}
whose coefficients are given by 
\begin{equation} \label{eq:coefficients_abstract_cubic} 
\begin{aligned} 
\mu_3 & =  2\scalar{\wh{B}}{\cA[B,\cB^{-1} \cQ \cA[B,B]]}- \frac{2 \scalar{\wh{B}}{\cA[B,B_*]} \scalar{E_-}{\cB^{-1} \cQ \cA[B,B]}}{\scalar{E_-}{B_*}}, 
\\ \mu_2 & = \scalar{\wh{B}}{\cA[B,B]}, 
\\ \mu_1 & = -\beta \scalar{\wh{B}}{B} - 2 \scalar{\wh{B}}{\cA[B, \cB^{-1} \cQ [{ Z}]]} + \frac{2 \scalar{E_-}{B}}{\scalar{E_-}{B_*}} \scalar{\wh{B}}{\cA[\cB^{-1} \cQ[{ Z}],B_*]}, 
\\ \mu_0 & = \scalar{\wh{B}}{\cA[\cB^{-1}\cQ[{ Z}], \cB^{-1} \cQ[{ Z}]] - { Z}}.
\end{aligned} 
\end{equation} 
Moreover, $Y$ can be expressed by $\Theta$ and ${ Z}$ via 
\begin{equation} \label{eq:expansion_Y_abstract_cubic} 
\begin{aligned} 
 & Y  =  \Theta B - \cB^{-1} \cQ[{ Z}] + \Theta^2 \cB^{-1} \cQ \cA[B,B] - \Theta^2 \frac{\scalar{E_-}{\cB^{-1} \cQ \cA[B,B]}}{\scalar{E_-}{B_*}} B_* \\ 
 & \qquad \qquad \qquad \qquad \qquad \qquad \qquad + \lambda^{30} \ord\big( \abs{\Theta}^3 + \abs{\Theta}(\norm{{ Z}} + \abs{\scalar{E_-}{B}}) + \norm{{ Z}}^2 +  \abs{\scalar{E_-}{\cB^{-1} \cQ[{ Z}]}} \big).  
\end{aligned} 
\end{equation}
\end{lemma}

\begin{proof} 
We decompose $Y$ according to the spectral subspaces of $\cB$. This yields 
\begin{equation} \label{eq:proof_abstract_cubic_Y_decom_spectral}
 Y = \Theta B + \Theta_* B_* + \cQ[Y], \qquad \qquad  
\Theta\defeq \frac{\scalar{\wh{B}}{Y}}{\scalar{\wh{B}}{B}}, \quad \Theta_*\defeq \frac{\scalar{\wh{B}_*}{Y}}{\scalar{\wh{B}_*}{B_*}}.  
\end{equation}
We define 
\begin{equation} \label{eq:proof_abstract_cubic_decomposition} 
\begin{aligned} 
 Y_1 & \defeq \Theta B - \cB^{-1} \cQ[{ Z}], & Y_3 \defeq \frac{\scalar{E_-}{ \cB^{-1} \cQ[{ Z}] -\Theta {B} - {\cB^{-1} \cQ \cA[Y_1, Y_1]} }}{\scalar{E_-}{B_*}} B_*,\\ 
 Y_2 & \defeq \cQ[Y] + \cB^{-1} \cQ[{ Z}] + \Theta_* B_* - Y_3. \quad & 
\end{aligned} 
\end{equation}  
Obviously, $Y = Y_1 + Y_2 + Y_3$ as well as $Y_1 = \lambda\ord_1$ and $Y_3 = \lambda^7\ord_2$, where we introduced the notation
\[ \ord_k = \ord\Big( \abs{\Theta}^k + \norm{{ Z}}^k + \abs{\scalar{E_-}{B} \Theta}^{k/2} + \abs{\scalar{E_-}{\cB^{-1} \cQ[{ Z}]}}^{k/2}\Big) \] 
and used the convention that $ R  =\ord_k$ means $\norm{ R } = \ord_k$. Here and in the following, the implicit constant in $\ord$ will always be independent of $\lambda$. 

From $\scalar{E_-}{Y} = 0$ and \eqref{eq:proof_abstract_cubic_Y_decom_spectral}, we obtain 
\[ \Theta_* \scalar{E_-}{B_*} = - \Theta \scalar{E_-}{B} - \scalar{E_-}{\cQ[Y]} = \scalar{E_-}{\cB^{-1} \cQ [{ Z}] - \Theta B- \cB^{-1} \cQ \cA[Y, Y]}.  \] 
Here, we used that $\cQ[Y] = \cB^{-1} \cA[Y,Y] - \cB^{-1} \cQ[{ Z}]$ by \eqref{eq:quadratic_equation} in the second step. 
This shows that $\Theta_*$ is the coefficient of $B_*$ in the definition of $Y_3$ up to replacing $Y$ by $Y_1$. 
Thus, we deduce that 
\begin{equation} \label{eq:proof_abstract_cubic_8} 
 \norm{\Theta_* B_* - Y_3} = \lambda^5\ord\big(\norm{Y_1}( \norm{Y_2} + \norm{Y_3}) + \norm{Y_2}^2 + \norm{Y_3}^2\big). 
\end{equation}

We insert $Y =Y_1 + Y_2 + Y_3$ into \eqref{eq:quadratic_equation} and obtain 
\begin{equation} \label{eq:proof_abstract_cubic_1} 
 \Theta \beta B + \Theta_* \beta_* B_* + \cB \cQ[Y_2] + (1 - \cQ)[{ Z}] = \cA[Y,Y]. 
\end{equation}
Applying $\cB^{-1}\cQ$ to the previous relation implies 
\begin{equation} \label{eq:proof_abstract_cubic_3} 
 \cQ[Y_2] = \cB^{-1} \cQ \cA[Y,Y]. 
\end{equation}
Hence, $\norm{\cQ[Y_2]} = \lambda^2\ord(\norm{Y_1}^2 + \norm{Y_2}^2 + \norm{Y_3}^2)$. As $Y_2 = \cQ[Y_2] + \Theta_* B_* - Y_3$, we get from \eqref{eq:proof_abstract_cubic_8} that 
\begin{equation} \label{eq:proof_abstract_cubic_9} 
 \norm{Y_2} = \lambda^5\ord(\norm{Y_1}^2 + \norm{Y_2}^2 + \norm{Y_3}^2). 
\end{equation}
The definition of $Y_1$, $Y_2$ and $Y_3$ as well as the conditions \eqref{eq:conditions_abstract_cubic} and $\norm{{ Z}} + \norm{Y}\leq c \lambda^{-12}$ yield $\norm{Y_2} \leq cC \lambda^{-5}$ for some universal constant $C>0$. 
Hence, \eqref{eq:proof_abstract_cubic_9} implies $Y_2 = \lambda^{19}\ord_2$ as $Y_1 = \lambda \ord_1$ and $Y_3 = \lambda^7 \ord_2$ if $c$ is chosen sufficiently small independently of $\lambda$. Therefore, we conclude from \eqref{eq:proof_abstract_cubic_3}, \eqref{eq:proof_abstract_cubic_8} and $\abs{\Theta} + \norm{{ Z}} = \ord(\lambda^{-10})$ that 
\begin{equation} \label{eq:proof_abstract_cubic_4} 
 Y_2 = \cB^{-1} \cQ \cA[Y_1, Y_1] + \lambda^{30}\ord_3.
\end{equation}
In particular, this implies \eqref{eq:expansion_Y_abstract_cubic} as $Y = Y_1 + Y_2 + Y_3$. 

Applying $\scalar{\wh{B}}{\genarg}$ to \eqref{eq:proof_abstract_cubic_1} yields 
\begin{equation} \label{eq:proof_abstract_cubic_6} 
 \Theta \beta \scalar{\wh{B}}{B} + \scalar{\wh{B}}{{ Z}} = \scalar{\wh{B}}{\cA[Y,Y]}. 
\end{equation}
Therefore, we now show \eqref{eq:abstract_cubic} by computing $\scalar{\wh{B}}{\cA[Y,Y]}$. Using $Y_1 = \lambda\ord_1$, $Y_2 = \lambda^{19} \ord_2$, $Y_3 = \lambda^{7}\ord_2$ and \eqref{eq:proof_abstract_cubic_4}, we deduce 
\begin{equation} \label{eq:proof_abstract_cubic_2} 
\scalar{\wh{B}}{\cA[Y,Y]}  = \scalar{\wh{B}}{\cA[Y_1, Y_1]} + 2\scalar{\wh{B}}{\cA[Y_1, Y_3]} + 2 \scalar{\wh{B}}{\cA[Y_1,  \cB^{-1} \cQ \cA[Y_1, Y_1]]} + \lambda^{40}\ord_4. 
\end{equation}
For a linear operator $\cal{K}_1$ and a bilinear operator $\cal{K}_2$ with $\norm{\cal{K}_1} + \norm{\cal K_2} \leq 1$, we have 
\[ \norm{\Theta \cal{K}_2[R,R]} \leq \delta \abs{\Theta}^3 + \delta^{-1/2} \norm{R}^3, \qquad \norm{\Theta^2 \cal{K}_1[R]} \leq \delta \abs{\Theta}^3 + \delta^{-2} \norm{R}^3 \] 
for any matrix $R \in \C^{2n \times 2n}$ since $\delta >0$. Therefore, as $\delta \in (0,1)$, we obtain 
\begin{equation} \label{eq:proof_abstract_cubic_5} 
 \scalar{\wh{B}}{\cA[Y_1,\cB^{-1} \cQ\cA[Y_1,Y_1]]} = \Theta^3 \scalar{\wh{B}}{\cA[B, \cB^{-1} \cQ\cA[B,B]]} + \lambda^7\ord(\delta\abs{\Theta}^3 + \delta^{-2} \norm{{ Z}}^3). 
\end{equation}
Similarly, we conclude 
\begin{equation} \label{eq:proof_abstract_cubic_7} 
\hspace*{-0.33cm}
\begin{aligned} 
\scalar{\wh{B}}{\cA[Y_1, Y_3]} = \, &  \frac{\Theta \scalar{\wh{B}}{\cA[\cB^{-1} \cQ[{ Z}], B_*]} \scalar{E_-}{B} - \Theta^3 \scalar{\wh{B}}{\cA[B,B_*]}\scalar{E_-}{\cB^{-1} \cQ\cA[B,B]} }{\scalar{E_-}{B_*}} \\ 
& + \lambda^{10}\ord\big(\delta \abs{\Theta}^3 + \delta^{-2} \norm{{ Z}}^3 + \delta^{-2} \abs{\scalar{E_-}{\cB^{-1} \cQ[{ Z}]}}^{3/2} + \abs{\Theta}^2( \abs{\scalar{E_-}{B}}^2 + \abs{\scalar{\wh{B}}{\cA[B,B_*]}}^2) \big).
\end{aligned} 
\end{equation} 

Finally, we expand the first term on the right-hand side of \eqref{eq:proof_abstract_cubic_2} using the definition of $Y_1$ from \eqref{eq:proof_abstract_cubic_decomposition} and 
insert \eqref{eq:proof_abstract_cubic_5} as well as \eqref{eq:proof_abstract_cubic_7} into \eqref{eq:proof_abstract_cubic_2} to compute the second and third term. 
We apply the result to \eqref{eq:proof_abstract_cubic_6} and obtain the cubic equation in \eqref{eq:abstract_cubic} with the coefficients detailed in \eqref{eq:coefficients_abstract_cubic}. 
\end{proof}

\section{Non-Hermitian perturbation theory}  \label{app:perturbation_th} 

In this section, we present for the reader's convenience the perturbation theory for a non-Hermitian operator $\cK$ on $\C^{2n\times 2n}$ with an isolated eigenvalue $\kappa$.  
We denote by $\PK$ the spectral projection of $\cK$ associated to $\kappa$ and set $\QK \defeq 1- \PK$.
We assume that the algebraic multiplicity of $\kappa$ coincides with its geometric multiplicity. 
In particular, this condition ensures that, for any $L \in \C^{2n\times 2n}$, we have 
\[ \cK [K] = \kappa K, \qquad \cK^*[\wh{K}] = \bar \kappa \wh{K},  \] 
where $K \defeq \PK [L]$ and $\wh{K} \defeq \PK^*[L]$. 
That is, $K$ and $\wh{K}$ are right and left eigenvectors of $\cK$ corresponding to $\kappa$, respectively. 

Throughout this section, we suppose that there is a constant $C>0$ such that 
\begin{equation} \label{eq:assumptions_cK} 
\norm{\cK} + \norm{(\cK-\kappa)^{-1} \QK} +\norm{\PK} \leq C. 
\end{equation}
Here and in the following, $\norm{\genarg}$ denotes the operator norm of operators on $\C^{2n\times 2n}$ induced by some norm $\norm{\genarg}$ on the matrices in $\C^{2n\times 2n}$.

\begin{lemma} \label{lem:perturbation_theory_2} 
There is $\eps >0$, depending only on $C$ from \eqref{eq:assumptions_cK},  such that the following holds. 

If $\cL$ is a linear map on $\C^{2n\times 2n}$ satisfying $\norm{\cK - \cL} \leq \eps$ and $\lambda$ is an eigenvalue of $\cL$ 
satisfying $\abs{\kappa - \lambda} \leq \eps$ then, for any right and left normalized eigenvectors $L$ and $\wh{L}$ of $\cL$ associated to $\lambda$, we have 
\begin{subequations} 
 \begin{align} 
\lambda \scalar{\wh{L}}{L} & = \kappa\scalar{\wh{K}}{K} + \scalar{\wh{K}}{\cD[K]} + \scalar{\wh{K}}{\cD \QK (2 \kappa - \cK) (\cK- \kappa)^{-2} \QK \cD [K]}  + \ord(\norm{\cD}^3),\label{eq:pert_th_eigenvalue}\\ 
L & =  K -(\cK- \kappa)^{-1}\QK \cD [K] + L_2 + \ord(\norm{\cD}^3) ,\label{eq:pert_th_eigenvector_right} \\ 
\wh{L} &= \wh{K} -(\cK^*- \bar \kappa)^{-1}\QK^* \cD^* [\wh{K}] + \wh{L}_2 + \ord(\norm{\cD}^3)\label{eq:pert_th_eigenvector_left} , 
\end{align} 
\end{subequations} 
where we used the definitions $\cD \defeq \cL - \cK$, $K \defeq \PK[L]$ and $\wh{K} \defeq \PK^*[\wh{L}]$ as well as  
\begin{align*}
L_2 & \defeq (\cK-\kappa)^{-1} \QK \cD (\cK-\kappa)^{-1} \QK \cD [K] - (\cK- \kappa)^{-2} \QK \cD \PK\cD [K], \\ 
\wh{L}_2 & \defeq (\cK^*-\bar \kappa)^{-1} \QK^* \cD^* (\cK^*-\bar \kappa)^{-1} \QK^* \cD^* [\wh{K}] - (\cK^*- \bar\kappa)^{-2} \QK^* \cD^*\PK^*\cD^* [\wh{K}]. 
\end{align*} 
\end{lemma} 

In the previous lemma and in the following, the implicit constants in the comparison relation $\lesssim$ and in 
$\ord$ depend only on $C$ from \eqref{eq:assumptions_cK}.  

\begin{proof} 
We first establish the relations \eqref{eq:pert_th_eigenvector_right} and \eqref{eq:pert_th_eigenvector_left} for the eigenvectors. 
The eigenvector relation $\cL[L] = \lambda L$ together with the definition $\delta \defeq \lambda - \kappa$ yield 
\begin{equation} \label{eq:delta_a} 
 \cK \QK [L] + \cD [L] = \delta K + \beta \QK [L ]. 
\end{equation}
Here, we also employed $\cK[K]= \kappa K$ and $ L = K + \QK [L]$. 

By applying $(\cK-\kappa)^{-1} \QK$ in \eqref{eq:delta_a}, we get 
\begin{equation} \label{eq:Q_b}
 \QK[L] = - (\cK-\kappa)^{-1} \QK \cD [L ] + \delta (\cK-\kappa)^{-1} \QK [L]. 
\end{equation}
This relation immediately implies 
\[ \norm{\QK [L]} \lesssim \norm{\cD} + \abs{\delta}\norm{\QK[L]}.  \]  
Thus, we obtain $\norm{\QK[L]} \lesssim \norm{\cD}$ by choosing $\eps$ sufficiently small. 
Hence, \eqref{eq:delta_a} implies 
\[ \abs{\delta} \norm{K} \leq \abs{\beta}\norm{\QK[L]} + \norm{\cK \QK [L]}  + \norm{\cD [L]} = \ord(\norm{\cD}).  \] 
Since $\norm{K} \geq \norm{L} - \norm{\QK[L]} \geq 1/2$ for sufficiently small $\eps$, we conclude 
\[ \abs{\delta} \lesssim \norm{\cD}. \] 

\noindent We start from $ L = K  + \QK[L]$ and iteratively replace $\QK[L]$ by using \eqref{eq:Q_b} to obtain 
\begin{equation} \label{eq:proof_expansion_b}
\hspace*{-0.20cm} 
\begin{aligned} 
 L & = K - (\cK- \kappa)^{-1} \QK \cD [L] + \delta (\cK- \kappa)^{-1} \QK[L] \\ 
 & = K - (\cK- \kappa)^{-1} \QK \cD [L] - \delta (\cK- \kappa)^{-2} \QK \cD [L]  + \ord(\norm{\cD}^3)  \\ 
& =  K -(\cK- \kappa)^{-1}\QK \cD [K] + (\cK-\kappa)^{-1} \QK \cD (\cK-\kappa)^{-1} \QK \cD [K] - \delta (\cK- \kappa)^{-2} \QK \cD [K] + \ord(\norm{\cD}^3)\\ 
& =  K -(\cK- \kappa)^{-1}\QK \cD [K] + (\cK-\kappa)^{-1} \QK \cD (\cK-\kappa)^{-1} \QK \cD [K] - (\cK- \kappa)^{-2} \QK \cD \PK\cD [K] + \ord(\norm{\cD}^3) . 
\end{aligned} 
\end{equation} 
Here, we also used that $\abs{\delta} + \norm{\QK[L]} =\ord(\norm{\cD})$. 
The last step in \eqref{eq:proof_expansion_b} follows from 
\[ \delta K = \cD [L] + \cK \QK [L] - \beta \QK [L] = \cD [K] + (\cK - \kappa)\QK[L] + \ord(\norm{\cD}^2) = \cD [K] - \QK \cD [K] +\ord(\norm{\cD}^2), \] 
which is a consequence of \eqref{eq:delta_a} and \eqref{eq:Q_b}.
This completes the proof of \eqref{eq:pert_th_eigenvector_right}. A completely analogous argument yields \eqref{eq:pert_th_eigenvector_left}. 

For the proof of \eqref{eq:pert_th_eigenvalue}, we first define $L_1 \defeq -(\cK- \kappa)^{-1}\QK \cD [K]$ and $\wh{L}_1 \defeq -(\cK^*- \bar \kappa)^{-1}\QK^* \cD^* [\wh{K}]$. 
Using \eqref{eq:pert_th_eigenvector_right} and \eqref{eq:pert_th_eigenvector_left} in the relation $\beta \scalar{\wh{L}}{L} = \scalar{\wh{L}}{(\cK +\cD)[L]}$ yields 
\[ \begin{aligned} 
\beta \scalar{\wh{L}}{L}  =\, & \kappa \scalar{\wh{K}}{K} + \kappa \scalar{\wh{L}_1}{K} + \kappa \scalar{\wh{L}_2}{K} + \scalar{\wh{K}}{\cK[L_1]} + \scalar{\wh{K}}{\cK[L_2]} + \scalar{\wh{L}_1}{\cK[L_1]} \\ 
& \qquad \qquad \qquad + \scalar{\wh{K}}{\cD[K]} + \scalar{\wh{L}_1}{\cD[K]}  + \scalar{\wh{K}}{\cD[L_1]} + \ord(\norm{\cD}^3) 
\end{aligned} \] 
since $L_1, \wh{L}_1 = \ord(\norm{\cD})$ and $L_2, \wh{L}_2 = \ord(\norm{\cD}^2)$. 
We remark that $\scalar{\wh{L}_1}{K} = \scalar{\wh{K}}{\cK[L_1]} = \scalar{\wh{L}_2}{K} = \scalar{\wh{K}}{\cK[L_2]} = 0$ since $\QK[K] = 0$ and $\QK^*[\wh{K}] = 0$. 
For the remaining terms, we get 
\[ \scalar{\wh{L}_1}{\cK[L_1]} = \scalar{\wh{K}}{\cD \cK ( \cK -\kappa)^{-2} \QK \cD [K] }, \qquad \qquad 
\scalar{\wh{L}_1}{\cD [K]} = \scalar{\wh{K}}{\cD[L_1]} = - \scalar{\wh{K}}{\cD ( \cK -\kappa)^{-1} \QK \cD [K]}. \] 
Therefore, a simple computation yields \eqref{eq:pert_th_eigenvalue}. This completes the proof of Lemma~\ref{lem:perturbation_theory_2}. 
\end{proof}

\bibliography{literature} 
\bibliographystyle{amsplain} 

\end{document}